\documentclass[french,11pt]{article}   

\usepackage{lmodern}            
\usepackage[T1]{fontenc}         
\usepackage[utf8]{inputenc}      
\usepackage[english]{babel}
\usepackage{bbold}
\usepackage{xspace}
\usepackage{amsmath,amssymb}
\usepackage{amsthm}
\usepackage{newtxmath}
\usepackage{stmaryrd}
\usepackage{dsfont}
\usepackage[margin=3.2cm]{geometry}
\usepackage[unicode=true]{hyperref}
\usepackage{xcolor}
\usepackage{enumitem}
\hypersetup{
    colorlinks,
    linkcolor={black},
    citecolor={black},
    urlcolor={black}}
\usepackage{graphicx}
\usepackage{subcaption}
\graphicspath{ {./images/} }
\input{insbox}

\newtheorem{theorem}{Theorem}[section]
\newtheorem{corollary}[theorem]{Corollary}

\newtheorem{lemme}[theorem]{Lemma}
\newtheorem{proposition}[theorem]{Proposition}
\newtheorem{remark}[theorem]{Remark}
\theoremstyle{definition}

\newcommand{\E}{\mathbb{E}}		

\newcommand{\N}{\mathbb{N}}	

\renewcommand{\P}{\mathbb{P}}	
\newcommand{\R}{\mathbb{R}}		
\renewcommand{\S}{\mathbb{S}}	
\newcommand{\T}{\mathbb{T}}

\newcommand{\Z}{\mathbb{Z}}		

\newcommand{\cB}{\mathcal B}		
\newcommand{\cC}{\mathcal C}	

\newcommand{\cM}{\mathcal M}	
\newcommand{\cR}{\mathcal R}	
\newcommand{\cS}{\mathcal S}
\newcommand{\cT}{\mathcal T}		

\newcommand{\cZ}{\mathcal Z}

\newcommand{\1}{\mathds{1}}			

\definecolor{Myblue}{rgb}{0.2,0.2,0.7}
\newtheoremstyle{correction}{0.75em}{0.75em}{\color{Myblue}}{}{\bfseries}{.}{ }{\thmname{#1}\thmnumber{ #2}\thmnote{ (#3)}}
\theoremstyle{correction}

\usepackage{biblatex} 
\title{Local scaling limits of quadrangulations rooted on geodesics}
\author{Mathieu Mourichoux}
\date{}

\begin{document}
\maketitle

\begin{abstract}
We identify the local scaling limit of the Uniform Infinite Planar Quadrangulation (UIPQ) and of critical Boltzmann quadrangulations, when one simultaneously rescales the distances and reroot them far away from the root of a distinguished geodesic. The limiting space is the bigeodesic Brownian plane, which appears as the local limit of the Brownian sphere around an interior point of a geodesic. We also show that the $\overline{\mathrm{UIPQ}}$ introduced by Dieuleveut, which is the local limit of the $\mathrm{UIPQ}$ rerooted at a far away point of its infinite geodesic, has the same scaling limit. These results can be seen as a commutation property between local limit, scaling limit and moving forward on a geodesic in random quadrangulations. The proofs are based on spinal decompositions of random trees and coupling results.

\end{abstract}

\tableofcontents

\section{Introduction}

In the last decades, a lot of attention has been dedicated to random planar maps, which are a model of discrete geometry. Of particular interest is the study of local limits, scaling limits, and geodesics in random planar maps. The purpose of this paper is to connect these different topics, by establishing the scaling limit of some models of random planar maps rooted on a geodesic.

Consider the embedding of a finite connected graph into the sphere $\mathbb{S}^2$ (loops and multiple edges are allowed). A planar map is an equivalence class of such embeddings, modulo orientation-preserving homeomorphisms of the sphere. The faces of $m$ are the connected components of the complement of $E(m)$. The degree of a face is the number of oriented edges that are incident to it.  We say that a planar map $m$ is a quadrangulation if all its faces have degree $4$. All the planar maps considered in this paper are rooted, meaning that they have a distinguished edge. A positive quadrangulation is a rooted quadrangulation with a distinguished vertex such that the distinguished point is closer from the tip of the root edge than from its origin. 

The first random maps that we will be dealing with are critical positive Boltzmann quadrangulations $Q$, which are random positive quadrangulations such that for every positive quadrangulation $\mathbf{q}$ with $n$ faces,
\[\P(Q=\mathbf{q})=\frac{1}{2\times12^n}.\]

This random map is called critical as its expected number of faces is infinite, see \cite{peeling} for more details.
This random map has a distinguished geodesic between its root vertex and its distinguished vertex (called the left-most geodesic from the root), that we denote by $\gamma$.

On the one hand, one can consider the local limit of such maps. This operation consists in looking at a fixed ball of radius $r>0$ around the root edge, and proving the convergence in distribution of these balls for every $r$, when the size of the map goes to infinity. When we condition $Q$ to be large, the limiting space is called the $\mathrm{UIPQ}$ \cite{Krikun}, that we denote by $Q_\infty$. This random map has been studied extensively since, see \cite{UIHPQ,GeometryUIHPQ,SeparatingCycles} for instance. In particular, it was proved in \cite{Viewfrominfinity} that, in a sense, $Q_\infty$ has an essentially unique infinite geodesic $\gamma_\infty$, meaning that there exists a sequence of points of $\gamma_\infty$ such that every infinite geodesic visits all except finitely many of these points. Moreover, Dieuleveut showed in \cite{dieuleveut} that $Q_\infty$ admits a local limit when we reroot it further and further on its infinite geodesic $\gamma_\infty$. The resulting space is called the $\overline{\mathrm{UIPQ}}$, and we denote it by $\overline{Q}_\infty$. Moreover, this random space has an infinite bigeodesic $\overline{\gamma}_\infty$.

On the other hand, one may consider the scaling limit of random maps. This consists in rescaling the distances by a certain factor, and studying the convergence in distribution of the whole map. It was proved in \cite{uniqueness,convergence} that, if we condition $Q$ to have $n$ faces and rescale distances by $n^{-1/4}$, the scaling limit exists, which is called the Brownian sphere. Moreover, this random metric space, denoted by $\cS$, has a distinguished geodesic $\Gamma$. Then, it was shown in \cite{Bigeodesicbrownianplane} that the Brownian sphere admits a local limit when one looks closer and closer at an interior point of $\Gamma$. The limiting space is called the bigeodesic Brownian plane, and we denote it by $\overline{\mathcal{BP}}$. We also mention that similar topics were studied for LQG-surfaces in \cite{basu2021environment}.

Our goal is to connect all these results together. Let $Q_n$ stand for the random variable $Q$ conditioned on the event $\{\mathrm{length}(\gamma)\geq n\}$, and denote the distinguished geodesic of $Q_n$ by $\gamma_n$. Our main result shows that, seen from their distinguished geodesics, the random maps $Q_n,Q_\infty,\overline{Q}_\infty$ have a common scaling limit, which is $\overline{\mathcal{BP}}$.

\begin{theorem}\label{Main result}
    We have
    \[\left(V(\overline{Q}_\infty),\lambda\cdot d_{\overline{Q}_\infty},\overline{\gamma}_\infty(0)\right)\xrightarrow[\lambda\rightarrow0]{(d)}\left(\overline{\mathcal{BP}},\overline{D}_\infty,\overline{\rho}_\infty\right)\]
 in distribution for the local Gromov-Hausdorff topology. Moreover, for any sequence $(k_n)_{n\in\N}$ of non-negative real numbers such that $k_n\rightarrow\infty$ and $k_n=o(n)$, we have 
    \[\left(V(Q_n),k_n^{-1}\cdot d_{Q_n},\gamma_n(n)\right)\xrightarrow[\lambda\rightarrow0]{(d)}\left(\overline{\mathcal{BP}},\overline{D}_\infty,\overline{\rho}_\infty\right)\]
    and
    \[\left(V(Q_\infty),k_n^{-1}\cdot d_{Q_\infty},\gamma_\infty(n)\right)\xrightarrow[\lambda\rightarrow0]{(d)}\left(\overline{\mathcal{BP}},\overline{D}_\infty,\overline{\rho}_\infty\right)\]
    in distribution for the local Gromov-Hausdorff topology.     
\end{theorem}

Informally, together with some results of \cite{brownianplane,Bigeodesicbrownianplane} about links between $Q,Q_\infty,\overline{\mathcal{BP}}$ and the Brownian plane, this means that the operations of local limits, scaling limits and ``moving forward on a geodesic'' commute.

\begin{figure}
    \centering
    \includegraphics[width=0.5\linewidth]{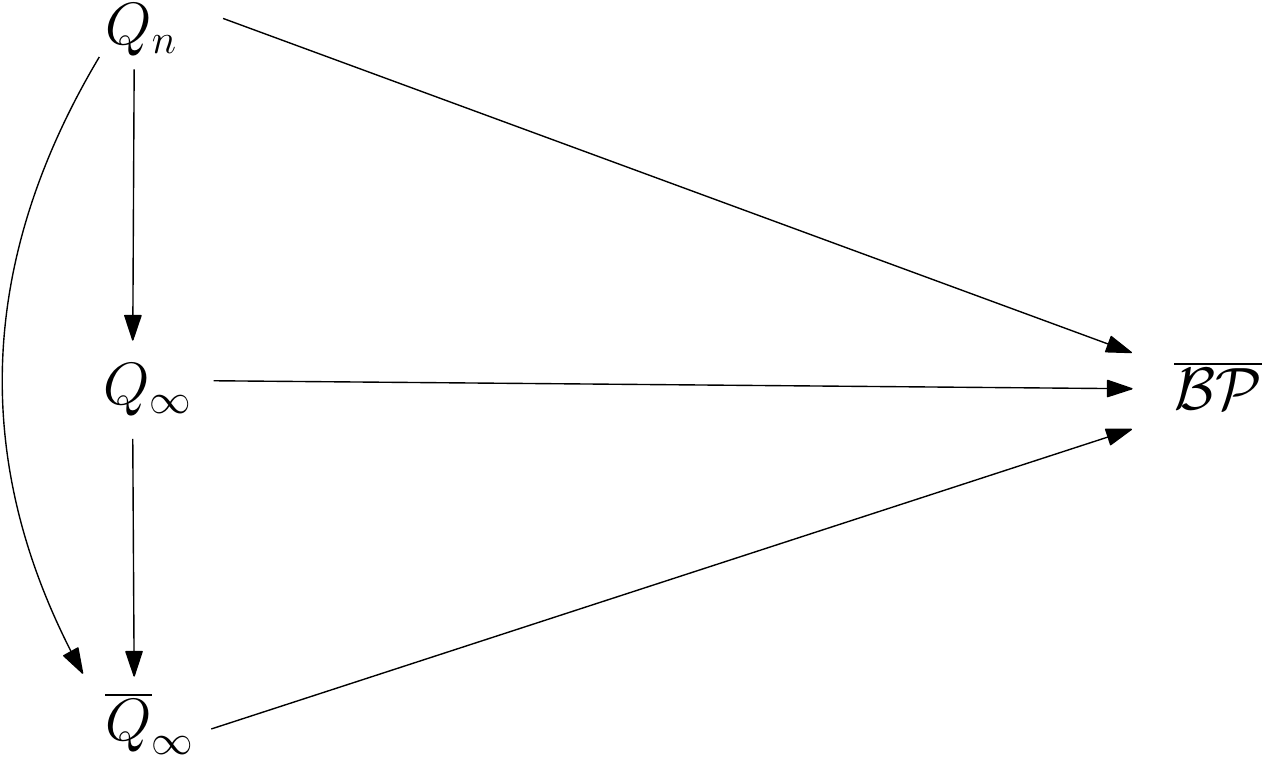}
    \caption{The diagram of convergences between the different surfaces. The arrows pointing down represent local limits, and those pointing right represent scaling limits.}
    \label{fig:enter-label}
\end{figure}

We also obtain a strong coupling result for $Q_n,Q_\infty,\overline{Q}_\infty$ around geodesics. 

\begin{theorem}\label{theorem 2}
    For every $\varepsilon>0$, there exists $\alpha>0$ such that for every $n\in\N$ large enough, the following is true. We can couple $\left(Q_n,Q_\infty,\overline{Q}_\infty\right)$ such that we have 
    \[B_{Q_n}(\gamma_n(n),\alpha n)=B_{Q_\infty}(\gamma_\infty(n),\alpha n)=B_{\overline{Q}_\infty}(\overline{\gamma}_\infty(0),\alpha n)\] 
    with probability at least $1-\varepsilon$. 
\end{theorem}

In particular, we obtain a direct link between $Q_n$ and $\overline{Q}_\infty$.
\begin{corollary}\label{coro convergence}
    We have the convergence in distribution 
    \[\left(Q_n,d_{Q_n},\gamma_n(n)\right)\xrightarrow[n\rightarrow\infty]{(d)}\left(\overline{Q}_\infty,d_{\overline{Q}_\infty},\overline{\gamma}_\infty(0)\right)\]
    for the local topology. 
\end{corollary}

This paper is organized as follows. 
Section \ref{sec preliminaries} contains a number of preliminaries about random trees, random maps and the CVS bijection. In Section \ref{Coupling trees}, we prove some coupling results about several random trees, that rely on spine decompositions. Section \ref{Localisation} is devoted to the proof of a key lemma, which allows us to transfer the coupling result of random trees to coupling results of random maps. In Section \ref{sec main result}, we prove Theorem \ref{Main result}. Finally, Appendix \ref{appendix} contains the proof of some lemmas.

\subsection*{Acknowledgements}

I thank Grégory Miermont for his constant support and for his comments on earlier versions of this paper. I also thank Lou Le Bihan and Simon Renouf for their help in typing this article.

\section{Preliminaries}\label{sec preliminaries}

\subsection{Planar maps}

As mentioned in the introduction, a finite planar map is an equivalence class of embeddings of a finite connected graph into the sphere $\mathbb{S}^2$ (loops and multiple edges are allowed), modulo orientation-preserving homeomorphisms of the sphere. For a map $m$, let $V(m),E(m)$ and $\overrightarrow{E}(m)$ denote the set of vertices, edges and oriented edges of $m$. A quadrangulation with a boundary is a planar map with a distinguished face (which has a simple boundary), and such that all the other faces have degree $4$. In this paper, we will only consider rooted maps, which are maps with a distinguished oriented edge. This edge is called the root edge, and its origin is called the root vertex. We denote by $\mathcal{Q}_f$ (respectively $\mathcal{Q}_{f,b}$) the set of finite rooted quadrangulations (respectively finite rooted quadrangulations with a boundary).

Let us define the local topology for finite maps (we refer to \cite[Section 2.1.2]{peeling} for more details). Every map $m$ can be equipped with the associated graph distance $d_m$. Let $B_m(r)$ stand for the ball of radius $r$ around the root vertex in $m$. For every finite planar map $m$ and $m'$, we can define 
\[D(m,m')=\left(1+\sup\{r\geq0:B_r(m)=B_r(m')\}\right)^{-1}.\]
The local topology is the topology induced by the distance $D$. We denote by $\mathcal{Q}$ and $\mathcal{Q}_b$ the completion of $\mathcal{Q}_f$ and $\mathcal{Q}_{f,b}$ for this topology, and we set $\mathcal{Q}_\infty=\mathcal{Q}\backslash\mathcal{Q}_f$ and $\mathcal{Q}_{\infty,b}=\mathcal{Q}_b\backslash\mathcal{Q}_{f,b}$. In this article, we will only work with elements of $\mathcal{Q}_\infty$ that are quadrangulations of the plane, and elements of $\mathcal{Q}_{\infty,b}$ that are quadrangulations of the half-plane with an infinite boundary.

\subsection{Labeled trees}

\paragraph{Formalism about trees.}

In this work, we will use the standard formalism for planar trees that can be found in \cite{Neveu1986}, and that we recall here. Define
\[\mathcal{U}=\bigcup_{n=0}^\infty \N^n\]
where $\N=\{1,2,3...\}$ and $\N^0=\{\varnothing\}$. Hence, an element of $\mathcal{U}$ is a finite sequence of positive integers. If $u,v\in\mathcal{U}$, $uv$ stands for the concatenation of $u$ and $v$. If $v$ is of the form $uj$ for some $j\in\N$, we say that $u$ is the parent of $v$, and that $v$ is the child of $u$. Similarly, if $v$ is of the form $uw$ for some $u,w\in\mathcal{U}$, we say that $u$ is an ancestor of $v$ and that $v$ is a descendant of $u$. A rooted planar tree $\mathbf{t}$ is a (finite or infinite) subset of $\mathcal{U}$ such that 
\begin{enumerate}
    \item $\varnothing\in\mathbf{t}$ ($\varnothing$ is called the root of $\mathbf{t}$)
    \item if $v\in\mathbf{t}$ and $v\neq\varnothing$, the parent of $v$ belongs to $\mathbf{t}$
    \item for every $u\in\mathcal{U}$, there exists $k_u(\mathbf{t})\geq0$ such that $uj\in\mathbf{t}$ if and only if $j\leq k_u(\mathbf{t})$.
\end{enumerate}
A rooted planar tree can be seen as a graph, where two vertices $u$ and $v$ are linked by an edge if and only if $u$ is the parent of $v$, or vice versa. This graph has a natural embedding in the plane, in which the edges from a vertex $u$ to its children $u1,u2,...,uk_u(\mathbf{t})$ are drawn from left to right. A spine of $\mathbf{t}$ is an infinite sequence $(u_i)_{i\geq0}$ such that $u_0=\varnothing$ and $u_i$ is the parent of $u_{i+1}$ for every $i\geq0$. If $u$ and $v$ are two vertices of $\mathbf{t}$, $\llbracket u,v \rrbracket$ denotes the unique simple path from $u$ to $v$. For every tree $\theta$, we denote by $|\theta|$ the number of edges of $\theta$. A forest is a sequence of rooted planar trees, and its size is the number of trees it contains.\\
We also define the notion of corner in a tree. A corner of a vertex $v\in\mathbf{t}$ is an angular sector formed by two consecutive edges (taken in clockwise order) around $v$. If $\mathbf{t}$ has $n$ edges, we define the corner sequence of $\mathbf{t}$ as the sequence of corners $(c_0,...,c_{2n-1})$ taken in clockwise order, starting from the corner incident to the root and located to the left of the oriented edge from $\varnothing$ to $1\in\mathbf{t}$ (we call this corner the root corner). The notion of corner sequence also makes sense for infinite trees with a unique spine. Let $(c_0^{(L)},c_1^{(L)},c_2^{(L)},...)$ be the sequence of corners visited by the contour process on the left side of the tree (that is, when we explore the tree clockwise, starting from the root). Similarly, let $(c_0^{(R)},c_1^{(R)},c_2^{(R)},...)$ be the sequence of corners visited when we explore the tree counter-clockwise. Then, we can concatenate these sequences to define $(c_i)_{i\in\mathbb{Z}}$:
\[c_i=\left\{
    \begin{array}{ll}
    c^{(L)}_i\quad&\text{ if $i\geq0$} \\
     c^{(R)}_{-i}\quad&\text{ if $i<0$}
    \end{array}\right. \]
In what follows, we write $c_i<c_j$ if $i<j$. This allows us to define intervals on $\mathbf{t}$, by setting 
\[[c_i,c_j]=\left\{
    \begin{array}{ll}
    \{c_i,c_{i+1},...,c_j\}\quad&\text{ if $i\leq j$} \\
    \{c_K:k\notin(i,j)\}\quad&\text{ if $j>i$}
    \end{array}\right.\]
    
We also extend this definition to vertices. Given two vertices $v$ and $v'$, let $C_v$ and $C_{v'}$ be the sets of corner respectively associated to $v$ and $v'$. We define $[v,v']$ as the set of vertices with at least one corner in
\[\bigcap_{c_v\in C_v,c_{v'}\in C_{v'}}[c_v,c_{v'}].\]

\paragraph{Trees with labels.}

A rooted labeled tree is a pair $\theta=(\mathbf{t},(\ell_v)_{v\in\mathbf{t}})$, where $\mathbf{t}$ is a rooted planar tree and $\ell:\mathbf{t}\mapsto\mathbb{Z}$ is a labeling function such that if $u,v\in\mathbf{t}$ and $v$ is a child of $u$, then $|\ell(u)-\ell(v)|\leq1$. If  $\theta=(\mathbf{t},\ell)$ is a rooted labelled tree and $k\in\mathbb{Z}$, we use the notation $\theta+k=(\mathbf{t},\ell+k)$. We denote by $\mathbb{T}$ the set of labeled trees, and for every $x\in\mathbb{Z}$, $\mathbb{T}^{(x)}$ is the subset of $\mathbb{T}$ consisting of the labeled trees such that $\ell(\varnothing)=x$. Finally, let $\mathbb{T}^{(x)}_\infty$ (respectively $\mathbb{T}^{(x)}_f$) be the subset of $\mathbb{T}^{(x)}$ consisting of infinite (respectively finite) trees. For $\theta=(\mathbf{t},\ell)$, we let $|\theta|=|\mathbf{t}|$ be the size of $\theta$. Let us mention that we can define a distance on $\mathbb{T}$, which is analogous to the local metric on planar maps, and which turns this set into a Polish space. \newline
In this work, we will mostly be interested in infinite labeled  trees with a unique spine, denoted by $(S_{\mathbf{t}}(i))_{i\geq0}$ (we call this set $\mathbb{S}$). If $\theta=(\mathbf{t},\ell)\in\mathbb{S}$, the spine of $\mathbf{t}$ splits the tree in two parts, and for every element $S_{\mathbf{t}}(i)$  of the spine, we can define a subtree of $\mathbf{t}$ attached to its left and to its right. We call these subtrees $L_n(\theta)$ and $R_n(\theta)$, and formally define them as follows. Consider the lexicographic order $<$ and the partial genealogical order $\prec$. Then
\begin{align*}
    L_n(\theta)&=\{u\in \mathbf{t},\, S_{\mathbf{t}}(n)\leq u<S_{\mathbf{t}}(n+1)\},\\
    R_n(\theta)&=\{u\in \mathbf{t},\, S_{\mathbf{t}}(n)\preceq u, S_{\mathbf{t}}(n+1)\nprec u \text{ and } S_{\mathbf{t}}(n+1)<u\}.
\end{align*}
These subtrees inherit the labels from $\theta$, and can be seen as elements of $\mathbb{T}^{(\ell(S_{\mathbf{t}}(n))}$. Moreover, note that the tree $\theta$ is completely determined by the two sequences $\left(L_n(\theta),R_n(\theta)\right)_{n\geq 0}$. When there is no ambiguity, we will often just write $S(n),L_n,R_n$. 

\subsection{The CVS bijection}\label{sec cvs}

\paragraph{For finite labeled trees.}

The Cori-Vauquelin-Schaeffer bijection, which associates well-labelled trees and quadrangulations, goes as follows. Let $\theta=(\mathbf{t},\ell)$ be an element of $\mathbb{T}_f$, where $\mathbf{t}$ is canonically embedded in the plane. If $\mathbf{t}$ has $n$ edges, we consider its corner sequence $(c_0,...,c_{2n-1})$, and we extend it by $2n$-periodicity. For every $0\leq i\leq2n-1$, we define $\sigma_\theta(c_i)$ the successor of $c_i$ as $c_{\phi(i)}$, where
\begin{equation}\label{Successor}
\phi(i)=\inf\{j>i,\,\ell(c_j)=\ell(c_i)-1\}
\end{equation} 
(by convention, we set $\inf\{\varnothing\}=+\infty$). If $\phi(i)=+\infty$, then $\sigma_\theta(c_i)$ is an extra element $\partial$. \\
We can construct a new graph as follows. First, add an extra vertex $v_*$ in the plane, which does not belong to the embedding of $\mathbf{t}$. Then, for every corner $c$, draw an arc between $c$ and its successor (if $\sigma_\theta(c)=\partial$, draw an edge between $c$ and $v_*$). This can be done in such a way that the arcs do not cross each other, and do not cross the edges of $\mathbf{t}$. Then, if we remove the interior of the edges of $\mathbf{t}$, we are left with an embedded graph, with vertex set $\mathbf{t}\cup \{v_*\}$, and edges given by the arcs. Finally,  we root this graph on the arc from $c_0$ to its successor. We denote the resulting graph by $\Phi(\theta)$.
\begin{proposition}[\cite{cori_vauquelin_1981,schaeffer}]
    The mapping $\Phi$ is a bijection between $\mathbb{T}_f^{(0)}$ and finite planar positive quadrangulations. 
\end{proposition}
Let us state some properties of this bijection. For every $\theta\in\mathbb{T}_f^{(0)}$, set $\ell_*=\inf\{\ell_v,v\in\theta\}-1$. For every $u\in\theta$, and viewing $u$ as a vertex of $\Phi(\theta)$, we have
\begin{equation}\label{Distance}
    d_{\Phi(\theta)}(u,v_*)=\ell_u-\ell_*.
\end{equation}
In particular, the path made of $c_0$ and its consecutive successors correspond to the geodesic $\gamma$ mentioned in the introduction.
More generally, for every $u,v\in\theta$, we have the bound 
\begin{equation}\label{Borne}
    d_{\Phi(\theta)}(u,v)\geq|\ell_u-\ell_v|.
\end{equation}

\paragraph{For infinite labelled trees.}\label{CVS infinite}

\begin{figure}
        \centering
        \includegraphics[scale=0.5]{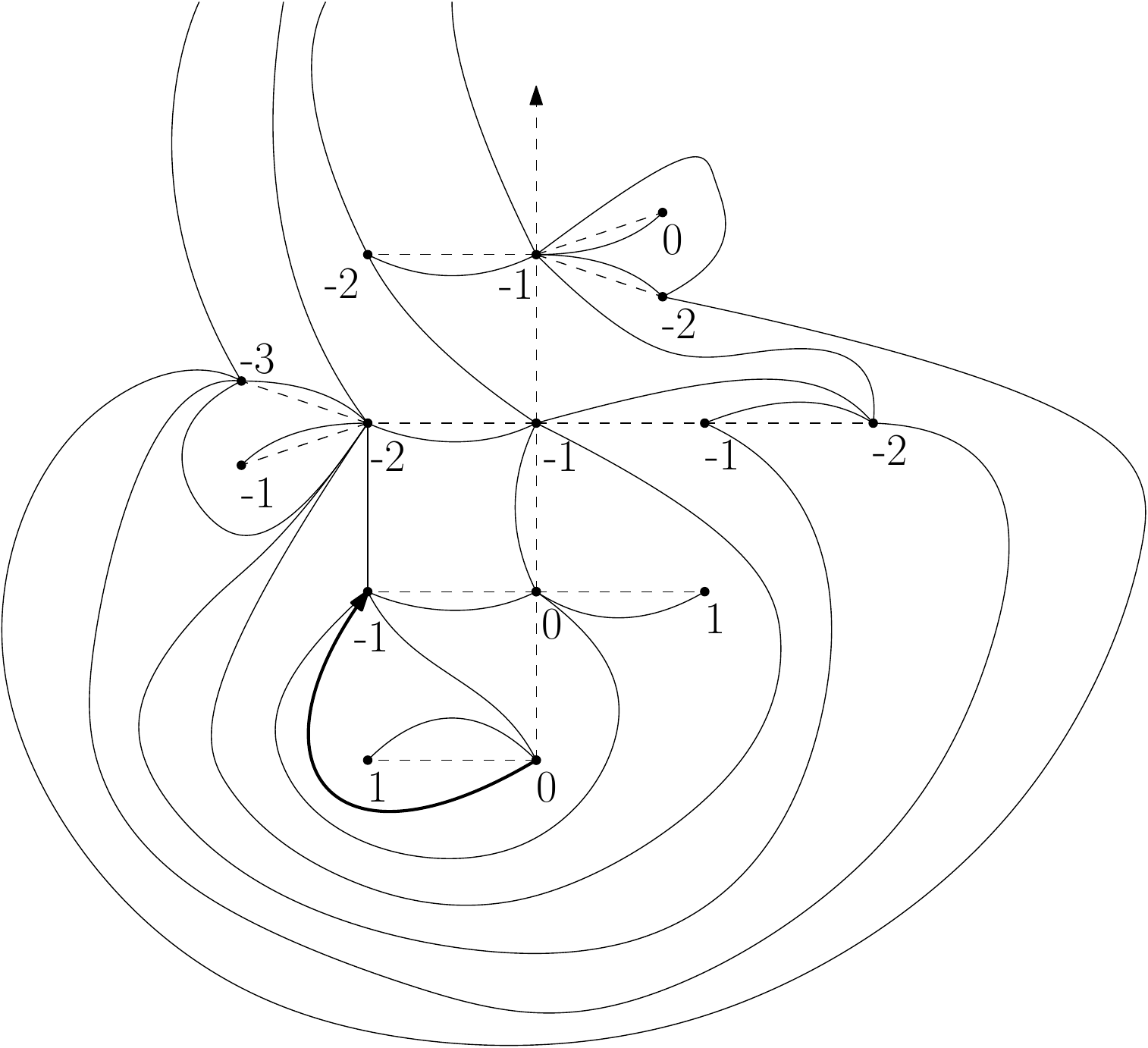}
        \caption{An illustration of the CVS bijection when $\theta\in\S_{-}$.}
        \label{CVS UIPQ}
\end{figure}

Here, we present some extensions of the $\mathrm{CVS}$ bijection to some families of infinite labelled trees. These families are the following :
\begin{align*}
\mathbb{S}_-:&=\{\theta\in\mathbb{S},\inf_{n\geq0}\ell(S(n))=-\infty\}.\,\\
\mathbb{S}_{(1)}:&=\left\{\theta\in\mathbb{S}:\,\ell(c_0)=0,\min_{n\leq-1}\ell(c_n)=1,\lim_{n\rightarrow-\infty}\ell(c_n)=+\infty\text{ and }\inf_{n\geq0}\ell(c_n)=-\infty\right\}\\
    \mathbb{S}_{(2)}:&=\left\{\theta\in\mathbb{S}:\,\ell(c_0)=1,\min_{n\geq 1}\ell(c_n)=2,\lim_{n\rightarrow\infty}\ell(c_n)=+\infty\text{ and }\inf_{n\leq0}\ell(c_n)=-\infty\right\}.
\end{align*} 
Given a labelled tree $\theta\in\mathbb{S}$, consider its corner sequence $(c_i)_{i\in\mathbb{Z}}$. 
    \begin{itemize}[label=\textbullet]
        \item If $\theta\in\mathbb{S}_-$, we can define the successor of a corner with formula \eqref{Successor}. Note that in this case, every corner has a successor. Then, we can construct a graph just like we did in the final case (but we don't need to add a distinguished vertex $v_*$). See Figure \ref{CVS UIPQ} for an example.
        \item If $\theta\in\mathbb{S}_{(1)}$, we can proceed as in the previous case. See Figure \ref{CVS UIPQ1} for an example.
        \item If $\theta\in\mathbb{S}_{(2)}$, we consider $(\lambda_i)_{i\in\mathbb{Z}}$ a sequence of points outside the planar imbedding of $\theta$. For every $i\in\Z$, we let where $i'>i$ be the smallest integer such that $\ell(c_{i'})=\ell(c_i)-1$ (if it exists). We define the successor of a corner $c_i$ as 
        \[\sigma(c_i)=\left\{
    \begin{array}{ll}
   c_{i'}\quad&\text{if $i'$ it exists, } \\
     \lambda_{\ell(c_i)-1}\quad&\text{ otherwise. }
    \end{array}\right.\]
    \end{itemize}
    Also set $\ell(\lambda_i)=i$. Then, we consider the graph obtained by drawing an edge between every corner and its successor, together with an edge between every pair $(\lambda_i,\lambda_{i-1})$.
    We abuse notations by letting $\Phi(\theta)$ stand for either of the previous constructions. The following result was proved in \cite{Viewfrominfinity} and \cite{dieuleveut}.
    \begin{proposition}\label{CVS general}
        For every $\theta\in\mathbb{S}_-$, the graph $\Phi(\theta)$ is a quadrangulation of the plane. Moreover, for every $\theta\in\mathbb{S}_{(1)}\cup\mathbb{S}_{(2)}$, the graph $\Phi(\theta)$ is a quadrangulation of the half plane. 
    \end{proposition}

    \begin{figure}
        \centering
        \includegraphics[scale=0.4]{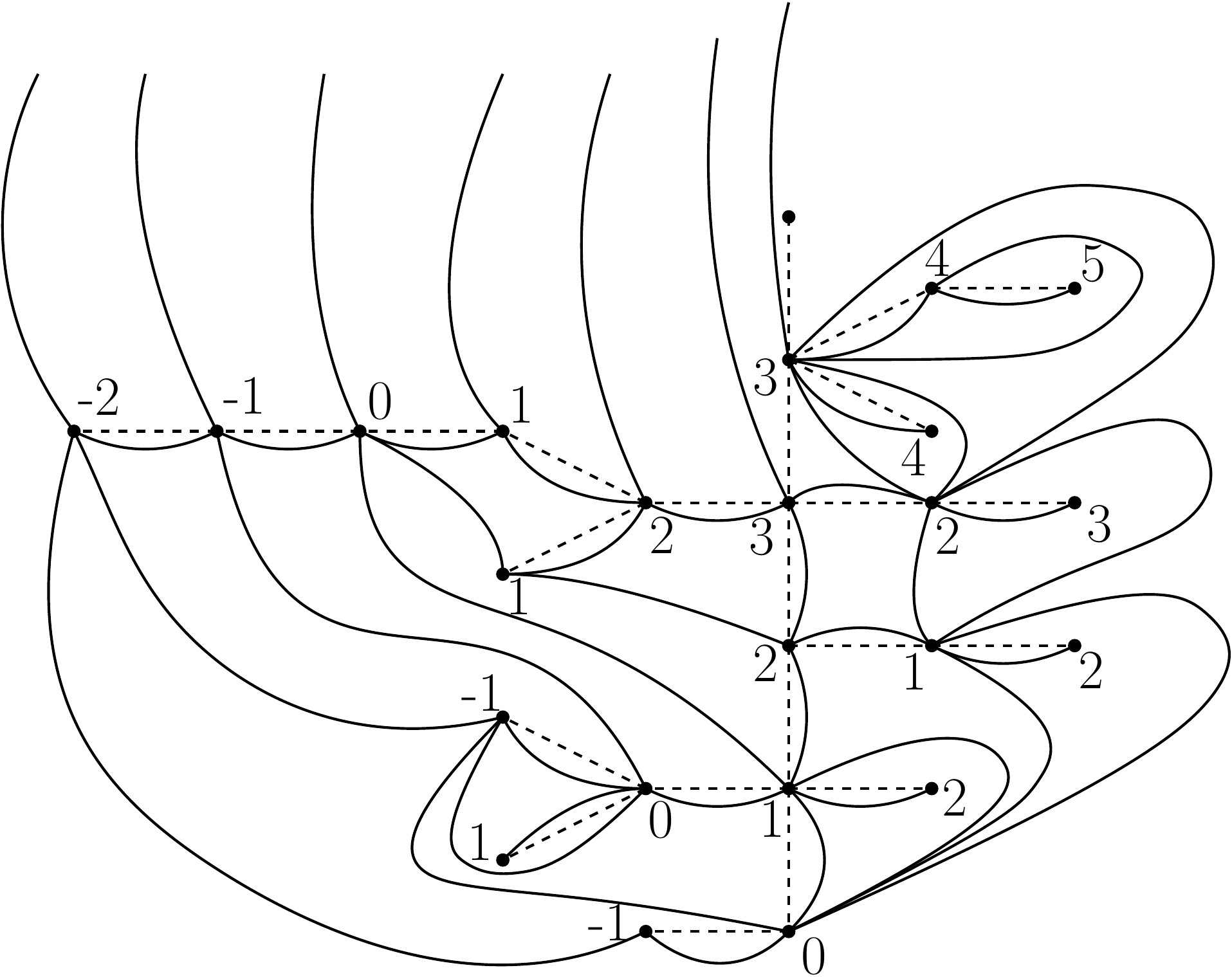}
        \caption{An illustration of the CVS bijection when $\theta\in\S_{(1)}$.}
        \label{CVS UIPQ1}
\end{figure}

\subsection{Random labeled trees}\label{Random labeled trees}

In this section, we introduce the three different random labeled trees that we will study.

\paragraph{The conditioned uniform labeled tree.}\label{Proba arbre}

For every $x\in \mathbb{Z}$, let $\rho_{(x)}$ be the distribution on $\T^{(x)}$ defined by the following formula: for every $\theta\in\mathbb{T}^{(x)}$, we have 
\[\rho_{(x)}(\{\theta\})=\frac{1}{2\times12^{|\theta|}},\]

so that the tree is distributed as a critical Bienaymé-Galton-Watson with offspring distribution $\mathrm{Geom}(1/2)$, and the labels are chosen uniformly at random among all possibles labellings with a root label $x$.  

Let $\mathbb{T}^+$ be the set of labeled tree with positive labels. It is shown in \cite[Proposition 2.4]{ChassaingDurhuus} that for every $x\geq 0$,
\begin{equation}\label{minimum discret}
    \rho_{(x)}(\mathbb{T}^+)=w(x):=\frac{x(x+3)}{(x+1)(x+2)}.
\end{equation}
Hence, we can define $\rho_{(x)}^+$ as the law $\rho_{(x)}$ conditioned to have positive labels. In particular, for every $\theta\in \mathbb{T}^+\cap\mathbb{T}^{(x)}$, we have 
\begin{equation}\label{Condition weight}
\rho_{(x)}^+(\{\theta\})=\frac{1}{2w(x)12^{|\theta|}}.
\end{equation}
Finally, for every $x\geq0$, we define $\rho_{(x)}^-$ as the law $\rho_{(x)}$ conditioned to have a non-positive label. Consequently, for every $\theta\in\mathbb{T}^{(x)}\backslash\mathbb{T}^+$, we have 
\[\rho_{(x)}^-(\{\theta\})=\frac{1}{2(1-w(x))12^{|\theta|}}.\]
In what follows, we will consider a random variable $\Theta_n$ whose distribution is $\rho_{(n)}^-$.

\paragraph{The uniform infinite labeled tree.} 

We consider a random variable $\Theta_\infty=(\mathbf{t}_\infty,\ell_\infty)$ which takes value in $\mathbb{T}_\infty^{(0)}\cap\mathbb{S}$ as follows :
\begin{itemize}[label=\textbullet]
    \item The root has label $0$
    \item The labels of the spine evolve as a random walk with uniform steps in $\{-1,0,1\}$
    \item Given the labels on the spine $(\ell(S(n)))_{n\geq0}=(x_n)_{n\geq0}$, the sequences of subtrees $(L_n)_{n\geq0}$ and $(R_n)_{n\geq0}$ form to independent sequences of independent labeled trees with law $(\rho_{(x_n)})_{n\geq0}$.
\end{itemize}
Finally, note that almost surely, $\Theta_\infty\in\mathbb{S}_-$.

\paragraph{The infinite rerooted labeled tree.}\label{inf rerooted}

Before introducing one last random tree, we need to introduce a Markov chain $X$ taking values in $\N$. Its transition probabilities are given as follows : we have $\P(X_1=1\,|\,X_0=0)=1$, and otherwise, for $x\geq 1$, 
\begin{align}\label{Transition}
    p(x,x+1):=\P(X_1=x+1|X_0=x)&=\frac{(x+4)(2x+5)}{3(x+2)(2x+3)}=\frac{f(x+1)w(x)}{3f(x)}\nonumber\\
    p(x,x):=\P(X_1=x|X_0=x)&=\frac{x(x+3)}{3(x+1)(x+2)}=\frac{w(x)}{3}\\
    p(x,x-1):=\P(X_1=x-1|X_0=x)&=\frac{(x-1)(2x+1)}{3(x+1)(2x+3)}=\frac{f(x-1)w(x)}{3f(x)}\nonumber
\end{align} 
where $f(x)=x(x+3)(2x+3)$. Moreover, if $\P_x=\P(\cdot\,|\,X_0=x)$ and
\[H_k=\inf\{n\geq1,\,X_n=k\},\] 
the proof of \cite[Lemma 2.3]{dieuleveut} shows that
\begin{equation}\label{Probabilité de retour 2}
    \P_k(H_k<\infty)=\frac{6k-1}{3(2k+3)}
\end{equation}

and 
\begin{equation}\label{Probabilité de retour}
    \P_x(H_k<\infty)=\frac{h(k)}{h(x)}
\end{equation} 
for $x>k$, where $h(x)=x(x+1)(x+2)(x+3)(2x+3).$\\
Finally, as observed in \cite{dieuleveut}, a Theorem of Lamperti \cite{Lamperti} shows that we have the convergence
\begin{equation}\label{convergence Bessel}
\left(\frac{X_{\lfloor nt\rfloor}}{\sqrt{n}}\right)_{t\geq0}\xrightarrow[n\rightarrow\infty]{(d)}  (Z_{2t/3})_{t\geq0},
\end{equation} 
where $Z$ is a Bessel process of dimension $7$ starting at $0$.

With this process being defined, we introduce another random variable $\overline{\Theta}_\infty^{(1)}$ taking values in $\mathbb{T}_\infty^{(1)}\cap\mathbb{S}$ as follows :
\begin{itemize}[label=\textbullet]
    \item The labels on the spine $(S_n)_{n\geq0}$ evolve as the Markov chain $X$ started from $0$
    \item Given the labels on the spine $(\ell(S_n))_{n\geq0}=(y_n)_{n\geq0}$, the sequences of subtrees $(L_n)_{n\geq0}$ and $(R_n)_{n\geq0}$ form two independent sequences of independent labeled trees with respective laws $(\rho_{(y_n)})_{n\geq0}$ and $(\rho_{(y_n)}^+)_{n\geq0}$.
\end{itemize}

We will write $(X^{(\infty)},F^{(\infty)}_1,F^{(\infty)}_2)$ for the process and the forests encoding this random tree.
Finally, we also introduce the random variable $\overline{\Theta}_\infty^{(2)}=(\overline{T}_\infty^{(2)},\overline{\ell}_\infty^{(2)})$, which is independent of $\overline{\Theta}_\infty^{(1)}=(\overline{T}_\infty^{(1)},\overline{\ell}_\infty^{(1)})$, and whose distribution is related to the one of $\overline{\Theta}_\infty^{(1)}$ by the following relations :
\begin{itemize}[label=\textbullet]
    \item $(\ell(S_n(\overline{\Theta}_\infty^{(2)}))_{n\geq0}\overset{(d)}{=}(\ell(S_n(\overline{\Theta}_\infty^{(1)})+1)_{n\geq0}$
    \item $(L_n(\overline{T}_\infty^{(2)}),\overline{\ell}_\infty^{(2)})_{n\geq0}\overset{(d)}{=}(R_n(\overline{T}_\infty^{(1)}),\overline{\ell}_\infty^{(1)}+1)_{n\geq0}$
    \item $(R_n(\overline{T}_\infty^{(2)}),\overline{\ell}_\infty^{(2)})_{n\geq0}\overset{(d)}{=}(L_n(\overline{T}_\infty^{(1)}),\overline{\ell}_\infty^{(1)}+1)_{n\geq0}$.
\end{itemize}
As previously, we write $(\widehat{X}^{(\infty)},\widehat{F}^{(\infty)}_1,\widehat{F}^{(\infty)}_2)$ for the process and the forests encoding the random tree $\overline{\Theta}_\infty^{(2)}$.

Note that for these two trees, the labels on the spine diverge toward $+\infty$ almost surely. More precisely, almost surely, we have $\Theta_\infty^{(1)}\in\mathbb{S}_{(1)}$ and $\Theta_\infty^{(2)}\in\mathbb{S}_{(2)}$.
    
\subsection{Boltzmann planar maps, the $\mathrm{UIPQ}$ and $\mathrm{\overline{UIPQ}}$}\label{random maps}

We introduce three random quadrangulations that will be studied in this work. First, for every $n\geq1$, consider the random tree $\Theta_n$ introduced in Subsection \ref{Random labeled trees}, and set $Q_n=\Phi(\Theta_n)$. As a consequence of \eqref{Distance}, this random quadrangulation is a critical positive Boltzmann pointed quadrangulation, conditioned on $\{d_Q(v_0,v_*)\geq n\}$. Furthermore, for every $0\leq i\leq n-\ell_*$, let $\tau_n(i)$ (respectively $\widehat{\tau}_n(i)$) be the first corner in $\Theta_n$ with label $n-i$ after the root, in clockwise order (respectively counter-clockwise order). It is easy to see that the path $\gamma_n:=(\tau_n(i))_{0\leq i\leq n-\ell_*}$ is a geodesic in $Q_n$ between $v_0$ and $v_*$, and that $\tau_n(i)$ is the successor of $\widehat{\tau}_n(i-1)$.

Then, we define $Q_\infty=\Phi(\Theta_\infty)$. This random quadrangulation of the plane is called the UIPQ (Uniform Infinite Planar Quadrangulation), and was introduced in \cite{Krikun} as the local limit of uniform quadrangulations with $n$ faces as $n\rightarrow\infty$. This space is a central object in random geometry, and has been extensively studied (see \cite{ChassaingDurhuus,MenardLaw,LegallMenard,SeparatingCycles} for instance). The definition used in this article was proven to be equivalent to the previous ones in \cite{Viewfrominfinity}. For every $i>0$, let $\tau_\infty(i)$ be the first corner with label $-n$ after the root, in clockwise order. We use the same notation for the associated vertices. Then, the path $\gamma_\infty:=(\tau_\infty(i))_{i\geq0}$ is an infinite geodesic starting at the root, called the left-most geodesic. Furthermore, it was shown in \cite{Viewfrominfinity} that every infinite geodesic starting from the root must contain infinitely many vertices of $\gamma_\infty$. For later purposes, for every $n>0$, we also let $\widehat{\tau}_\infty(i)$ be the first corner with label $-i$ after the root, in counterclockwise order. Note that these vertices do not form a geodesic, but we always have $d_{Q_\infty}(\tau_\infty(i),\widehat{\tau}_\infty(i-1))=1$, since $\tau_\infty(i)$ is the successor of $\widehat{\tau}_\infty(i-1))=1$.

Similarly, we define $\overline{Q}_\infty^{(1)}=\Phi(\overline{\Theta}_\infty^{(1)})$ and $\overline{Q}_\infty^{(2)}=\Phi(\overline{\Theta}_\infty^{(2)})$. These random quadrangulations of the half-plane were introduced and studied in \cite{dieuleveut}. Just as previously, for every $n\in\mathbb{Z}$, let $\overline{\tau}_\infty^{(1)}(n)$ be either 
\begin{itemize}[label=\textbullet]
    \item the first corner in $\overline{\Theta}_\infty^{(1)}$ with label $n$ after the root in clockwise order if $n\leq0$, 
    \item or the last corner in $\overline{\Theta}_\infty^{(1)}$ with label $n$ after the root in counter-clockwise order if $n\geq0$.
\end{itemize}
We also set $\overline{\tau}_\infty^{(2)}(n)=\lambda_n$, where the vertices $(\lambda_i)_{i\in\Z}$ were introduced right before Proposition \ref{CVS general}. The two paths $\overline{\gamma}_\infty^{(1)}=(\overline{\tau}_\infty^{(1)}(n))_{n\in\mathbb{Z}}$ and $\overline{\gamma}_\infty^{(2)}=(\overline{\tau}_\infty^{(2)}(n))_{n\in\mathbb{Z}}$ are infinite bigeodesics of $\overline{Q}_\infty^{(1)}$ and $\overline{Q}_\infty^{(2)}$ respectively. Finally, we define $\overline{Q}_\infty$ as the gluing of $\overline{Q}_\infty^{(1)}$ and $\overline{Q}_\infty^{(2)}$ along the curves $\overline{\gamma}_\infty^{(1)}$ and $\overline{\gamma}_\infty^{(2)}$. More precisely, for every $n\in\mathbb{Z}$, we identify the vertices $\overline{\tau}_\infty^{(1)}(i)$ and $\overline{\tau}_\infty^{(2)}(i)$. We call $\overline{\gamma}_\infty$ the curve obtained after the gluing, and we root this space at $\overline{\gamma}_\infty(0)$. The resulting space, which we denote by $\overline{\mathrm{UIPQ}}$, is a random quadrangulation of the plane, and appears as the local limit of the UIPQ, re-rooted further and further along its left most geodesic $\gamma_\infty$ (see \cite{dieuleveut}). 

\subsection{The Brownian sphere and the bigeodesic Brownian plane}

Here, we recall a few results about the Brownian sphere and the bigeodesic Brownian plane that will be needed in this article. The standard Brownian sphere $(\mathcal{S},D)$ is a random compact metric space, which has a root $x_0$ and a distinguished point $x_*$. It also has a volume measure $\mu$, of mass $1$. This random space appears as the scaling limit of several models of random planar maps (see \cite{ConvergenceBJM,ConvergenceSimple,AbrahamConvergence,uniqueness,convergence}). To construct the Brownian sphere, we use a measurable function $\Psi$ that associates to any pair $(B,Z)$ of functions satisfying some properties a pointed metric space $\Psi(B,Z)$.  The Brownian sphere is defined as $\Psi(\cB,\cZ)$ for a well-chosen pair of random processes $(\cB,\cZ)$. However, we will not give an explicit definition of the function $\Psi$ here, and refer to \cite{uniqueness,convergence} for details.

It is sometimes more convenient to work with the free Brownian sphere, which is constructed from the excursion measure of the Brownian snake $\mathbb{N}_0$, which can be viewed as a measure on pairs of function $(B,Z)$ satisfying the conditions required to apply $\Psi$. In this case, the volume is not fixed, and corresponds to the duration of the encoding processes. We will not give more details about this measure here, and we refer to \cite{serpent} for a detailed study of this object. 

Almost surely, there exists a unique geodesic $\Gamma$ between $x_0$ and $x_*$ (see \cite{geodesic1}). The length of this geodesic corresponds to the minimum label reached by the Brownian snake under $\mathbb{N}_0$, called $W_*$. Even though $\mathbb{N}_0$ is a $\sigma$-finite measure, for every $b>0$, the event $\{W_*<-b\}$ has a finite mass. Therefore, we can define a probability measure $\mathbb{N}_0(\cdot\,|\,W_*<-b)$, which corresponds to the Brownian sphere with $D(x_0,x_*)>b$. 

The following result deals with the scaling limit the contour and label processes of $\Theta_n$, together with the quadrangulation $Q_n$. Recall that, informally, the contour and label processes of a labelled tree record the distance to the root and the label as one moves around the tree in clockwise order (see \cite{LeGall2005} for more details).

\begin{theorem}\label{Convergergence sphere}
   Let $(X_n,L_n)$ be the contour and label processes of a random tree $\Theta_n$, that we interpolate linearly between integer times. Also let $(Q_n,d_{Q_n},v_0^{(n)},v_*^{(n)})$ be the random positive quadrangulation associated to \textit{the same} random tree $\Theta_n$ via the CVS bijection. We have the following joint convergence 
 \begin{multline*}
     \left(\left(\frac{X_n(\lfloor n^4t\rfloor)}{n^2},\frac{L_n(\lfloor n^4t\rfloor)}{n}\right)_{t\in\R_+},\left(Q_n,\frac{\sqrt{3}}{\sqrt{2}n}d_{Q_n},v_0^{(n)},v_*^{(n)}\right)\right)
     \xrightarrow[n\rightarrow\infty]{(d)}\left(\left(\cB_t,\cZ_t\right)_{t\in\R_+},\left(\mathcal{S},D,x_0,x_*\right)\right)
 \end{multline*}
        for the Skorokhod topology for the first coordinate and the marked Gromov-Hausdorff topology for the second  coordinate, where $(\cB,\cZ)$ is distributed under $\mathbb{N}_0(\cdot\,|\,W_*<-1)$, and $\left(\mathcal{S},D,x_0,x_*\right)=\Psi(\cB,\cZ)$. 
\end{theorem}
This theorem follows from the convergence uniform quadrangulations towards the standard Brownian sphere \cite{uniqueness,convergence} by using standard arguments from excursion theory. We postpone the proof to the Appendix \ref{appendix0}.

On the other hand, the bigeodesic Brownian plane $(\overline{\mathcal{BP}},\overline{D},\overline{\rho}_\infty)$ is a non-compact random metric space, that was introduced and studied in \cite{Bigeodesicbrownianplane}. This space is the tangent plane (in distribution) of the Brownian sphere at a point of the geodesic $\Gamma$. More precisely, we have the following convergence result.
\begin{theorem}\label{convergergenceBP}
     For every $\delta>0$ and $b>0$, we can find $\varepsilon>0$ such that we can construct $(\cS,D,\rho)$ (under $\N_0(\cdot\,|\,W_*<-b)$) and $(\overline{\mathcal{BP}},\overline{D}_\infty,\overline{\rho}_\infty)$ on the same probability space such that $B_\varepsilon(\cS,\Gamma_b)$ and $B_\varepsilon(\overline{\mathcal{BP}})$ are isometric with probability at least $1-\delta$. Moreover, we have the convergence 
    \begin{equation}        (\cS,\varepsilon^{-1}D,\Gamma_b)\xrightarrow[\varepsilon\rightarrow 0]{(d)}(\overline{\mathcal{BP}},\overline{D}_\infty,\overline{\rho}_\infty)
    \end{equation}
    in distribution for the local Gromov-Hausdorff topology. 
\end{theorem}
Even though our convergence results involve the bigeodesic Brownian plane, we do not need any result about this space other than Theorem \ref{convergergenceBP}. However, we emphasize that the proof of Theorem \ref{Main result} relies on several methods used to prove Theorem \ref{convergergenceBP}.

\section{Coupling trees}\label{Coupling trees}

In this section, we establish coupling results between the trees $\Theta_n,\Theta_\infty$ and $\overline{\Theta}_\infty$. For $\Theta_\infty$ and $\overline{\Theta}_\infty$, this is an improvement of \cite[Theorem 1.4]{dieuleveut}. The proof relies on spine decomposition of $\Theta_n$ and explicit calculations in $\Theta_\infty$. We start with some notations.

For every $n\in\N$, let $\eta_n$ (respectively $\widehat{\eta}_{n}$) be the geodesic path (in $\Theta_n$) from $\tau_n(n)$ to $\rho_n$ (respectively, from $\widehat{\tau}_{n}(n-1)$ to $\rho_n$). Let $\kappa_m$ (respectively $\widehat{\kappa}_m$) be the last hitting time of $m$ by $\ell(\eta_n)$ (respectively $\ell(\widehat{\eta}_n)$). Then, for every $\beta<1$, let $A_{\beta n}(\Theta_n)$ be the subset of $\Theta_n$ formed by the path $ \{\eta_n(0),...,\eta_n(\kappa_{\lfloor\beta n\rfloor})\}$, together with the subtrees branching off this path (see Figure \ref{Fig ensemble arbres}).
It is clear that $A_{\beta n}(\Theta_n)$ is a tree, which corresponds to a portion of the trajectory between $\tau_n(n)$ and $\rho_n$, together with the subtrees branching off this trajectory. We also define $\widehat{A}_{\beta n}(\Theta_n)$ in a similar fashion, but with the trajectory $\widehat{\eta}_{n}$ instead of $\eta_n.$ Note that $A_{\beta n}(\Theta_n)$ and $\widehat{A}_{\beta n}(\Theta_n)$ are not necessarily disjoint.

\begin{figure}
    \centering
    \includegraphics[width=0.6\linewidth]{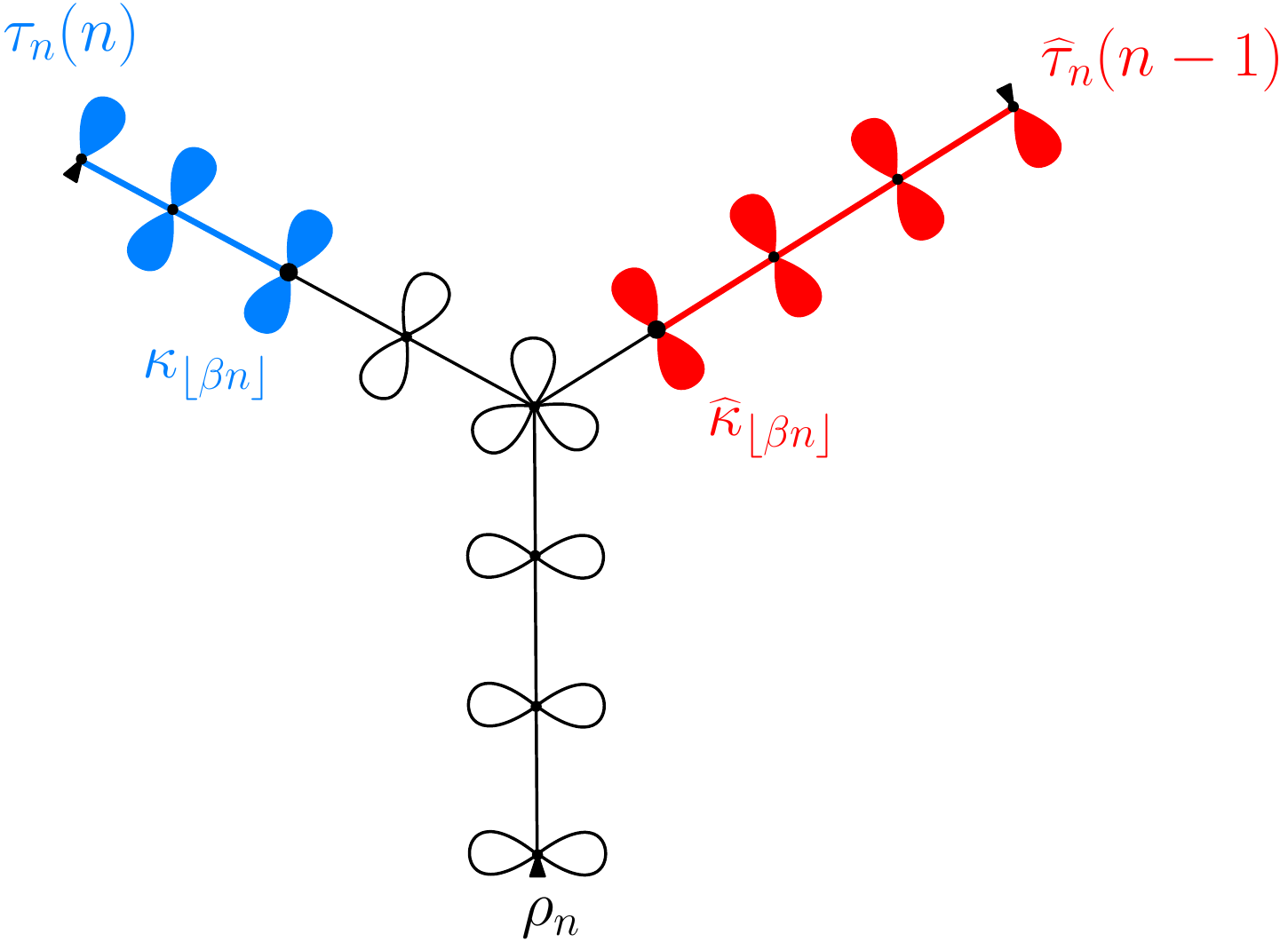}
    \caption{Illustration of the sets $A_{\beta n}(\Theta_n)$ and $\widehat{A}_{\beta n}(\Theta_n)$, respectively in blue and red.}
    \label{Fig ensemble arbres}
\end{figure}

Then, let $A_{\beta n}(\overline{\Theta}_\infty^{(1)})$ and $\widehat{A}_{\beta n}(\overline{\Theta}_\infty^{(2)})$ be the restrictions of $\overline{\Theta}_\infty^{(1)}$ and $\overline{\Theta}_\infty^{(2)}$ to their unique infinite spines up to their last hitting time of $\beta n$, together with the subtrees branching off these paths. 

\subsection{Spine decompositions for discrete trees}\label{sec spine}

We start by giving new constructions of the random labelled tree $\Theta_k$ introduced in Section \ref{Proba arbre}, which rely on spine decompositions. 

Consider a random tree $\mathbb{T}_k$ obtained as follows. First, let $(X^{(k)}_n)_{0\leq n\leq S_k}$ denote the process $X$ starting at $0$ and stopped at the last time it visits $k$. Then, consider a line formed by $S_k$ points, equipped with labels given by the process $X^{(k)}$. Then, conditionally on $X^{(k)}$ :
\begin{itemize}[label=\textbullet]
    \item For every $0\leq i\leq S_k$, graft to the left of the vertex $i$ independent trees with laws $(\rho_{X_i^k})_{0\leq i\leq S_k}$.
    \item For every $1\leq i\leq S_k$, graft to the right of the vertex $i$ independent trees with laws $(\rho_{X_i^k}^+)_{1\leq i\leq S_k}$.
\end{itemize}
Finally, we root this tree at the bottom of the spine, which is identified with $0$. Observe that this tree is very similar to a ``finite version'' of the tree $\overline{\Theta}_\infty^{(1)}$.
\begin{proposition}\label{decompo1}
    The random tree $\mathbb{T}_k$ has the same law as $\Theta_k$, rerooted at $\tau_k(k)$.
\end{proposition}
\begin{proof}
    Note that every tree $\theta\in\mathbb{T}^{(x)}\backslash\mathbb{T}^+$ can be entirely described by the path between its first vertex with label $0$ and the root, together with the subtrees grafted to this path. Fix $\mathbf{x}=(x_0,...,x_n)$ a path in $\N^{n+1}$ with 
positive labels (except for $x_0$) satisfying $|x_{i+1}-x_i|\leq1$, $x_0=0$ and $x_n=k$, and two labeled forests $\mathbf{f}=(f_0,...,f_n),\,\mathbf{f'}=(f'_1,...,f'_n)$ such that the root of $f_i$ and $f'_i$ have label $x_i$, and the trees $(f'_i)$ have positive labels. For simplicity, we denote by $\left[\mathbf{x},\mathbf{f},\mathbf{f'}\right]=\theta$ the resulting tree. We have, using \eqref{Condition weight}, 
      \begin{align*}
\P\left(\mathbb{T}_k=\left[\mathbf{x},\mathbf{f},\mathbf{f'}\right]\right)&=\left(\prod_{i=1}^{n-1}p(x_i,x_{i+1})\right)\left(\prod_{i=0}^n\rho_{x_i}(f_i)\right)\left(\prod_{i=1}^n\rho^+_{x_i}(f'_i)\right)\P_k(T_k=\infty)\\
    &=\frac{f(k)}{f(1)3^{n-1}}\frac{10}{3(2k+3)}\left(\prod_{i=1}^{n-1}w(x_i)\right)\left(\prod_{i=0}^n \rho_{x_i}(f_i)\right)\left(\prod_{i=1}^n \rho^+_{x_i}(f'_i)\right)\\
    &=\frac{f(k)}{2\times3^{n}(2k+3)}\left(\prod_{i=1}^{n-1}w(x_i)\right)\left(\prod_{i=0}^n \frac{1}{2\times 12^{|f_i|}}\right)\left(\prod_{i=1}^n \frac{1}{2w(x_i) 12^{|f'_i|}}\right)\\
    &=\frac{f(k)}{2\times3^{n}w(k)(2k+3)}\left(\prod_{i=0}^n \frac{1}{2\times 12^{|f_i|}}\right)\left(\prod_{i=1}^n \frac{1}{2\times 12^{|f'_i|}}\right).
      \end{align*}
    Since
      \[\frac{f(k)}{w(k)(2k+3)}=(k+1)(k+2)=\frac{2}{1-w(k)}\]
      and
\begin{equation}
    |\theta|=n+\sum_{i=1}^n |f_i|+\sum_{i=0}^n |f'_i|,
\end{equation}
      we have
      \begin{align*}
         \P\left( \mathbb{T}_k=\left[\mathbf{x},\mathbf{f},\mathbf{f'}\right]\right)&=\frac{1}{2\times3^n(1-w(k))}\frac{1}{4^n12^{\sum_{i=1}^n |f_i|+\sum_{i=0}^n |f'_i|}}\\
          &=\frac{1}{(1-w(k))2\times12^{|\theta|}}\\
          &=\P\left(\Theta_k=\theta\right)
      \end{align*}         
\end{proof}

Before proving our next result, let us explain how to obtain a labelled tree from three paths and six forests. Informally, the paths represent the labels along three distinguished branches of a tree, and the forest encodes the subtrees grafter to these branches. For every $m,x,x'\in\N$, let $\mathcal{M}_{m,x\rightarrow x'}$ be the set of walks $\mathbf{x}=(x_0,...,x_m)\in\Z^{m+1}$ such that $x_0=x,x_m=x'$ and for every $0\leq i\leq m-1$, $|x_{i+1}-x_i|\leq 1$. Consider three paths $\mathbf{x}\in \cM_{a,k\rightarrow n},\mathbf{\widehat{x}}\in \cM_{b,l\rightarrow n},\mathbf{y}\in \cM_{c,m\rightarrow n}$ for any admissible values of $a,b,c>0$ and $k,l,m,n\in\Z$, the only condition being that these three paths have the same terminal value. Then, consider six forests of labelled trees $(\mathbf{f}^0,\mathbf{\widehat{f}}^0,\mathbf{f}^1,\mathbf{f}^2,\mathbf{\widehat{f}}^1,\mathbf{\widehat{f}}^2)$ of sizes $(c+1,c,a,a,b,b)$. One can glue these paths and forest together, in order to obtain a tree as shown in Figure \ref{Fig trois bouts}. In order to have a well defined labelled tree, the root label of each subtree must be equal to the value of the path at the point where it is grafted. We denote by \[\left[\mathbf{y},\mathbf{x},\mathbf{\widehat{x}},\mathbf{f}^0,\mathbf{\widehat{f}}^0,\mathbf{f}^1,\mathbf{f}^2,\mathbf{\widehat{f}}^1,\mathbf{\widehat{f}}^2\right]\] the resulting labelled tree, which is rooted at the vertex corresponding to $y_0$. Note that this tree also comes with two distinguished corners, which are the ones associated to $x_0$ and $\widehat{x}_0$.  

\begin{figure}
    \centering
    \includegraphics[width=0.55\linewidth]{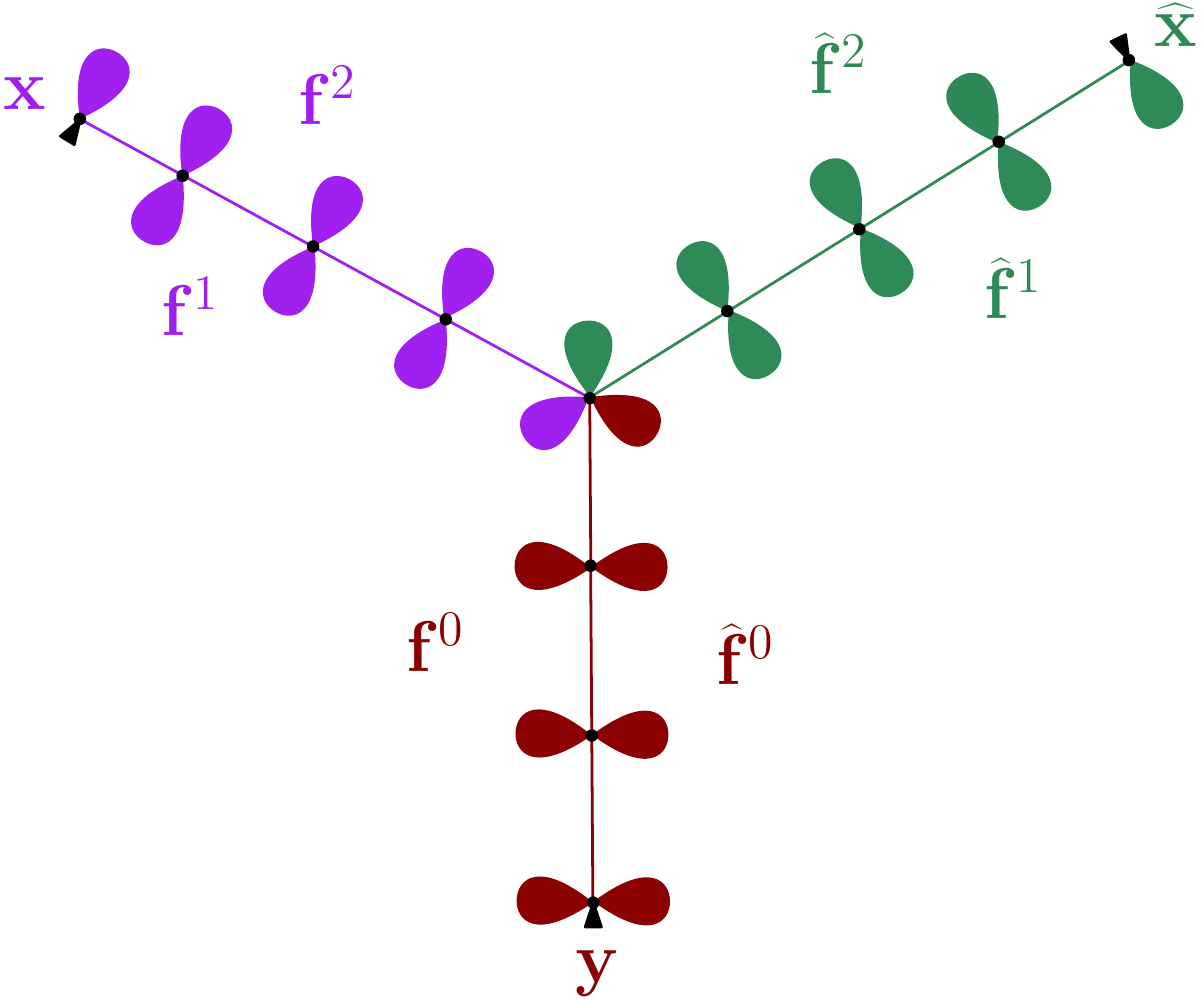}
    \caption{The random tree $\Theta_k$, obtained by gluing three subtrees}
    \label{Fig trois bouts}
\end{figure} 

Using this approach, we give another construction of the random tree $\Theta_n$, which is more symmetric with respect to $\tau_k(k)$ and $\widehat{\tau}_k(k-1)$. We mention that this construction is similar to the one of the random tree $\Theta_\infty^{(n)}$ that will be given in Section \ref{Couplage Arbre 2}. To this end, we introduce another Markov chain $(Y_n)_{n\geq0}$ with values in $\N$ and absorbed in $0$ with transition probabilities : 
\begin{align}\label{Transition2}
    \P(Y_{n+1}=x+1\,|\,Y_n=x)=q(x,x+1)=\frac{w(x)(1-w(x+1))}{3(1-w(x))}\nonumber\\
    \P(Y_{n+1}=x\,|\,Y_n=x)=q(x,x)=\frac{w(x)}{3}\\
    \P(Y_{n+1}=x-1\,|\,Y_n=x)=q(x,x-1)=\frac{w(x)(1-w(x-1))}{3(1-w(x))}.\nonumber
\end{align}
One can check that these formulas define probability transitions for a Markov chain. We write $T_0$ for its first hitting time of $0$. Furthermore, conditionally on $Y$, we define a random variable $S$ with values in $\llbracket 0,T_0\rrbracket$ by the formula :
\begin{align}\label{Formule J}
\P(J=x\,|\,Y)=(1-w(Y_x-1))\prod_{i=0}^{x-1}w(Y_i-1).
\end{align}
Note that, considering a sequence of independent trees with law $(\rho_{Y_i})_{0\leq i\leq T_0}$, this formula is also the probability that the first tree with a label below $1$ is the $x-$th one.

Consider four random variables $(Y^{(k)},J,X^{(k)},\widehat{X}^{(k)})$, defined as follows 
\begin{itemize}[label=\textbullet]
    \item $Y^{(k)}$ has transition probabilities given by \eqref{Transition2} and starts at $k$,
    \item $J$ is the random variable associated to $Y^{(k)}$ defined with the formula \eqref{Formule J}
    \item Conditionally on $(Y^{(k)},J)$, $X^{(k)}=\left(X^{(k)}_i\right)_{0\leq i \leq S_{Y^{(k)}_J}}$ and $\widehat{X}-1=\left(\widehat{X}^{(k)}_i-1\right)_{0\leq i \leq \widehat{S}_{Y^{(k)}_J}}$ are two independent Markov chains with transition probabilities \eqref{Transition} starting at $0$, and stopped at their last hitting time of $Y^{(k)}_J$.
\end{itemize}
We denote these last hitting times by $S_{Y^{(k)}_J}$ and $\widehat{S}_{Y^{(k)}_J}$.
Furthermore, conditionally on $(Y^{(k)},J,X^{(k)},\widehat{X}^{(k)})$, we consider six forests $\left(F_0^{(k)},\widehat{F}_0^{(k)},F_1^{(k)},F_2^{(k)},\widehat{F}_1^{(k)},\widehat{F}_2^{(k)}\right)$ with the following distributions : 
\begin{itemize}
    \item The six forests are independent. 
    \item The trees $\left(F^{(k)}_{0,i}\right)_{0\leq i\leq J-1}$ and $\left(\widehat{F}^{(k)}_{0,i}\right)_{0\leq i\leq J}$ are independent random variables, and the tree $F^{(k)}_{0,i}$ (respectively $\widehat{F}^{(k)}_{0,i}$) has law $\rho^+_{Y_i^{(k)}}$ (respectively the law obtained by adding $1$ to the labels of $\rho^+_{Y_i^{(k)}-1}$).
    \item The trees $\left(F^{(k)}_{1,i}\right)_{1\leq i\leq S_{Y^{(k)}_J}}$ and $\left(F^{(k)}_{2,i}\right)_{0\leq i\leq S_{Y^{(k)}_J}-1}$ are independent, and $F^{(k)}_{1,i}$ (respectively $F^{(k)}_{2,i}$) has law $\rho^+_{X_i^{(k)}}$ (respectively $\rho_{X_i^{(k)}}$).
    \item The trees $\left(\widehat{F}^{(k)}_{1,i}\right)_{0\leq i\leq \widehat{S}_{Y^{(k)}_J}-1}$ and $\left(\widehat{F}^{(k)}_{2,i}\right)_{1\leq i\leq \widehat{S}_{Y^{(k)}_J}}$ are independent, and $\widehat{F}^{(k)}_{1,i}$ (respectively $\widehat{F}^{(k)}_{2,i}$) has the law obtained by adding $1$ to the labels of $\rho^+_{\widehat{X}_i^{(k)}}$ (respectively $\rho_{\widehat{X}_i^{(k)}}$).
\end{itemize}
As explained previously, we can use these processes to construct a random labelled tree, namely 
\[\mathbb{T}'_k:=\left[\left(Y^{(k)}_i\right)_{0\leq i\leq J},X^{(k)},\widehat{X}^{(k)},F_0^{(k)},\widehat{F}_0^{(k)},F_1^{(k)},F_2^{(k)},\widehat{F}_1^{(k)},\widehat{F}_2^{(k)}\right].\]
We construct a random tree $\mathbb{T}'_k$ by gluing these paths together with their subtrees, as depicted in Figure \ref{Fig trois bouts}. 
The following result identifies the law of $\mathbb{T}'_k$. 
\begin{proposition}\label{Three arms}
    The random tree $\mathbb{T}'_k$ has the same law as $\Theta_k$.
\end{proposition}
\begin{proof}
    We proceed as in Proposition \ref{decompo1}. Consider three paths $(\mathbf{y},\mathbf{x},\mathbf{\widehat{x}})$ (of respective durations $l,n,m$) satisfying $y_0=b,x_0=0,\widehat{x}_0=1,$ and $y_l=x_n=\widehat{x}_m$. We also consider 6 forests $(\mathbf{f}^0,\mathbf{\widehat{f}}^0,\mathbf{f}^1,\mathbf{f}^2,\mathbf{\widehat{f}}^1,\mathbf{\widehat{f}}^2)$ with appropriate sizes (we keep the notations introduced previously), and let $\left[\mathbf{y},\mathbf{x},\mathbf{\widehat{x}},\mathbf{f}^0,\mathbf{\widehat{f}}^0,\mathbf{f}^1,\mathbf{f}^2,\mathbf{\widehat{f}}^1,\mathbf{\widehat{f}}^2\right]=\theta$ be the tree obtained from these paths and forests. We have 
\[\P\left(\mathbb{T}'_k=\left[\mathbf{y},\mathbf{x},...,\mathbf{\widehat{f}^2}\right]\right)=A(\mathbf{y})  A(\mathbf{x}) A(\mathbf{\widehat{x}})\P(J=l\,|\,\mathbf{y})\] where 
\begin{align*}
    &A(\mathbf{y})=\left(\prod_{i=0}^{l-1}q(y_i,y_{i+1})\right)\left(\prod_{i=0}^l \rho^+_{y_i}(f^0_i) \rho^+_{y_i-1}(\widehat{f}^0_i-1)\right)\\
    &A(\mathbf{x})=\left(\prod_{i=1}^{n-1}p(x_i,x_{i+1})\right)\left(\prod_{i=0}^n\rho_{x_i}(f^2_i)\right)\left(\prod_{i=1}^{n-1}\rho^+_{x_i}(f^1_i)\right)\P_{x_n}(H_{x_n}=\infty)\\
    &A(\widehat{\mathbf{x}})=\left(\prod_{i=1}^{m-1}p(\widehat{x}_i-1,\widehat{x}_{i+1}-1)\right)\left(\prod_{i=0}^{m-1}\rho_{\widehat{x}_i-1}(\widehat{f}^2_i-1) \right)\left(\prod_{i=1}^{m-1}\rho^+_{\widehat{x}_i-1}(\widehat{f}^1_i-1)\right)\P_{\widehat{x}_n-1}(H_{\widehat{x}_n-1}=\infty)
\end{align*} and
\[\P(J=l\,|\,\mathbf{y})=(1-w(y_l-1))\prod_{i=0}^{l-1}w(y_i-1).\] The last terms in $A(\mathbf{x})$ and $A(\mathbf{\widehat{x}})$ comes from the fact that the paths $X^{(k)}$ and $\widehat{X}^{(k)}$ are stopped at their last hitting time of $Y^{(k)}_J$. Moreover, by \eqref{Probabilité de retour 2}, this term is explicit.

Recall that $w(x)=\frac{x(x+3)}{(x+1)(x+2)}$, and $f(x)=x(x+3)(2x+3)$.
First, note that 
\[\prod_{i=0}^{l-1}q(y_i,y_{i+1})=\frac{1-w(y_l)}{3^l(1-w(k))}\prod_{i=0}^{l-1}w(y_i).\]
Similarly, 
\[\prod_{i=1}^{n-1}p(x_i,x_{i+1})=\frac{f(x_n)}{f(1)3^{n-1}}\prod_{i=1}^{n-1}w(x_i)\quad\text{ and }\quad\prod_{i=1}^{m-1}p(\widehat{x}_i-1,\widehat{x}_{i+1}-1)=\frac{f(\widehat{x}_m-1)}{f(1)3^{m-1}}\prod_{i=1}^{m-1}w(\widehat{x}_i-1).\]
Using the fact that $y_l=x_n=\widehat{x}_m$, and by formula $\eqref{Probabilité de retour 2}$ we have 
\[f(y_l)f(y_l-1)(1-w(y_l))(1-w(y_l-1))\P_{y_l}(H_{y_l}=\infty)\P_{y_l-1}(H_{y_l-1}=\infty)=\frac{400w(y_l)w(y_l-1)}{9}.\]
Putting together these calculations, and since $f(1)=20$, we obtain 
\begin{multline*}
    \P\left(\mathbb{T}'_k=\left[\mathbf{y},\mathbf{x},...,\mathbf{\widehat{f}^2}\right]\right)=\frac{B(\mathbf{y},\mathbf{x},...,\mathbf{\widehat{f}^2})}{3^{l+n+m}}
    \frac{w(y_l)w(y_l-1)}{36(1-w(k))}\\ \times\left(\prod_{i=0}^{l-1}w(y_i)w(y_i-1)\right)\left(\prod_{i=1}^{n-1}w(x_i)\right)\left(\prod_{i=1}^{m-1}w(\widehat{x}_i-1)\right),
\end{multline*}

where 
\begin{multline*}
    B(\mathbf{y},\mathbf{x},...,\mathbf{\widehat{f}^2})=\left(\prod_{i=0}^l \rho^+_{y_i}(f^0_i) \rho^+_{y_i-1}(\widehat{f}^0_i-1)\right)\left(\prod_{i=0}^n\rho_{x_i}(f^2_i)\right)\left(\prod_{i=1}^{n-1}\rho^+_{x_i}(f^1_i)\right)\\ \times\left(\prod_{i=0}^{m-1}\rho_{\widehat{x}_i-1}(\widehat{f}^2_i-1)\right)\left(\prod_{i=1}^{m-1}\rho^+_{\widehat{x}_i-1}(\widehat{f}^1_i-1)\right).
\end{multline*}

Finally, using \eqref{Condition weight}, we have 
\begin{equation*}    B(\mathbf{y},\mathbf{x},...,\mathbf{\widehat{f}^2})=\frac{\left(\prod_{i=0}^{l}w(y_i)w(y_i-1)\right)^{-1}\left(\prod_{i=1}^{n-1}w(x_i)\right)^{-1}\left(\prod_{i=1}^{m-1}w(\widehat{x}_i-1)\right)^{-1}}{2\times4^{l+n+m}\times 12^{|\mathbf{f}^0|+|\mathbf{\widehat{f}}^0|+|\mathbf{f}^1|+|\mathbf{\widehat{f}}^1|+|\mathbf{f}^2|+|\mathbf{\widehat{f}}^2|}}.
\end{equation*}
Consequently, this gives
\begin{align*}
    \mathbb{T}'_k(\theta=(\mathbf{y},\mathbf{x},...,\mathbf{\widehat{f}^2}))&=\frac{1}{2\times12^{|\theta|}(1-w(k))}\\
    &=\Theta_k(\theta)
\end{align*}
\end{proof}
\begin{remark}
    These two decompositions are reminiscent of \cite[Proposition 3.2, Proposition 3.6]{Bigeodesicbrownianplane} for a random labelled tree with law $\N_0(\cdot\,|\,W_*<-b)$. 
\end{remark}
\begin{remark}\label{Remark1}
    Using Proposition \ref{Three arms}, if we identify $\Theta_n$ with $\mathbb{T}_n'$, we see that $A_{\beta n}(\Theta_n)$ and $\widehat{A}_{\beta n}(\Theta_n)$ are disjoint if and only if $Y^{(n)}_J>\beta n$.
\end{remark}

\subsection{Coupling $\Theta_n$ with $(\overline{\Theta}_\infty^{(1)},\overline{\Theta}_\infty^{(2)})$}\label{CouplageArbre1}

First, we establish the following coupling result between the trees $\Theta_n$ and $(\overline{\Theta}_\infty^{(1)},\overline{\Theta}_\infty^{(2)})$.
\begin{proposition}\label{CouplingTrees1}
    For every $\delta>0$, there exists $\beta>0$ and $n_0\in\N$ such that for every $n>n_0$, we can couple $\Theta_n$ and $(\overline{\Theta}_\infty^{(1)},\overline{\Theta}_\infty^{(2)})$ in such a way that the equality 
    \begin{equation*}        
    \left(A_{\beta n}(\Theta_n),\widehat{A}_{\beta n}(\Theta_n)\right)=\left(A_{\beta n}(\overline{\Theta}^{(1)}_\infty),\widehat{A}_{\beta n}(\overline{\Theta}^{(2)}_\infty)\right)
    \end{equation*}
    holds with probability at least $1-\delta$.
\end{proposition}
Before proving this statement, we need an intermediate lemma. 
\begin{lemme}\label{ancetre commun}
    There exists a random variable $\cZ_{\tilde{\cT}}$ with values in 
$(0,\infty)$ such that 
    \begin{equation}\label{ancêtre convergence}
        \frac{Y^{(n)}_S}{n}\xrightarrow[n\rightarrow\infty]{(d)}\cZ_{\tilde{\cT}}.
    \end{equation}
\end{lemme}
\begin{proof}
    Recall that $Y^{(n)}_S$ corresponds to the label of the first common ancestor of $\tau_n(n)$ and $\widehat{\tau}_n(n-1)$ in $\Theta_n$. Consider $(C_n,Z_n)$ the contour and label processes of the random tree $\Theta_n$, and set 
    \[T_n=\inf\{i\geq0,\,Z_n(i)=0\}\quad\text{ and }\quad\widehat{T}_n=\sup\{i\geq0,\,Z_n(i)=1\}\]
    (these two times correspond to the vertices $\tau_n(n)$ and $\widehat{\tau}_n(n-1)$).
    We recall that these processes, once rescaled properly, converge in distribution toward a pair $(\cC,\cZ)$ (see Theorem \ref{Convergergence sphere}). We also define  
    \[\cT=\inf\{t\geq0,\,\cZ_t=0\}\quad\text{ and }\quad\widehat{\cT}=\sup\{t\geq0,\,\cZ_t=0\}.\]
  By the Markov property of the Brownian snake and with standard arguments from excursion theory, we know that for every $\varepsilon>0$, we have $\inf_{t\in[\cT,\cT+\varepsilon]}\cZ_t<0$ a.s (see \cite[Lemma 4.7.1]{Duquesne}. Together with Theorem \ref{Convergergence sphere}, this implies that, jointly with the convergence of $(C_n,Z_n)$, we have
    \[\left(\frac{T_n}{n^4},\frac{\widehat{T}_n}{n^4}\right)\xrightarrow[n\rightarrow\infty]{(d)}\left(\cT,\widehat{\cT}\right).\]
    Then, let $\tilde{T}_n=\inf\left\{i\geq0,\,C_n(i)=\min_{T_n\leq j\leq\widehat{T}_n}C_n(j)\right\}$. This time corresponds to the first common ancestor of $\tau_n(n)$ and $\widehat{\tau}_n(n-1)$. The previous convergences imply that  
    \[\frac{\tilde{T}_n}{n^4}\xrightarrow[n\rightarrow\infty]{(d)}\tilde{\cT},\] where $\tilde{\cT}=\inf\left\{t\geq0,\,\cC_t=\inf_{s\in[\cT,\widehat{\cT}]}\cC_s\right\}$. Finally, this gives 
    \[\frac{Y_S^{(n)}}{n}=\frac{Z_n\left(n^4\times\frac{\tilde{T}_n}{n^4}\right)}{n}\xrightarrow[n\rightarrow\infty]{(d)}\cZ_{\tilde{\cT}}.\]
    To conclude, we just need to check that $\P(\cZ_{\tilde{\cT}}=0)=0$. Let $\cR=\sup\{0\leq t\leq\cT,\cC_t=\cC_{\tilde{\cT}}\}.$ It is easy to see that $\cZ_{\cR}=\cZ_{\tilde{\cT}}$, and that $\cR<\cT$ a.s. By definition, this implies that $\cZ_{\cR}>0$, which concludes the proof. 
\end{proof}
\begin{proof}[Proof of Proposition \ref{CouplingTrees1}]
   Fix $\delta>0$. By Lemma \ref{ancetre commun}, there exists $\beta>0$ such that for every $n\in\N$ large enough, we have 
   \begin{equation*}
       \P(Y^{(n)}_H>\beta n)>1-\delta/2.
   \end{equation*}
   Using Proposition \ref{Three arms} and as mentioned in Remark \ref{Remark1}, on the event $\{Y^{(n)}_H>\beta n\}$, we see that $A_{\beta n}(\Theta_n)$ (respectively $\widehat{A}_{\beta n}(\Theta_n)$) is composed of the path $X^{(n)}$ (respectively $\widehat{X}^{(n)}$) up to its last hitting time of $\beta n$, together with the subtrees branching off this path.

   Then, note that the total variation distance between the laws of \[\left(X^{(n)}|_{\llbracket 0,S_{\lfloor \beta n\rfloor\wedge Y^{(n)}_H}\rrbracket},\widehat{X}^{(n)}|_{\llbracket 0,S_{\lfloor \beta n\rfloor\wedge Y^{(n)}_H}\rrbracket}\right) \quad\text{ and } \quad\left(X^{(\infty)}|_{\llbracket 0,S^{(\infty)}_{\lfloor \beta n\rfloor}\rrbracket},\widehat{X}^{(\infty)}|_{\llbracket 0,\widehat{S}^{(\infty)}_{\lfloor \beta n\rfloor}\rrbracket}\right)\] 
   is bounded above by $2\P(Y^{(n)}_H<\beta n)$. Therefore, we can construct these processes on the same property space in such a way that they are equal with probability at least $1-\delta$.\\
   Finally, on the event where these processes are equal, by comparing Proposition \ref{Three arms} with the construction of $\overline{\Theta}_\infty^{(1)}$ and $\overline{\Theta}_\infty^{(2)}$ given in Section \ref{Random labeled trees}, we see that we can couple the forests \[\left(F^{(n)}_1,F^{(n)}_2,\widehat{F}^{(n)}_1,\widehat{F}^{(n)}_2\right)\quad\text{ and }\quad \left(F^{(\infty)}_1,F^{(\infty)}_2,\widehat{F}^{(\infty)}_1,\widehat{F}^{(\infty)}_2\right)\] so that the subtrees appearing in $\left(A_{\beta n}(\Theta_n),\widehat{A}_{\beta n}(\Theta_n)\right)$ and $\left(A_{\beta n}(\overline{\Theta}^{(1)}_\infty),\widehat{A}_{\beta n}(\overline{\Theta}^{(2)}_\infty)\right)$ are the same. Hence, on an event of probability at least $1-\delta$, the processes encoding the sets $\left(A_{\beta n}(\Theta_n),\widehat{A}_{\beta n}(\Theta_n)\right)$ and $\left(A_{\beta n}(\overline{\Theta}^{(1)}_\infty),\widehat{A}_{\beta n}(\overline{\Theta}^{(2)}_\infty)\right)$ are equal, which gives the wanted result. 
\end{proof}

\subsection{Coupling $\Theta_\infty$ with $(\overline{\Theta}_\infty^{(1)},\overline{\Theta}_\infty^{(2)})$}\label{Couplage Arbre 2}

The goal of this section is to prove a result which is analogous to Proposition \ref{CouplageArbre1}, but with the tree $\Theta_\infty$ instead of $\Theta_n$.

 However, the calculations are more technical in this case. To prove this result, we rely on another construction of $\Theta_\infty$, which was obtained in \cite{dieuleveut}. This construction is based on a decomposition of $\Theta_\infty$ into three paths, together with subtrees branching off these paths.

We start with some notation. For every $m,x,x'\in\N$, let $\mathcal{M}^+_{m,x\rightarrow x'}$ be the set of walks $\mathbf{x}=(x_0,...,x_m)\in\N^{m+1}$ such that $x_0=x,x_m=x'$ and for every $0\leq i\leq m-1$, $|x_{i+1}-x_i|\leq 1$. We also define $\mathcal{M}^{>0}_{m,x\rightarrow x'}$ (resp. $\mathcal{M}^{>1}_{m,x\rightarrow x'}$) to be the elements of $\mathcal{M}^+_{m,x\rightarrow x'}$ such that $x_i>0$ (resp. $x_i>1$) for every $1\leq i\leq m$. Finally, given a path $(x_0,...,x_m)$, let $\mu_{(x_0,...,x_m)}$ denote the distribution of the labelled forest $(\tau_0,...,\tau_m)$ defined as follows. Let $I$ be a uniform random variable in $\{0,...,m\}$. Then, $\tau_I$ is a random tree distributed as $(\Theta_\infty,\ell_\infty+x_I)$ and for every $i\in\{0,...,m\}\backslash\{I\}$, $\tau_i$ is an independent random tree distributed as $\rho_{x_i}$, and independent of $\tau_I$. \\

We can now apply this construction with random paths and forests. Let \[H_\infty^{(n)}=(a_\infty^{(n)},b_\infty^{(n)},c_\infty^{(n)},X_\infty^{(n)},\widehat{X}_\infty^{(n)},Y_\infty^{(n)})\] be a random variable on $\N^3\times\mathcal{M}^3$ with the following distribution : for every $a,b,c,k\geq1,\,\mathbf{x}\in\mathcal{M}^{>0}_{a,0\rightarrow k},\,\mathbf{y}\in\mathcal{M}^{>1}_{b,1\rightarrow k}\,\mathbf{z}\in\mathcal{M}^{>1}_{c,k\rightarrow n}$, 
\[\P\bigg(H_\infty^{(n)}=(a,b,c,\mathbf{x},\mathbf{y},\mathbf{z})\bigg)=\frac{a+b+1}{3^{a+b+c}}\bigg(\prod_{i=1}^a w(x_i)\bigg)\bigg(\prod_{i=1}^b w(y_i-1)\bigg)\bigg(\prod_{i=1}^{c} w(z_i)w(z_i-1)\bigg).\]
Then, we need to consider the subtrees branching off the previous paths. Given $H_\infty^{(n)}=(a,b,c,\mathbf{x},\mathbf{y},\mathbf{z})$, consider six random forests  $\left(F^{(\infty,n)}_0,\widehat{F}^{(\infty,n)}_0,F^{(\infty,n)}_1,F^{(\infty,n)}_2,\widehat{F}^{(\infty,n)}_1,\widehat{F}^{(\infty,n)}_2\right)$ of sizes 
$(c,c+1,a,a,b,b)$ with the following distribution : 
\begin{itemize}
    \item The six forests are independent. 
    \item The trees $\left(F^{(\infty,n)}_{0,i}\right)_{0\leq i\leq c-1}$ and $\left(\widehat{F}^{(\infty,n)}_{0,i}\right)_{0\leq i\leq c}$ are independent random variables, and the tree $F^{(\infty,n)}_{0,i}$ (respectively $\widehat{F}^{(\infty,n)}_{0,i}$) has law $\rho^+_{z_i}$ (respectively the law obtained by adding $1$ to the labels of $\rho^+_{z_i-1}$).
    \item The trees $\left(F^{(\infty,n)}_{1,i}\right)_{1\leq i\leq a}$ and $\left(\widehat{F}^{(\infty,n)}_{1,i}\right)_{0\leq i\leq b-1}$ are independent, and $F^{(\infty,n)}_{1,i}$ (respectively $\widehat{F}^{(\infty,n)}_{1,i}$) has law $\rho^+_{x_i}$ (respectively the law obtained by adding $1$ to the labels of $\rho^+_{y_i-1}$).
    \item The forest $\left(F^{(\infty,n)}_{2,0},...,F^{(\infty,n)}_{2,a-1},\widehat{F}^{(\infty,n)}_{2,b},...,\widehat{F}^{(\infty,n)}_{2,1}\right)$ has the law of $\mu_{(x_0,...,x_{a-1},y_{b},...,y_1)}$.
\end{itemize}

As explained in Section \ref{sec spine}, we define  
\[\Theta_\infty^{(n)}=\left[Y^{(n)}_\infty,X^{(n)}_\infty,\widehat{X}^{(n)}_\infty,F_0^{(\infty,n)},\widehat{F}_0^{(\infty,n)},F_1^{(k)},F_2^{(\infty,n)},\widehat{F}_1^{(\infty,n)},\widehat{F}_2^{(\infty,n)}\right]\]
We also let $I_\infty^{(n)}$ be the uniform random variable in $\left[0,a_\infty^{(n)}+b_\infty^{(n)}\right]$ corresponding to the rank of the unique infinite tree of the forest $\left(F^{(\infty,n)}_{2,0},...,F^{(\infty,n)}_{2,a-1},\widehat{F}^{(\infty,n)}_{2,b},...,\widehat{F}^{(\infty,n)}_{2,}\right)$. The following proposition is just a restatement of \cite[Corollary 3.2]{dieuleveut}.

\begin{proposition}\label{construction alternative}
    The random labelled tree $\Theta_\infty^{(n)}$ rooted at the vertex corresponding to $x_0$ is distributed as $\left(\mathbf{t}_\infty,\ell_\infty+n\right)$  rooted at $\tau_\infty(n)$.
\end{proposition}

Since a shift of labels does not impact the CVS bijection, for every $n\in\N$, $\Phi(\Theta_\infty^{(n)})$ is also distributed as $Q_\infty$ rerooted at the corner $\tau_\infty(n)$. For every $n\in\N$, we abuse notation by letting $\tau_\infty(i)$ (respectively $\widehat{\tau}_\infty(i)$) be the first corner of $\Theta_\infty^{(n)}$ with label $n-i$ after the root, in clockwise order (respectively in counter-clockwise order). In order to avoid confusions, we will always specify which tree we are considering before using this notation. In order to study $Q_\infty$ near $\gamma_\infty(n)$, it will be more convenient to work with the tree $\Theta_\infty^{(n)}$. By construction, $\tau_\infty(n)$ and $\widehat{\tau}_\infty(n-1)$ correspond to the corners associated to $x_0$ and $y_0$ in $\Theta_\infty^{(n)}$. 

We define $A_{\beta n}(\Theta_\infty^{(n)})$ and $\widehat{A}_{\beta n}(\Theta_\infty^{(n)})$ as subtrees of $\Theta_\infty^{(n)}$, just like we defined $A_{\beta n}(\Theta_n)$ and $\widehat{A}_{\beta n}(\Theta_n)$ for $\Theta_n$, \textit{mutatis mutandis}. Note that in this case, we can use Proposition \ref{construction alternative} to study $A_{\beta n}(\Theta_\infty^{(n)})$ and $\widehat{A}_{\beta n}(\Theta_\infty^{(n)})$. Moreover, these trees may be infinite, but we will always work on events where both trees are finite. With all these objects at hand, we can state as the main result of this section.

\begin{proposition}\label{CouplingTrees2}
    For every $\varepsilon>0$, there exists $\beta>0$ and $n_0\in\N$ such that for every $n>n_0$, we can couple $\Theta_\infty^{(n)}$ and $(\overline{\Theta}_\infty^{(1)},\overline{\Theta}_\infty^{(2)})$ in such a way that the equality 
    \begin{equation*}        
    \left(A_{\beta n}(\Theta_\infty^{(n)}),\widehat{A}_{\beta n}(\Theta_\infty^{(n)})\right)=\left(A_{\beta n}(\overline{\Theta}^{(1)}_\infty),\widehat{A}_{\beta n}(\overline{\Theta}^{(2)}_\infty)\right)
    \end{equation*}
    holds with probability at least $1-\varepsilon$.
\end{proposition}

 The proof of Proposition \ref{CouplingTrees2} rely on three technical lemmas.
\begin{proposition}\label{control 1}
    For every $\varepsilon>0$, there exists $c>0$ such that for every $n$ large enough,
    \begin{equation*}
     \P(a_\infty^{(n)}>cn^2)>1-\varepsilon   
    \end{equation*}
    and
    \begin{equation*}
        \P\left(cn^2<I_\infty^{(n)}<\left(a_\infty^{(n)}+b_\infty^{(n)}-cn^2\right)\right)>1-\varepsilon.
    \end{equation*}
\end{proposition}
\begin{proof}
    The labeled tree $\Theta_\infty$ is encoded by $(X_\infty,L_\infty)$, where $X_\infty$ is a two-sided simple random walk and $L_\infty$ is the label process of the infinite forest encoded by $X$. By \cite[Theorem 3]{ConvergenceSerpent}, we have : 
    \begin{equation}\label{Convergence Brownian}
        \bigg(\bigg(\frac{X_\infty({\lfloor nt \rfloor})}{n^{1/2}}\bigg)_{t\in \R},\,\bigg(\frac{L_\infty({\lfloor nt \rfloor})}{n^{1/4}}\bigg)_{t\in\R}\bigg)\xrightarrow[n\rightarrow\infty]{(d)}\bigg(\mathcal{B}_t,\mathcal{Z}_t\bigg)_{t\in\R},
    \end{equation}
    where $\mathcal{B}$ is a two-sided standard Brownian motion and $\mathcal{Z}$ the Brownian snake driven by $\mathcal{B}$. Let $S$ be the infinite spine of the tree $\Theta_\infty$, and for every $n\geq0$, set 
    \[\delta_n=\inf\{k\geq0,L_\infty(k)=-n\}.\]
    Then, observe that both $a_\infty^{(n)}+1$ and $I_\infty^{(n)}$ stochastically dominate the random variable
    \[D_n:=d_{\Theta_\infty}(\tau_\infty(n),S)=X_{\delta_n}-\inf_{0\leq u\leq\delta_n}X_u.\]
    Indeed, we have $I_\infty^{(n)}=d_{\Theta_\infty}(\tau_\infty(n),S)$ if $I_\infty^{(n)}\in[0,a_\infty^{(n)}+1]$, and $I_\infty^{(n)}\geq d_{\Theta_\infty}(\tau_\infty(n),S)=a_\infty^{(n)}+1$ otherwise.
    The convergence \eqref{Convergence Brownian} implies that 
    \[\frac{D_n}{n^2}\xrightarrow[n\rightarrow\infty]{}\mathcal{B}_{T_1}-\inf_{0\leq s\leq T_1}\mathcal{B}_s\]
    where $T_1=\inf\{s\geq0,\,\mathcal{Z}_s=-1\}.$
    Furthermore, as a consequence of \cite[Lemma 2.2]{sphericity}, we know that $p_{\cT_\infty}(T_1)$ is almost surely a leaf of $\cT_\infty$, which means that $\mathcal{B}_{T_1}-\inf_{0\leq s\leq T_1}\mathcal{B}_s>0$ a.s. The first inequality follows from this observation and the convergence \eqref{Convergence Brownian}, by taking $c$ small enough. 
    To obtain the second inequality, we just need to show that
    \[\P\left(I_\infty^{(n)}<a_\infty^{(n)}+b_\infty^{(n)}-cn^2\right)>1-\varepsilon.\]
    However, this follows from the first part of the proof (up to reducing $c$), since $I_\infty^{(n)}$ is uniform on $[0,a_\infty^{(n)}+b_\infty^{(n)}]$. This concludes the proof.
\end{proof}
Next, the following lemma deals with the convergence of the last hitting time of the process $X$. 
\begin{lemme}\label{convergence temps d'atteinte}
    For every $n\in\N$, define 
    \begin{equation*}
        S_n=\sup\{i\in\N,\,X_i=n\}
    \end{equation*}
    and for every $t\in\R_+$
    \begin{equation*}
        \mathcal{S}_t=\sup\{s\geq0,\,Z_s=t\}.
    \end{equation*}
    Then, for every $t\geq0$,
    \begin{equation*}
        \frac{S_{\lfloor nt\rfloor}}{n^2}\xrightarrow[n\rightarrow\infty]{(d)}\frac{3}{2}\mathcal{S}_t.
    \end{equation*}
\end{lemme}
\begin{proof}
    To lighten notations, we only give the proof for $t=1$. The idea is to show that with an arbitrary large probability, the event $\{\sup_{t\geq s}X_s>1\}$ can be verified on a compact time interval, on which we can use the convergence of the process.

For $a\in\N$, denote $\tau_a=\inf\{n\geq0:X_n=a\}$. 
From the proof of \cite[Lemma 2.3]{dieuleveut}, we have, for every $x>k$:
\begin{equation*}
 \P_{x}(\tau_k<\infty)=\frac{h(k)}{h(x)}
\end{equation*} where $h(x)=x(x+1)(x+2)(x+3)(2x+3)$. In particular, for every $a>1$, 
\begin{equation}\label{momo}
    \P_{\lfloor an\rfloor}(\tau_n<\infty)\xrightarrow[n\rightarrow\infty]{}a^{-5}.
\end{equation}
Using the Markov property, for every $n\in\N$, $z\in\R_+$ and $a,b$, we have 
\begin{align*}
  \P(S_{n}\leq zn^2)&=\P\left(\inf_{s>zn^2}X_s>n\right)\\  &=\P\left(\inf_{s>zn^2}X_s>n,\,X_{bn^2}>an\right)+\P\left(\inf_{s>zn^2}X_s>n,\,X_{bn^2}<an\right)\\  &=\E\left[\1_{\{\inf_{zn^2<s<bn^2}X_s>n,\,X_{bn^2}>an\}}\E_{X_{bn^2}}[\1_{\{\tau_n=+\infty\}}]\right]+\P\left(\inf_{s>zn^2}X_s>n,\,X_{bn^2}<an\right).
\end{align*}
On the event $\{X_{bn^2}>an\}$, using \eqref{momo}, we have $1-2a^{-5}<\P_{X_{bn^2}}(\tau_n=+\infty)<1$ for $n$ large enough, and thus 
\begin{equation}\label{Satie}
    \left|\P(S_n\leq zn^2)-\P\left(\inf_{zn^2<s<bn^2}X_s>n,\,X_{bn^2}>an\right)\right|\leq\P\left(\inf_{s>zn^2}X_s>n,\,X_{bn^2}<an\right)+2a^{-5}.
\end{equation}
Furthermore, we have 
\begin{equation}\label{Tonka}
\P\left(\inf_{s>zn^2}X_s>n,\,X_{bn^2}<an\right)\leq \P(X_{bn^2}<an)\xrightarrow[n\rightarrow\infty]{}\P(Z_{2b/3}<a).    
\end{equation}
We can now prove the statement. Fix $\varepsilon>0,\,z\geq 0$, and choose $a,b>1$ such that 
\begin{itemize}[label=\textbullet]
    \item $a^{-5}<\varepsilon/12$
    \item $\P(Z_{2b/3}<a)<\varepsilon/12$ (which is possible because $Z$ diverges to $\infty$).
\end{itemize}
We have 
\begin{align}
    |\P(S_n\leq zn^2)-&\P(\cS_1\leq 2z/3)|\leq \left|\P(S_n\leq zn^2)-\P\left(\inf_{zn^2<s<bn^2}X_s>n,\,X_{bn^2}>an\right)\right|\label{calcul}\\
    &+\left|\P\left(\inf_{zn^2<s<bn^2}X_s>n,\,X_{bn^2}>an\right)-\P\left(\inf_{2z/3<s<2b/3}Z_s>1,\,Z_{2b/3}>a\right)\right|\nonumber\\    &+\left|\P\left(\inf_{2z/3<s<2b/3}Z_s>1,\,Z_{2b/3}>a\right)-\P(\cS_1\leq 2z/3)\right|.\nonumber
\end{align}
By \eqref{Satie} and \eqref{Tonka}, the first term of the sum is smaller than $\varepsilon/3$ for $n$ large enough. The second term also goes to 0 as $n\rightarrow\infty$, as a consequence of the convergence of our processes on compact sets. Finally, with the same idea that we used to obtain \eqref{Satie}, we have 
\[\left|\P\left(\inf_{2t/3<s<2b/3}Z_s>1,\,Z_{2b/3}>a\right)-\P(\cS_1>2t/3)\right|\leq \P_a(T_1<\infty)+\P(Z_{2b/3}<a).\]
Then, we can use the explicit formula (see the proof of \cite[Proposition 2.1]{lawlerbessel}) :
\[\P_a(T_1<\infty)=a^{-5}<\varepsilon/12,\]
we have proved that the right-hand side of \eqref{calcul} is smaller than $\varepsilon$ for $n$ large enough. The result follows, $\varepsilon$ being arbitrary.
\end{proof}
Finally, we need a strong coupling result between the pair of processes $\left(X^{(n)}_{\infty,k},Y^{(n)}_{\infty,k}\right)$ and $(\widehat{X}_k,\widehat{Y}_k)$. We mention that similar results were already obtained in \cite[Proposition 3.3]{dieuleveut}, where it has been proved that the processes converge in distribution. However, this is not sufficient for our purposes, since we need to compare the processes for duration that depends on $n$.   
\begin{proposition}\label{couplage processus 2}
    For every $\varepsilon>0$, there exist $\beta>0$ and $n_0>0$ such that for every $n>n_0$, the total variation distance between the law of $(X^{(n)}_{\infty,k},Y^{(n)}_{\infty,k})_{0\leq k\leq\beta n^2}$ and the law of $(\widehat{X}_k,\widehat{Y}_k)_{0\leq k\leq\beta n^2}$ is smaller than $\varepsilon$. 
\end{proposition}
The proof requires many calculations, and is given in the appendix. With these three lemmas, we can prove Proposition \ref{CouplingTrees2}.

\begin{proof}[Proof of Proposition \ref{CouplingTrees2}]
    Fix $\varepsilon>0$ . By Proposition \ref{couplage processus 2}, we can choose $\alpha >0$ such that for $n$ large enough, there exists a coupling of $(X^{(n)}_{\infty,k},\widehat{X}^{(n)}_{\infty,k})_{0\leq k\leq\alpha n^2}$ and $(X_k^\infty,\widehat{X}_k^\infty)_{0\leq k\leq\alpha n^2}$ for which the processes are equal with probability at least $1-\varepsilon/3$. Then, we can choose $\beta>0$ small enough so that 
    \[\P(S_{\beta n}<\alpha n^2,\widehat{S}_{\beta n}<\alpha n^2)>1-\varepsilon/3\] 
    and 
    \[\P\left(\beta n^2<I_\infty^{(n)}<\left(a_\infty^{(n)}+b_\infty^{(n)}+1-\beta n^2\right)\right)>1-\varepsilon/3.\]
    Note that the existence of such $\beta$ is given by Propositions \ref{control 1} and \ref{convergence temps d'atteinte}. Then, consider the event $\mathcal{J}_{\beta n}$ which is the intersection of the events 
    \begin{itemize}[label=\textbullet]
        \item $\left\{(X^{(n)}_{\infty,k},\widehat{X}^{(n)}_{\infty,k})_{0\leq k\leq\alpha n^2}=(X_k^\infty,\widehat{X}_k^\infty)_{0\leq k\leq\alpha n^2}\right\}$
        \item $\left\{S_{\beta n}<\alpha n^2,\widehat{S}_{\beta n}<\alpha n^2\right\}$
        \item $\left\{\beta n^2<I_\infty^{(n)}<\left(a_\infty^{(n)}+b_\infty^{(n)}+1-\beta n^2\right)\right\}$.
    \end{itemize}
    For $n$ large enough and with our choice of $\beta$, we have $\P(\mathcal{J}_{\beta n})>1-\varepsilon$. Moreover, on this event, by comparing the construction of 
$\Theta_\infty^{(n)}$ with the construction of $\overline{\Theta}_\infty^{(1)}$ and $\overline{\Theta}_\infty^{(2)}$, we see that we can couple the first $\alpha n^2$ trees of the forests $\left(F^{(\infty,n)}_1,F^{(\infty,n)}_2,\widehat{F}^{(\infty,n)}_1,\widehat{F}^{(\infty,n)}_2\right)$ and $\left(F^{(\infty)}_1,F^{(\infty)}_2,\widehat{F}^{(\infty)}_1,\widehat{F}^{(\infty)}_2\right)$ so that they are equal. This gives a coupling between $\left(A_{\beta n}(\Theta_\infty^{(n)}),\widehat{A}_{\beta n}(\Theta_\infty^{(n)})\right)$ and $\left(A_{\beta n}(\overline{\Theta}^{(1)}_\infty),\widehat{A}_{\beta n}(\overline{\Theta}^{(2)}_\infty)\right)$ where the processes encoding these sets are equal with probability at least $1-\varepsilon$, which concludes the proof. 
\end{proof}

\section{Localization lemma}\label{Localisation}

In this section, we prove that vertices that are close to $\tau_n(n)$ in $Q_n$ (respectively $\tau_\infty(n)$ in $Q_\infty$ and $\overline{\rho}_\infty$ in $\overline{Q}_\infty$) must be close to $\tau_n(n)$ or $\widehat{\tau}_n(n-1)$ in $\Theta_n$ (respectively $\tau_\infty(n)$ or $\widehat{\tau}_\infty(n-1)$ in $Q_\infty$, and $\overline{\rho}_\infty^{(1)}$ or $\overline{\rho}_\infty^{(2)}$ in $\overline{Q}_\infty$), meaning that the balls of radius $r$ around $\tau_n(n)$ in $Q_n$ is contained in the union of the balls of radius $h(r)$ around $\tau_n(n)$ and $\widehat{\tau}_n(n-1)$ in $\Theta_n$. These results were already known in \cite{dieuleveut} when $r$ was fixed. However, for our purpose, we need a strengthened version, by allowing $r$ to depend on $n$. More precisely, we will prove the following result.

\begin{proposition}\label{localisation 2}
    For every $\varepsilon,\beta>0$, there exists $\alpha>0$ such that, for every $n$ large enough
    \begin{itemize}[label=\textbullet]
     \item\centering $\P(B_{Q_n}(\gamma_n(n),\alpha n)\subset A_{\beta n}(\Theta_n)\cup\widehat{A}_{\beta n}(\Theta_n))>1-\varepsilon.$
        \item\centering $\P(B_{Q_\infty}(\gamma_\infty(n),\alpha n)\subset A_{\beta n}(\Theta_\infty^{(n)})\cup\widehat{A}_{\beta n}(\Theta_\infty^{(n)}))>1-\varepsilon.$
        \item\centering $\P(B_{\overline{Q}_\infty}(\overline{\gamma}_\infty(0),\alpha n)\subset A_{\beta n}(\overline{\Theta}^{(1)}_\infty)\cup\widehat{A}_{\beta n}(\overline{\Theta}^{(2)}_\infty))>1-\varepsilon.$
    \end{itemize}
\end{proposition}
\begin{remark}
    Similar results concerning balls around the root of $Q_\infty$ were obtained in \cite{brownianplane}, as well as an analogous result in the continuum (see \cite[Proposition 4.1]{Bigeodesicbrownianplane}). 
\end{remark}
The proof consists of showing the result for $\overline{Q}_\infty$, and then to transfer it to $Q_n$ and $Q_\infty$ using the coupling results obtained in Section \ref{Coupling trees}. Indeed, using Propositions \ref{CouplingTrees1} and \ref{CouplingTrees2}, we will often couples the different trees involved and work on events such as $\{A_{\beta n}(\Theta_n)=A_{\beta n}(\Theta_\infty^{(n)})=A_{\beta n}(\overline{\Theta}_\infty^{(1)}\}$. The proof of Proposition \ref{localisation 2} is based on the following observation. The set $A_{\beta n}(\overline{\Theta}^{(1)}_\infty)$ is encoded by the Markov chain $X^{(\infty)}$ up to its last hitting time of $\lfloor\beta n\rfloor$, together with two forests $(F^\infty_1,F^\infty_2)$ of Bienaymé-Galton-Watson trees whose laws are as described in Section \ref{inf rerooted}. Therefore, according to Proposition \ref{decompo1}, the tree $A_{\beta n}(\overline{\Theta}^{(1)}_\infty)$ has the same law as the tree $\Theta_{\lfloor\beta n\rfloor}$. Moreover, this tree encodes an instance of a quadrangulation with geodesic boundaries \textit{inside} $\overline{Q}_\infty$, which will be key in the proof.

\subsection{Some estimates on the labels}\label{estimates}

First, we give a few results about the sets $A_{\beta n}(\Theta_n),A_{\beta n}(\Theta_\infty^{(n)})$ and $A_{\beta n}(\overline{\Theta}^{(1)}_\infty)$ introduced in Section \ref{Coupling trees}. These results will allow us to control the labels in each of these sets. Our first result shows that with large probability, the set $A_{\beta n}(\Theta_\infty^{(n)})$ (respectively $A_{\beta n}(\Theta_n)$) contains a macroscopic portion of the set $(\tau_\infty(k))_{k\geq 0}$ (respectively $(\tau_n(k))_{0\leq k \leq -\ell_*}$).

\begin{lemme}\label{Celia}
    For every $\varepsilon,\beta>0$, there exists $0<\alpha<1$ such that, for every $n$ large enough, 
    \[\P(\tau_\infty(\lfloor(1-\alpha)n\rfloor)\in A_{\beta n}(\Theta_\infty^{(n)}))>1-\varepsilon\quad\text{ and }\quad\P(\tau_\infty(\lfloor(1+\alpha)n\rfloor)\in A_{\beta n}(\Theta_\infty^{(n)}))>1-\varepsilon.\]
    Furthermore, we also have
    \[\P(\tau_n(\lfloor(1-\alpha)n\rfloor)\in A_{\beta n}(\Theta_n))>1-\varepsilon\quad\text{ and }\quad\P(\tau_n(\lfloor(1+\alpha)n\rfloor)\in A_{\beta n}(\Theta_n))>1-\varepsilon.\]
\end{lemme}
\begin{proof}
    We only give the proof for the first inequality, since the other ones can be proved in a very similar manner. By definition, we have $B_{\Theta_\infty^{(n)}}(\tau_\infty(n),S_{\lfloor\beta n\rfloor})\subset A_{\beta n}(\Theta_\infty^{(n)})$. Furthermore, by Lemma \ref{convergence temps d'atteinte}, there exists $c>0$ such that $\P(S_{\lfloor\beta n\rfloor}>cn^2)>1-\varepsilon/2$. Hence, it is enough to prove that 
    \[\P(d_{\Theta_\infty^{(n)}}(\tau_\infty(n),\tau_\infty({\lfloor(1-\alpha)n\rfloor}))<cn^2)>1-\varepsilon/2\] for some $0<\alpha<1$ and $n$ large enough. However, the convergence \eqref{Convergence Brownian} implies that 
    \[\P(d_{\Theta_\infty^{(n)}}(\tau_\infty(n),\tau_\infty({\lfloor(1-\alpha)n\rfloor}))<cn^2)\xrightarrow[n\rightarrow\infty]{}\P(d_\cT(T_1,T_{1-\alpha})<c).\] The result follows by taking $\alpha$ small enough, using the almost sure continuity of the process $(T_s)_{s\geq0}$ at $s=1$. 
\end{proof}
The following result is similar to Lemma \ref{Celia}, but deals with the tree $\overline{\Theta}^{(1)}_\infty$. However, contrary to the previous proofs, we cannot use the fact that we know the scaling limit of this tree (even though we expect it to be the tree $\overline{\cT}_\infty$, used in \cite{Bigeodesicbrownianplane} to construct the bigeodesic Brownian plane). Consequently, the proof is a bit more technical and relies on explicit computations.
\begin{lemme}\label{Celia2}
    For every $\varepsilon,\beta>0$, there exists $\alpha>0$ such that for every $n$ large enough, 
    \begin{equation*}
        \P\bigg(\overline{\tau}_\infty({\lfloor \alpha n \rfloor})\in A_{\beta n}(\overline{\Theta}^{(1)}_\infty)\bigg)>1-\varepsilon\quad\text{and}\quad \P\bigg(\overline{\tau}_\infty({\lfloor -\alpha n \rfloor})\in A_{\beta n}(\overline{\Theta}^{(1)}_\infty)\bigg)>1-\varepsilon.
    \end{equation*}
\end{lemme}
\begin{proof}
  We start with the first inequality. Recall that in $\overline{\Theta}^{(1)}$, conditionally on the labels on the spine $(X_n)_{n\geq0}$, the subtrees $(R_n)_{n\geq0}$ are independent labelled trees with laws $(\rho_{(X_n)}^+)_{n\geq0}$. Let $j_\delta$ be the index of the unique subtree such that $\overline{\tau}_\infty({\lfloor \alpha n \rfloor})\in R_{j_\delta}$. Set 
    \[H(i)=\frac{2}{(X_i+1)(X_i+2)}\quad\text{ and }\quad M(i,n)=\frac{2}{(X_i-\lfloor\alpha n\rfloor+1)(X_i-\lfloor\alpha n\rfloor+2)}.\]
Observe that for $i\geq S_{\beta n}$, we have
\[0<H(i)\leq M(i,n)\leq\frac{2}{((\beta-\alpha)n)^2}\]
    We have : 
    \begin{align*}
        \P(j_\delta<S_{\beta n})&=\P(\text{for every $i>S_{\beta n}$, $\inf_{u\in R_i}\ell_u>\lfloor \alpha n \rfloor$})\\
         &=\E\bigg[\prod_{i=S_{\beta n}}^\infty \frac{w(X_i-\lfloor\alpha n\rfloor)}{w(X_i)}\bigg]\\
        &=\E\bigg[\exp\bigg(\sum_{i=S_{\beta n}}^\infty (\log(w(X_i-\lfloor\alpha n\rfloor))-\log(w(X_i)))\bigg)\bigg]\\
        &=\E\bigg[\exp\bigg(\sum_{i=S_{\beta n}}^\infty(\log(1-M(i,n))-\log(1-H(i))\bigg)\bigg]\\
        &\geq \E\bigg[\exp\bigg(-C_n(\beta,\alpha)\sum_{i=S_{\beta n}}^\infty(M(i,n)-H(i))\bigg)\bigg]\\
        &=\E\bigg[\exp\bigg(-C_n(\beta,\alpha)\int_{S_{\beta n}}^\infty \bigg(M(\lfloor t\rfloor,n)-H(\lfloor t\rfloor)\bigg)dt\bigg)\bigg]\\
        &=\E\bigg[\exp\bigg(-n^2C_n(\beta,\alpha)\int_{S_{\beta n}/n^2}^\infty \bigg( M(\lfloor n^2t\rfloor,n)-H(\lfloor n^2t\rfloor)\bigg)dt\bigg)\bigg]
    \end{align*} 
    where $C_n(\beta,\alpha)=\bigg(1-\frac{2}{((\beta-\alpha)n)^2}\bigg)^{-1}$.
    Using the convergence of $\left(\frac{X_{\lfloor nt\rfloor}}{\sqrt{n}}\right)_{t\geq0}$ towards $(Z_{2t/3})_{t\geq0}$ and Lemma \ref{convergence temps d'atteinte}, we have
    \[\liminf_{n\rightarrow\infty}\P(j_\delta<S_{\beta n})\geq\E\bigg[\exp\bigg(-2\int_{\frac{3}{2}\cS_\beta}^\infty\bigg(\frac{1}{(Z_{2t/3}-\alpha)^2}-\frac{1}{(Z_{2t/3})^2}\bigg)dt\bigg)\bigg].\]
    Observe that the integral in the last display is finite, since there exists a constant $C_\alpha$ such that
    \[\frac{1}{(Z_{2t/3}-\alpha)^2}-\frac{1}{(Z_{2t/3})^2}\leq\frac{C_\alpha}{Z_{2t/3}^3},\]
    which is almost surely integrable. 
  Therefore, by dominated convergence and convergence of the last integral, we have 
    \begin{equation*}
        \lim_{\alpha\rightarrow0}\int_{\frac{3}{2}\cS_\beta}^\infty\bigg(\frac{1}{(Z_{2t/3}-\alpha)^2}-\frac{1}{(Z_{2t/3})^2}\bigg)dt=0\quad\text{ a.s. }
    \end{equation*}
    Hence, we can choose $\alpha>0$ small enough such that 
    \[\E\bigg[\exp\bigg(-2\int_{\frac{3}{2}\cS_\beta}^\infty\bigg(\frac{1}{(Z_{2t/3}-\alpha)^2}-\frac{1}{(Z_{2t/3})^2}\bigg)dt\bigg)\bigg]>1-\varepsilon/2.\]
    The first inequality follows for $n$ large enough. \\
    For the second inequality, we set
    \[K(i,n)=\frac{2}{(X_i+\lfloor\alpha n\rfloor+1)(X_i+\lfloor\alpha n\rfloor+2)}.\]
    Similarly to the above, we have 
    \begin{align}\label{equation intégrale}
        \P\bigg(\overline{\tau}_\infty({\lfloor -\alpha n \rfloor})\in A_{\beta n}(\overline{\Theta}^{(1)}_\infty)\bigg)&=1-\P\left(\forall u\in A_{\beta n}(\overline{\Theta}^{(1)}_\infty),\ell_u>\lfloor-\alpha n\rfloor\right)\nonumber\\
        &=1-\E\left[\prod_{i=0}^{S_{\beta n}} w(X_i+\lfloor \alpha n\rfloor)\right]\nonumber\\
        &=1-\E\left[\exp\left(\sum_{i=0}^{S_{\beta n}}\log(1-K(i,n))\right)\right]\nonumber\\
        &=1-\E\left[\exp\left(-\sum_{i=0}^{S_{\beta n}}K(i,n)+o(K(i,n))\right)\right]\nonumber\\
        &\xrightarrow[n\rightarrow\infty]{}1-\E\left[\exp\left(-2\int_0^{\frac{3}{2}\cS_\beta}\frac{1}{(Z_{2t/3}+\alpha)^2}dt\right)\right].
    \end{align}
    By monotone convergence and the scaling property of Bessel processes, the last displayed integral tends to $+\infty$ a.s. when $\alpha$ tends to 0. Therefore, we can choose $\alpha $ small enough so that
    \[\E\left[\exp\left(-2\int_0^{\frac{3}{2}\cS_\beta}\frac{1}{(Z_{2t/3}+\alpha)^2}dt\right)\right]<\varepsilon/2,\]
    and this proves the result for $n$ large enough.
\end{proof}
Finally, we give a control of the minimal label in each set of the form $A_{\beta n}(\Theta_n)$. 
\begin{lemme}\label{Control minimum}
    For every $\varepsilon,\delta>0$, there exists $\beta>0$ such that for every $n$ large enough, 
    \[\P\left(\inf_{u\in A_{\beta n}(\Theta_n)}\ell_u>\lfloor-\delta n\rfloor\right)>1-\varepsilon\quad\text{ and }\quad\P\left(\inf_{u\in\widehat{A}_{\beta n}(\Theta_n)}\ell_u>\lfloor-\delta n\rfloor\right)>1-\varepsilon.\]    
\end{lemme}
\begin{proof}
        We only check the first inequality, the second one being obtained in a similar manner. First, we prove the same inequality for $\overline{\Theta}_\infty^{(1)}$ instead of $\Theta_n$. In this case, the convergence \eqref{equation intégrale} shows that 
        \begin{align*}
            \P\left(\inf_{u\in A_{\beta n}(\overline{\Theta}_\infty^{(1)})}\ell_u>\lfloor-\delta n\rfloor\right)&=\P\bigg(\overline{\tau}_\infty({\lfloor -\delta n \rfloor})\notin A_{\beta n}(\overline{\Theta}^{(1)}_\infty)\bigg)\\
            &\xrightarrow[n\rightarrow\infty]{}\E\left[\exp\left(-2\int_0^{\frac{3}{2}\mathcal{S}_\beta}\frac{1}{(Z_{2t/3}+\delta)^2}dt\right)\right].      
        \end{align*}
       Then, by dominated convergence, we can choose $\beta>0$ small enough so that 
        \begin{equation}\label{Estimée integrale}
        \E\left[\exp\left(-2\int_0^{\frac{3}{2}\mathcal{S}_\beta}\frac{1}{(Z_{2t/3}+\delta)^2}dt\right)\right]>1-\varepsilon/2.
        \end{equation}
     Up to reducing $\beta$, we can use the coupling of $\Theta_n$ and $\overline{\Theta}_\infty^{(1)}$ of Proposition \ref{CouplingTrees1}, which gives 
        \begin{align*}
            \P\left(\inf_{u\in A_{\beta n}(\Theta_n)}\ell_u>\lfloor-\delta n\rfloor\right)&\geq\P\left(\left\{\inf_{u\in A_{\beta n}(\Theta_n)}\ell_u>\lfloor-\delta n\rfloor\right\}\cap\left\{ A_{\beta n}(\Theta_n)=A_{\beta n}(\overline{\Theta}_\infty^{(1)}\right\}\right)\\
            &=\P\left(\left\{\inf_{u\in A_{\beta n}(\overline{\Theta}_\infty^{(1)})}\ell_u>\lfloor-\delta n\rfloor\right\}\cap \left\{A_{\beta n}(\Theta_n)=A_{\beta n}(\overline{\Theta}_\infty^{(1)}\right\}\right)\\
            &\geq 1-\varepsilon,
        \end{align*}
        where the last line follows from \eqref{Estimée integrale}. This concludes the proof
    \end{proof}
    \begin{corollary}\label{Mininmum ailleurs}
        For every $\varepsilon>0$, there exists $\beta>0$ such that for every $n$ large enough, 
        \[\P\left(\inf_{u\in\Theta_n}\ell_u<\inf_{u\in A_{\beta n}(\Theta_n)\cup\widehat{A}_{\beta n}(\Theta_n)}\ell_u\right)>1-\varepsilon.\]
    \end{corollary}
    \begin{proof}
        First, for every $k>0$, using \eqref{minimum discret}, we have 
        \[\P\left(\inf_{u\in\Theta_n}\ell_u<-k\right)=\frac{(n+1)(n+2)}{(n+k+1)(n+k+2)}.\]
        Therefore, we can choose $\delta>0$ small enough so that for $n$ large enough,
        \[\P\left(\inf_{u\in\Theta_n}\ell_u<-\delta n\right)>1-\varepsilon/2.\]
        Then, by Lemma \ref{Control minimum}, we can choose $\beta>0$ small enough so that for $n$ large enough,
        \[\P\left(\inf_{u\in A_{\beta n}(\Theta_n)\cup\widehat{A}_{\beta n}(\Theta_n)}\ell_u>-\delta n\right)>1-\varepsilon/2.\]
        For these choices of $\delta$ and $\beta$, we have
        \begin{align*}
            \P\left(\inf_{u\in\Theta_n}\ell_u<\inf_{u\in A_{\beta n}(\Theta_n)\cup\widehat{A}_{\beta n}(\Theta_n)}\ell_u\right)&\geq\P\left(\inf_{u\in\Theta_n}\ell_u<-\delta n,\inf_{u\in A_{\beta n}(\Theta_n)\cup\widehat{A}_{\beta n}(\Theta_n)}\ell_u>-\delta n\right)\\
            &\geq 1-(\varepsilon/2+\varepsilon/2)=1-\varepsilon,
        \end{align*}
        which concludes the proof.
    \end{proof}
    
\subsection{Quadrangulations with geodesic boundaries}\label{geodesic boundaries}

Here, we give a construction and some basic results about quadrangulations with geodesic boundaries. This object, introduced in \cite{uniqueness}, played a major role in the proof of the convergence of random uniform quadrangulations toward the Brownian sphere.

Consider a labelled tree $\theta=(\mathbf{t},\ell)\in\mathbb{T}_f^{(0)}$, and set
\[\delta=-\min\{\ell(v),v\in\theta\}+1.\]
Then, we add a line of $\delta-1$ vertices to $\theta$ in the following way. First, if $k=k_\varnothing(\mathbf{t})$ stands for the number of children of the root in $\mathbf{t}$, set $\Tilde{v}_1=(k+1)$. Then, set $\tilde{v}_2=(k+1,1),\tilde{v}_3=(k+1,1,1),...$ and so on until $\tilde{v}_{\delta-1}=(k+1,1,...,1)$. We also set $\tilde{v}_0=\varnothing$ and $\tilde{v}_\delta=v_*$. Consider the labelled tree $\tilde{\theta}=(\tilde{\mathbf{t}},\tilde{\ell})$ where of $\tilde{\mathbf{t}}=\mathbf{t}\cup\{\tilde{v}_1,...,\tilde{v}_{\delta-1}\}$, and where $\tilde{\ell}(v)=\ell(v)$ if $v\in\mathbf{t}$ and $\tilde{\ell}(\tilde{v}_i)=-i$.

\begin{figure}
        \centering
        \includegraphics[scale=0.7]{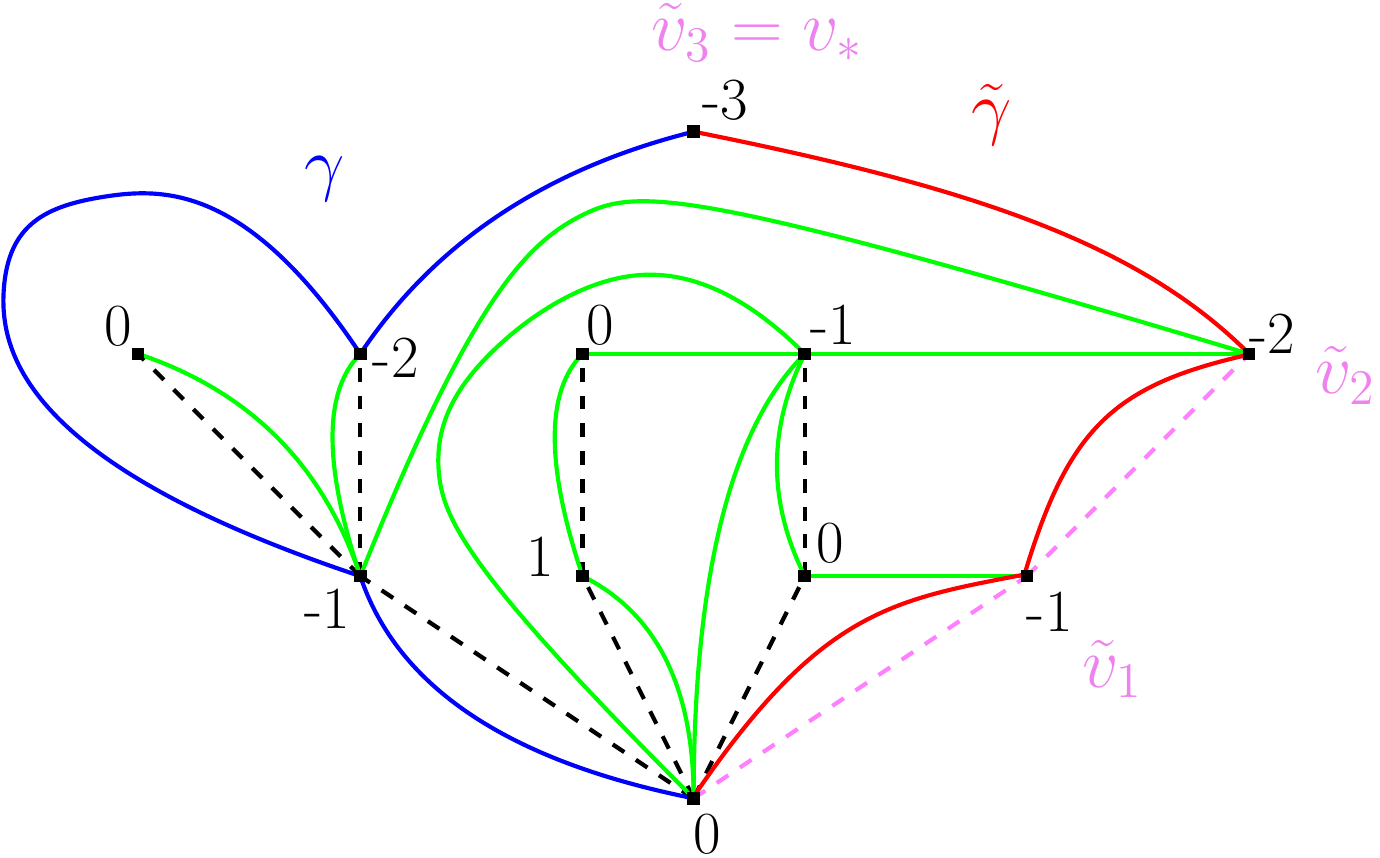}
        \caption{A representation of the CVS bijection for quadrangulations with geodesic boundaries. The underlying tree $\mathbf{t}$ is represented in black, while the elements added to obtain $\tilde{\mathbf{t}}$ are in pink.}
        \label{CVS Quad Bord}
\end{figure}

To construct a quadrangulation with geodesic boundaries, we apply a variant of the CVS bijection to $\tilde{\theta}$ (see Figure \ref{CVS Quad Bord}). First, for every corner $c$ of $\mathbf{t}$, we can define its successor $\sigma(c)$ \textbf{in $\tilde{\mathbf{t}}$}, and draw an edge between $c$ and $\sigma(c)$. Then, for every $i\in\{0,...,\delta-1\}$, draw an edge between $\tilde{v}_i$ and $\tilde{v}_{i+1}$. The resulting map $\tilde{Q}$ can be viewed as a quadrangulation with boundary. This boundary is formed of two paths $\gamma$ and $\tilde{\gamma}$, where $\gamma$ is the path from $c_0$ to its consecutive successors, and $\tilde{\gamma}$ is made of the edges $\left(\tilde{v}_i,\tilde{v}_{i+1}\right)_{0\leq i\leq\delta-1}$. 
   It is easy to verify that $\gamma$ and $\tilde{\gamma}$ are two geodesic paths between $\varnothing$ and $v_*$, only intersecting at their endpoints.

   Let $\tilde{Q}_n$ be the random quadrangulation with geodesic boundaries obtained by applying the previous construction to the random tree $\Theta_n$, and let $\gamma_n,\tilde{\gamma}_n$ be the two geodesics boundaries. Just like random quadrangulations, this random metric space admits a scaling limit $\left(\tilde{\cS},\tilde{D}\right)$ with two geodesic boundaries $\Gamma$ and $\tilde{\Gamma}$.
   \begin{theorem}\label{Convergergence slice}
    There exists a random metric space $\left(\tilde{\mathcal{S}},\tilde{D}\right)$,
    Such that we have the following convergence in distribution 
    \[\left(\tilde{Q}_n,\frac{\sqrt{3}}{\sqrt{2}n}d_{\tilde{Q}_n},\gamma_n,\tilde{\gamma}_n\right)\xrightarrow[n\rightarrow\infty]{(d)}\left(\tilde{\mathcal{S}},\tilde{D},\Gamma,\tilde{\Gamma}\right),\]
        for the marked Gromov-Hausdorff topology.
\end{theorem}
   This convergence was proved for uniform quadrangulations with geodesic boundaries with $n$ faces \cite[Proposition 9.2]{uniqueness}, but as we showed for the Brownian sphere in Theorem \ref{Convergergence sphere}, it also holds for $\tilde{Q}_n$ by excursion theory. Even though with the convergence result of \cite[Proposition 9.2]{uniqueness} does not take into account as a geodesic boundaries, the convergence of the geodesic boundaries follows of the encoding processes, see \cite[Section 4.5]{bettinelli2025compactbrowniansurfacesii} for a proof of this fact in a similar context.
   
   We will only use this scaling limit result to prove the following lemma, which states that $\gamma_n$ and $\tilde{\gamma}_n$ cannot be too close from each other away from their end point. 
   \begin{lemme}\label{distance bord}
   For every $\varepsilon>0$, there exists $\alpha,\delta>0$ such that for every $n\in\N$ large enough, 
   \begin{equation*}
       \P\left(d_{\tilde{Q}_n}(B_{\alpha n}(\Tilde{Q}_n,\tilde{\gamma}_n(n)),\gamma_n)>\delta n\right)>1-\varepsilon.
   \end{equation*}   
\begin{proof}
    This is just a consequence of a similar result for the Brownian slice (see \cite[Proposition 4.11]{Bigeodesicbrownianplane}), and of the convergence of quadrangulations with geodesic boundaries towards the Brownian slice \cite{uniqueness}.
\end{proof}
\end{lemme}

\subsection{End of the proof of Proposition \ref{localisation 2}}

This subsection is devoted to the proof of Proposition \ref{localisation 2}. To do so, our approach relies on Proposition \ref{sous ensemble}, together with the estimates obtained in Subsection \ref{estimates}.

For every $n>0$, define 
\[Z_*^{(n)}=\inf_{u\in A_n(\overline{\Theta}^{(1)}_\infty)}\ell_u.\] Let $\tilde{c}_n$ be the last corner associated to $S_n$, the last hitting time of $n$ by $X^{(\infty)}$, and for every $i\in [0,n-Z_*^{(n)}+1]$, set
\[\gamma_n^{(r)}(i)=\inf\{c\geq \tilde{c}_n,\,\ell_{v}=n-i\}.\]
In what follows, we keep the notation $\gamma_n^{(r)}(i)$ to denote the vertex associated to this corner.

Consider the map $\Tilde{Q}_n^{(1)}$ which is the restriction of $\overline{Q}_\infty$ to the vertex set
\[V_n=A_n(\overline{\Theta}^{(1)}_\infty)\cup \bigcup_{i\in [0,n-Z_*^{(n)}+1]}\gamma_n^{(r)}(i),\]
equipped with the intrinsic metric (see Figure \ref{plongement}).

\begin{figure}
        \centering
        \includegraphics[scale=0.5]{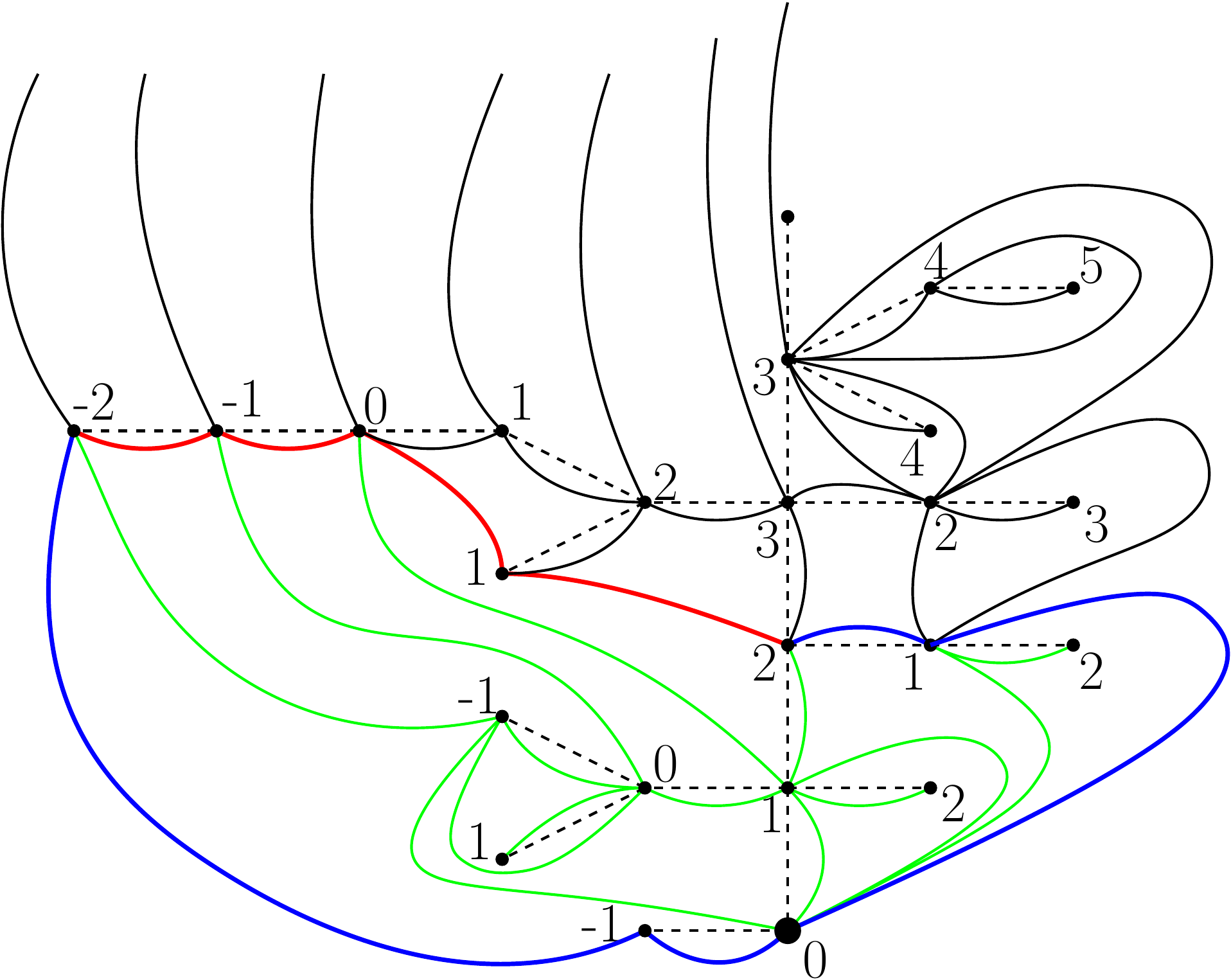}
        \caption{An exemple of $\tilde{Q}_2^{(1)}$ inside $\overline{Q}_\infty^{(1)}$. The edges composing $\tilde{Q}_2^{(1)}$ are colored in green, the blue path corresponds to $\gamma_2^{(\ell)}$ and the red one to $\gamma_2^{(r)}$.}
        \label{plongement}
\end{figure}

\begin{proposition}\label{sous ensemble}
   The quadrangulation $\tilde{Q}_n^{(1)}$ is distributed as $\tilde{Q}_n$.
\end{proposition}
\begin{proof}
    As mentioned after Proposition \ref{localisation 2}, the labelled tree $A_n(\overline{\Theta}^{(1)}_\infty)$ is distributed as $\Theta_n$. Then, by comparing the constructions given in Subsection \ref{CVS infinite} and \ref{geodesic boundaries}, one can see that the vertices $(\gamma_n^{(r)}(i))$ have the same role \textit{in the CVS bijection for $\overline{\Theta}_\infty^{(1)}$} as the line of vertices $(\tilde{v}_i)$ \textit{in the CVS bijection for quadrangulations with geodesic boundaries}. This gives the result.
\end{proof}
\begin{remark}\label{symetry}
    In a similar fashion, we can define $\tilde{Q}_n^{(2)}$, mutatis mutandis. The result of Proposition \ref{sous ensemble} still holds if we replace $\tilde{Q}_n^{(1)}$ by $\tilde{Q}_n^{(2)}$. We omit the proof, but the careful reader should bear in mind that the CVS bijection on $\overline{\Theta}_\infty^{(1)}$ and $\overline{\Theta}_\infty^{(2)}$ are not exactly the same.
\end{remark}
Next, for every $i\in[0,n-Z_*^{(n)}]$, we set 
\[\gamma_n^{(\ell)}(i)=\inf\{v\in A_n(\overline{\Theta}^{(1)}_\infty),\ell_v=n-i\}\] and 
\[\gamma_n^{(\ell)}(n-Z_*^{(n)}+1)=\gamma_n^{(r)}(n-Z_*^{(n)}+1).\]
These two paths correspond to the geodesics $\gamma$ and $\tilde{\gamma}$ of $\tilde{Q}_n^{(1)}$. Note that the bound \eqref{Borne} implies that, for every $i,j\in[0, n-Z_*^{(n)}+1]$,
\[d_{\overline{Q}_\infty}(\gamma_n^{(\ell)}(i),\gamma_n^{(\ell)}(j))=|i-j|\quad\text{ and }\quad d_{\overline{Q}_\infty}(\gamma_n^{(r)}(i),\gamma_n^{(r)}(j))=|i-j|.\]
Therefore, these two curves are geodesics paths in $\overline{Q}_\infty$. Observe that they are disjoint except at their endpoints. 
\begin{proposition}\label{Bord topologique}
    The paths $(\gamma_n^{(\ell)},\gamma_n^{(r)})$ are the boundaries of $\tilde{Q}_n^{(1)}$ in $\overline{Q}_\infty$, meaning that every path starting in $\tilde{Q}_n^{(1)}$ and ending in its complementary must intersect $\gamma_n^{(\ell)}\cup\gamma_n^{(r)}$.
\end{proposition}
\begin{proof}
 Fix $u\in V_n,v\notin V_n$ such that $d_{\overline{Q}_\infty}(u,v)=1$. Let $c_u$ and $c_v$ be the corresponding corners linked by an edge. There are two possibilities :
\begin{itemize}
    \item if $c_u$ is the successor of $c_v$ in the CVS bijection, then for every $c\in[c_v,c_u]$, $\ell_c>\ell_u$. This implies that 
    \[c_u=\inf\{c\in V_n,\ell_c=\ell_u\},\] which means that $u\in\gamma_n^{(\ell)}$.
    \item if $c_v$ is a successor of $c_u$ in the CVS bijection, then for every $c\in[c_u,c_v]\cap V_n$, $\ell_c>\ell_u$. We can easily see that the only elements of $V_n$ that satisfy this property are the ones in $\gamma_n^{(r)}$. 
\end{itemize}
This means that the boundary of $\tilde{Q}_n^{(1)}$ is included in $\gamma_n^{(\ell)}\cup\gamma_n^{(r)}$. Conversely, if $u\in\gamma_n^{(\ell)}\cup\gamma_n^{(r)}$, one can find an element $v\in\overline{Q}_\infty\backslash \tilde{Q}_n^{(1)}$ such that $d_{\overline{Q}_\infty}(u,v)=1$, which concludes the proof. 
\end{proof} 

Finally, the following Proposition allows us to compare distances between $Q_n,Q_\infty$ and $\overline{Q}_\infty$. In particular, it will be used to transfer some result obtained for one of these spaces to the other to ones. 

\begin{proposition}\label{comparaison distance1}
   On the event where \[(A_{\beta n}(\Theta_n),\widehat{A}_{\beta n}(\Theta_n))=(A_{\beta n}(\Theta_\infty^{(n)}),\widehat{A}_{\beta n}(\Theta_\infty^{(n)}))=(A_{\beta n}(\overline{\Theta}^{(1)}_\infty),\widehat{A}_{\beta n}(\overline{\Theta}^{(2)}_\infty))\] and $\left\{\inf_{u\in\Theta_n}\ell_u<\inf_{u\in A_{\beta n}(\Theta_n)\cup\widehat{A}_{\beta n}(\Theta_n)}\ell_u\right\}$, we have,  for every $u,v\in A_{\beta n}(\Theta_n)$ (or $u,v\in \widehat{A}_{\beta n}(\Theta_n)$),
    \[d_{Q_n}(u,v)=1\Longleftrightarrow d_{Q_\infty}(u,v)=1\Longleftrightarrow d_{\overline{Q}_\infty}(u,v)=1.\]
\end{proposition}
\begin{proof}
    Without loss of generality, suppose that $u$ and $v$ belong to $A_{\beta n}(\Theta_n)$ and satisfy $d_{Q_n}(u,v)=1$. One can see that the edge between $u$ and $v$ is an edge drawn in the CVS bijection from a corner $c$ towards a corner $c'$ such that $c<c'$. Indeed, for every corners $c'<c$ of $A_{\beta n}(\Theta_n)$, on the considered event, there exists a corner between $c$ and $c'$ with a label strictly less than $\ell_{c'}$, which implies that $c'$ cannot be the successor of $c$. Therefore, $c'$ is also the successor of $c$ in $\Theta_\infty^{(n)}$ and $\overline{\Theta}_\infty^{(1)}$, which means that $u$ and $v$ are also connected by an edge in $Q_\infty$ and $\overline{Q}_\infty$. Conversely, if $u$ and $v$ are connected by an edge in $Q_\infty$ or $\overline{Q}_\infty$, the same argument shows that they are also connected in $Q_n$, which proves the result.
\end{proof}

\begin{proof}[Proof of Proposition \ref{localisation 2}]
    Fix $\varepsilon,\beta>0$. We define $\mathcal{H}_{n,\alpha,\beta}$ as the intersection of the following events :
    \begin{itemize}[label=\textbullet]
        \item $(A_{\beta n}(\Theta_n),\widehat{A}_{\beta n}(\Theta_n))=(A_{\beta n}(\Theta_\infty^{(n)}),\widehat{A}_{\beta n}(\Theta_\infty^{(n)}))=(A_{\beta n}(\overline{\Theta}^{(1)}_\infty),\widehat{A}_{\beta n}(\overline{\Theta}^{(2)}_\infty))$
        \item $\inf_{u\in\Theta_n}\ell_u<\inf_{u\in A_{\beta n}(\Theta_n)\cup\widehat{A}_{\beta n}(\Theta_n)}\ell_u$
    \item $\tau_n({\lfloor(1-2\alpha)n\rfloor}),\tau_n(\lfloor(1+2\alpha)n\rfloor)\in A_{\beta n}(\Theta_n)$ and $\widehat{\tau}_n({\lfloor(1-2\alpha)n\rfloor}),\widehat{\tau}_n({\lfloor(1+2\alpha)n\rfloor})\in \widehat{A}_{\beta n}(\Theta_n)$
    \item $\tau_\infty({\lfloor(1-2\alpha)n\rfloor}),\tau_\infty(\lfloor(1+2\alpha)n\rfloor)\in A_{\beta n}(\Theta_\infty^{(n)})$ and $\widehat{\tau}_\infty({\lfloor(1-2\alpha)n\rfloor}),\widehat{\tau}_\infty({\lfloor(1+2\alpha)n\rfloor})\in \widehat{A}_{\beta n}(\Theta_\infty^{(n)})$
    \item $\overline{\tau}_\infty^{(1)}({\lfloor-2\alpha n\rfloor}),\overline{\tau}_\infty^{(1)}({\lfloor2\alpha n\rfloor})\in A_{\beta n}(\overline{\Theta}^{(1)}_\infty)$ and $\overline{\tau}_\infty^{(2)}({\lfloor-2\alpha n\rfloor}),\overline{\tau}_\infty^{(2)}({\lfloor2\alpha n\rfloor})\in \widehat{A}_{\beta n}(\overline{\Theta}^{(2)}_\infty)$
        \item ${d}_{\tilde{Q}_{\lfloor\beta n\rfloor}^{(1)}}\left(B_{\alpha n}(\Tilde{Q}_{\lfloor\beta n\rfloor}^{(1)},\gamma_{\lfloor\beta n\rfloor}^{(\ell)}(\beta n)),\gamma_{\lfloor\beta n\rfloor}^{(r)}\right)>\delta n$ and ${d}_{\tilde{Q}_{\lfloor\beta n\rfloor}^{(2)}}\left(B_{\alpha n}(\Tilde{Q}_{\lfloor\beta n\rfloor}^{(2)},\gamma_{\lfloor\beta n\rfloor}^{(\ell)}(\beta n)),\gamma_{\lfloor\beta n\rfloor}^{(r)}\right)>\delta n$.
    \end{itemize}
    By Lemmas \ref{Celia}, \ref{Celia2}, \ref{distance bord} and Propositions \ref{CouplingTrees1}, \ref{CouplingTrees2} we can choose $\alpha,\delta$ such that for every $n$ large enough, 
    \[\P(\mathcal{H}_{n,\alpha,\beta})>1-\varepsilon.\]
    We will show that 
    \[\mathcal{H}_{n,\alpha,\beta}\subset\left\{B_{\overline{Q}_\infty}(\overline{\gamma}_\infty(0),\alpha n)\subset A_{\beta n}(\overline{\Theta}_\infty^{(1)})\cup\widehat{A}_{\beta n}(\overline{\Theta}_\infty^{(2)})\right\}\]
    which will prove the result for $\overline{Q}_\infty$. We first fix $u\in \overline{\Theta}^{(1)}_\infty$ such that $u\notin A_{\beta n}(\overline{\Theta}_\infty^{(1)})$, and a geodesic $\eta$ between $u$ and $\overline{\gamma}_\infty(0)$ in $\overline{Q}_\infty$. Note that we can suppose that $\eta$ is only composed of edges between vertices of $\overline{\Theta}^{(1)}_\infty$. Indeed, if the path contains a vertex of $\overline{\Theta}^{(2)}_\infty$, it must cross the infinite bigeodesic $\overline{\gamma}_\infty$ at some vertex $w$. Therefore, we can replace the geodesic $\eta$ by a geodesic $\eta'$, which coincides with $\eta$ until it hits $w$, and then follows $\overline{\gamma}_\infty$ to reach $\overline{\gamma}_\infty(0)$.

   \begin{figure}
       \centering
       \includegraphics[scale=1]{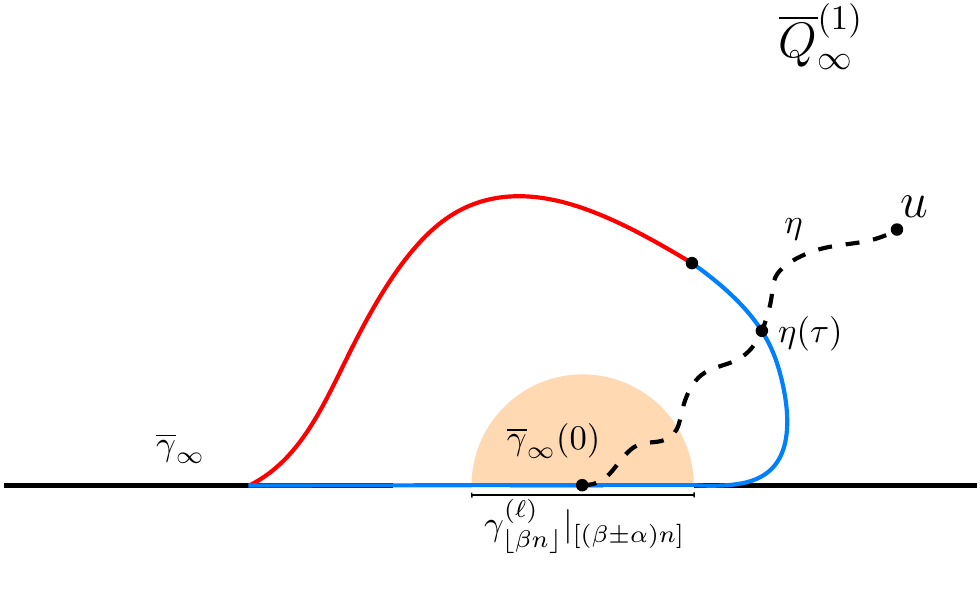}
       \caption{Illustration of the proof of Proposition \ref{localisation 2}. The region delimited by the blue and red paths correspond to $\Tilde{Q}_{\lfloor\beta n\rfloor}^{(2)}$, and the orange area represents $B_{\overline{Q}_{\infty}^{(1)}}(\overline{\gamma}_\infty(0),\alpha n)$. In order to apply Lemma \ref{distance bord}, we need to show that $\eta(\tau)\notin\gamma_{\lfloor\beta n\rfloor}^{(\ell)}|_{[\lfloor(\beta\pm\alpha) n\rfloor]}$.}
       \label{fig:loca}
   \end{figure} 
    Set 
    \[\tau=1+\sup\{k\geq0,\eta(k)\notin A_{\beta n}(\overline{\Theta}^{(1)}_\infty)\},\]
    which is well-defined because $u=\eta(0)\notin A_{\beta n}(\overline{\Theta}^{(1)}_\infty)$.Then, $\eta|_{\llbracket\tau,d_{\overline{Q}_\infty}(u,\overline{\gamma}_\infty(0))\rrbracket}$ is a geodesic between $\eta(\tau)$ and $\overline{\gamma}_\infty(0)$. Furthermore, by definition, it is only composed of edges that are also in $\Tilde{Q}_{\lfloor\beta n\rfloor}^{(1)}$. Therefore, the bound $d_{\overline{Q}_\infty}\leq d_{\Tilde{Q}_{\lfloor\beta n\rfloor}^{(1)}}$ implies that 
    \[d_{\overline{Q}_\infty}(\eta(\tau),\overline{\gamma}_\infty(0))=\Tilde{d}_{\tilde{Q}_{\lfloor\beta n\rfloor}^{(1)}}(\eta(\tau),\overline{\gamma}_\infty^{(1)}(0)).\]
    By Proposition \ref{Bord topologique}, we know that $\eta(\tau)\in\gamma_{\lfloor\beta n\rfloor}^{(\ell)}\cup\gamma_{\lfloor\beta n\rfloor}^{(r)}$. Let us prove that $\eta(\tau)\notin\gamma_{\lfloor\beta n\rfloor}^{(\ell)}|_{[\lfloor(\beta-\alpha) n\rfloor,\lfloor(\beta+\alpha) n\rfloor]}$.
    \begin{itemize}[label=\textbullet]
        \item if $|\ell_{\eta(\tau)}|>\alpha n$, $\eta(\tau)$ cannot be in $\gamma_{\lfloor\beta n\rfloor}^{(\ell)}|_{[\lfloor(\beta-\alpha) n\rfloor,\lfloor(\beta+\alpha) n\rfloor]}$, because every element $v$ of this set satisfies $|\ell_v|\leq\alpha n$. 
        \item suppose that $|\ell_{\eta(\tau)}|\leq\alpha n$. Note that if $\eta(\tau-1)$ cannot belong to the right part of the tree $\overline{\Theta}_\infty^{(1)}$. Indeed, since $\overline{\tau}_\infty^{(1)}\left({\lfloor2\alpha n\rfloor}\right)\in A_{\beta n}(\overline{\Theta}^{(1)}_\infty)$, we would have which means that then $\ell_{\eta(\tau)}\geq \ell_{\eta(\tau-1)}-1>2\alpha n-1$, which is not possible. Then, according to the CVS bijection, there exists two corners $c_0,c_1$ associated to $\eta(\tau),\eta(\tau-1)$ such that $c_0$ is the successor of $c_1$, or the other way around. Therefore, we either have $\ell_w\geq\ell_{\eta(\tau)}$ for every $w\in[c_0,c_1)$, or $\ell_w\geq\ell_{\eta(\tau)}$ for every $w\in[c_1,c_0)$. One the one hand, since the corner $c_1$ is on the left side of $\overline{\Theta}_\infty^{(1)}$, we have
        \[\inf_{w\in[c_1,c_0]}\ell_w=-\infty.\]
        On the other hand, observe that on $\mathcal{H}_{n,\alpha,\beta}$, for every corner $c$ associated to a vertex of $\gamma_{\lfloor\beta n\rfloor}^{(\ell)}|_{[\lfloor(\beta-\alpha) n\rfloor,\lfloor(\beta+\alpha) n\rfloor]}$, $[c,c_1]$ contains $\overline{\tau}_\infty^{(1)}\left({\lfloor-2\alpha n\rfloor}\right)\in A_{\beta n}(\overline{\Theta}^{(1)}_\infty)$. Therefore,
        \[\min_{w\in[c,c_1]}\ell_w\leq -2\alpha n.\]
        Since $|\ell_{\eta(\tau)}|\leq\alpha n$, this shows that $c_0$ is not associated to a vertex of $\gamma_{\lfloor\beta n\rfloor}^{(\ell)}|_{[\lfloor(\beta-\alpha) n\rfloor,\lfloor(\beta+\alpha) n\rfloor]}$, and proves the claim.
    \end{itemize}
    Finally, observe that on the event $\mathcal{H}_{n,\alpha,\beta}$, every element $v$ of $\gamma_{\lfloor\beta n\rfloor}^{(\ell)}\cup\gamma_{\lfloor\beta n\rfloor}^{(r)}$ that is not in $\gamma_{\lfloor\beta n\rfloor}^{(\ell)}|_{[\lfloor(\beta-\alpha) n\rfloor,\lfloor(\beta+\alpha) n\rfloor]}$ satisfies $d_{\tilde{Q}_{\lfloor\beta n\rfloor}^{(1)}}(\overline{\gamma}_\infty^{(1)}(0),v)\geq\alpha n$. Therefore, on $\mathcal{H}_n$, 
    \[d_{\overline{Q}_\infty}(u,\overline{\gamma}_\infty(0))\geq d_{\Tilde{Q}_{\lfloor\beta n\rfloor}^{(1)}}(\eta(\tau),\overline{\gamma}_\infty^{(1)}(0))\geq\alpha n,\]
    which proves the statement in this specific case. Note that the other case $u\in\overline{\Theta}^{(2)}_\infty\backslash \widehat{A}^\infty_{\beta n}(\overline{\Theta}^{(2)}_\infty)$ can be treated similarly by using $\Tilde{Q}_{\lfloor\beta n\rfloor}^{(2)}$, assuming without loss of generality that $\eta$ remains in $\overline{\Theta}^{(2)}_\infty\cup\gamma_\infty$.\\
    Let us use the result in this particular case to prove the statement for $Q_n$ and $Q_\infty$ (we only give the proof for $Q_n$, the ideas being similar for $Q_\infty$). As we did with $\left(A_{\beta n}(\overline{\Theta}^{(1)}_\infty),\widehat{A}_{\beta n}(\overline{\Theta}^{(2)}_\infty)\right)$, we can associate with $\left(A_{\beta n}(\Theta_n),\widehat{A}_{\beta n}(\Theta_n)\right)$ two quadrangulations with geodesics boundaries $\left(\widehat{Q}_{\lfloor\beta n\rfloor}^{(1)},\widehat{Q}_{\lfloor\beta n\rfloor}^{(2)}\right)$. Moreover, on $\mathcal{H}_{n,\alpha,\beta}$, Proposition \ref{comparaison distance1} implies that we have $\left(\widehat{Q}_{\lfloor\beta n\rfloor}^{(1)},\widehat{Q}_{\lfloor\beta n\rfloor}^{(2)}\right)=\left(\tilde{Q}_{\lfloor\beta n\rfloor}^{(1)},\tilde{Q}_{\lfloor\beta n\rfloor}^{(2)}\right)$. Fix $u\in Q_n\backslash(A_{\beta n}(\Theta_n)\cup\widehat{A}_{\beta n}(\Theta_n))$, and a geodesic $\eta$ between $u$ and $\gamma_n(n)$. We also define $\tau$ as we did previously, mutatis mutandis. Moreover, one can assume without loss of generality that $\eta|_{[\tau,d_{Q_n}(u,\gamma_n(n))]}$ has all its edges that are either in $\widehat{Q}_{\lfloor\beta n\rfloor}^{(1)}$, or in $\widehat{Q}_{\lfloor\beta n\rfloor}^{(2)}$. Without loss of generality, suppose that are in the first case. We have
    \[d_{Q_n}(\eta(\tau),\gamma_n(n))=d_{\widehat{Q}_{\lfloor\beta n\rfloor}^{(1)}}(\eta(\tau),\gamma_n(n))=d_{\tilde{Q}_{\lfloor\beta n\rfloor}^{(1)}}(\eta(\tau),\overline{\gamma}_\infty^{(1)}(0)).\]
    Then, we can conclude the proof with the same arguments used to prove the result for $\overline{Q}_\infty$, which gives the result. 
\end{proof}

\section{Proof of the main results}\label{sec main result}

In this section, we prove Theorem \ref{Main result}, Theorem \ref{theorem 2} and Corollary \ref{coro convergence}. To do so, we rely on Propositions \ref{CouplingTrees1}, \ref{CouplingTrees2} and \ref{localisation 2} to prove that we can couple our spaces so that macroscopic balls around $\gamma_n(n),\gamma_\infty(n)$ and $\overline{\gamma}_\infty(0)$ are isometric. Then, we use the method introduced in the proof of \cite[Theorem 2]{brownianplane} (together with Theorem \ref{convergergenceBP}) to prove the convergence.

Let $\mathcal{G}_{n,\alpha,\beta}$ be the intersection of the following events :
\begin{itemize}[label=\textbullet]
\item $\mathcal{H}_{n,\alpha,\beta}$ (introduced in the proof of Proposition \ref{localisation 2})
\item $\left\{B_{Q_n}(\gamma_n(n),2\alpha n)\subset A_{\beta n}(\Theta_n)\cup\widehat{A}_{\beta n}(\Theta_n)\right\}$
\item $\left\{B_{Q_\infty}(\gamma_\infty(n),2\alpha n)\subset A_{\beta n}(\Theta_\infty^{(n)})\cup\widehat{A}_{\beta n}(\Theta_\infty^{(n)})\right\}$
\item $\left\{B_{\overline{Q}_\infty}(\overline{\gamma}_\infty(0),2\alpha n)\subset A_{\beta n}(\overline{\Theta}_\infty^{(1)})\cup\widehat{A}_{\beta n}(\overline{\Theta}^{(2)}_\infty)\right\}.$
\end{itemize}  
On this event, we will be able to compare the balls $B_{Q_n}(\gamma_n(n),\alpha n), B_{Q_\infty}(\gamma_\infty(n),\alpha n)$ and $B_{\overline{Q}_\infty}(\overline{\gamma}_\infty(0),\alpha n)$. By Proposition \ref{localisation 2}, for every $\varepsilon,\beta>0$, we can choose $\alpha>0$ such that for every $n$ large enough, 
\[\P(\mathcal{G}_{n,\alpha,\beta})>1-\varepsilon.\]
The following result completes Proposition \ref{comparaison distance1}, and enables us to compare distances between all vertices in $B_{Q_n}(\gamma_n(n),2\alpha n),B_{Q_\infty}(\gamma_\infty(n),2\alpha n)$ and $B_{\overline{Q}_\infty}(\overline{\gamma}_\infty(0),2\alpha n)$.
\begin{proposition}\label{comparaison distance2}
    On $\mathcal{G}_{n,\alpha,\beta}$, for every $u\in A_{\beta n}(\Theta_n)$ and $v\in\widehat{A}_{\beta n}(\Theta_n)$, such that $u,v\in B_{Q_n}(\gamma_n(n),2\alpha n)$, we have
    \[d_{Q_n}(u,v)=1\Longleftrightarrow d_{Q_\infty}(u,v)=1\Longleftrightarrow d_{\overline{Q}_\infty}(u,v)=1.\]
\end{proposition}
\begin{proof}
    Note that the requirement $u,v\in B_{Q_n}(\gamma_n(n),2\alpha n)$ together with the bound \eqref{Borne} implies that 
    \[|\ell_u|\leq2\alpha n\quad\text{ and }\quad|\ell_v|\leq2\alpha n.\]
    Let $c_u$ and $c_v$ be the corners associated to $u$ and $v$ such that the edge between $u$ and $v$ corresponds to an edge drawn between $c_u$ and $c_v$ in the CVS bijection. As in the proof of Proposition \ref{comparaison distance1}, our hypotheses imply that $c_u$ must be the successor of $c_v$. This means that $\ell_u=\ell_v-1$ and that every corner $c$ encountered in the cyclic order between $c_v$ and $c_u$ must satisfy $\ell_c\geq\ell_v$. Observe that this is trivially satisfied for any corner between $\widehat{\tau}(\lfloor(1-2\alpha)n\rfloor)$ and $\tau(\lfloor(1-2\alpha)n\rfloor)$, because any such corner $c$ has $\ell_c\geq\lfloor(1-2\alpha)n\rfloor+1$. In particular, this means that $$c_u,c_v\notin\left[\widehat{\tau}(\lfloor(1-2\alpha)n\rfloor),\tau(\lfloor(1-2\alpha)n\rfloor)\right].$$ But this implies that $c_u$ is also the successor of $c_v$ in $\Theta_\infty^{(n)}$. Indeed, the requirement $\tau(\lfloor(1-2\alpha)n\rfloor),\tau(\lfloor(1+2\alpha)n\rfloor)\in A_{\beta n}(\Theta_\infty^{(n)})$ and $\widehat{\tau}(\lfloor(1-2\alpha)n\rfloor),\widehat{\tau}(\lfloor(1+2\alpha)n\rfloor)\in \widehat{A}_{\beta n}(\Theta_\infty^{(n)})$ ensures that $\ell_c\geq\ell_v$ for every corner between $c_v$ and $c_u$ that is not in $A_{\beta n}(\Theta_\infty^{(n)})\cup\widehat{A}_{\beta n}(\Theta_\infty^{(n)})$, and this also holds for the other corners in $A_{\beta n}(\Theta_\infty^{(n)})\cup\widehat{A}_{\beta n}(\Theta_\infty^{(n)})$ because of the equality $$\left(A_{\beta n}(\Theta_n),\widehat{A}_{\beta n}(\Theta_n)\right)=\left(A_{\beta n}(\Theta_\infty^{(n)}),\widehat{A}_{\beta n}(\Theta_\infty^{(n)})\right).$$ Therefore, $d_{Q_\infty}(u,v)=1$. Similarly, since $\overline{\tau}^{(2)}_\infty(\lfloor\alpha n\rfloor)\in\widehat{A}_{\beta n}(\overline{\Theta}_\infty^{(2)})$ and because there is no corner $c$ in $\widehat{A}_{\beta n}(\Theta_n)$ after $c_v$ such that $\ell_c<\ell_v$, the construction of $\overline{Q}_\infty$ implies that there is an edge between $c_v$ and the first corner of $\overline{\Theta}^{(1)}_\infty$ with label $\ell_v-1$, which is $c_u$. Therefore, we also have $d_{\overline{Q}_\infty}(u,v)=1$. Conversely, the same idea (and the fact that we work on the event $\mathcal{G}_{n,\alpha,\beta}$) shows that if there is an edge between $u$ and $v$ in $Q_\infty$ or $\overline{Q}_\infty$, this edge is also in $Q_n$, which concludes the proof. 
\end{proof} 
\begin{remark}
In fact, we proved a bit more than the statement of Proposition \ref{comparaison distance2} : we showed that on the considered event, the balls have the same set of vertices and edges. 
\end{remark}
Now, we can prove Theorem \ref{theorem 2} and Corollary \ref{coro convergence}, which are direct consequences of Propositions \ref{CouplingTrees1}, \ref{localisation 2} and \ref{comparaison distance2}. The following proposition gathers these two results, and is an enhancement of the main theorem of \cite{dieuleveut}. 
\begin{proposition}\label{couplage boules}
    For every $\varepsilon>0$, there exists $\alpha>0$ such that for every $n$ large enough, we can couple $Q_n,Q_\infty$ and $\overline{Q}_\infty$ in such way that the balls $B_{Q_n}(\gamma_n(n),\alpha n), B_{Q_\infty}(\gamma_\infty(n),\alpha n)$ and $ B_{\overline{Q}_\infty}(\overline{\gamma}_\infty(0),\alpha n)$ are isometric with probability at least $1-\varepsilon$. Moreover, we have the following convergence in distribution :
    \[\left(Q_n,\gamma_n(n)\right)\xrightarrow[n\rightarrow\infty]{(d)}\left(\overline{Q}_\infty,\overline{\gamma}_\infty(0)\right)\]
    for the local topology.
\end{proposition}
\begin{proof}[Proof of Proposition \ref{couplage boules}]
  First, we prove the coupling result. Fix $\varepsilon>0$ and $\beta>0$, and choose $\alpha>0$ such that for $n$ large enough, $\P(\mathcal{G}_{n,\alpha,\beta})>1-\varepsilon$. We will show that on $\mathcal{G}_{n,\alpha,\beta}$, for every $u,v\in A_{\beta n}(\Theta_n)\cup\widehat{A}_{\beta n}(\Theta_n)$ such that $u,v\in B_{Q_n}(\gamma_n(n),\alpha n)$, we have 
    \[d_{Q_n}(u,v)=d_{Q_\infty}(u,v)=d_{\overline{Q}_\infty}(u,v).\]
    
    Fix $u,v$ as described previously, and consider a geodesic $\eta$ between $u$ and $v$ (in $Q_n$). First, note that on the considered event, every vertex encountered by $\eta$ is in $A_{\beta n}(\Theta_n)\cup\widehat{A}_{\beta n}(\Theta_n)$. 
    Therefore, by Propositions \ref{comparaison distance1} and \ref{comparaison distance2}, every edge of $\eta$ is also in $Q_\infty$ and $\overline{Q}_\infty$, which implies that 
    \[d_{Q_n}(u,v)\geq d_{Q_\infty}(u,v)\quad\text{ and }\quad d_{Q_n}(u,v)\geq d_{\overline{Q}_\infty}(u,v).\] However, by symmetry, the converse bounds also hold, which gives the equality, and show that the balls are isometric. Finally, the convergence statement of the proposition is a direct consequence of the coupling result together with the definition of the local topology, which concludes the proof.   
\end{proof}

We can finally prove the main result of this paper.

\begin{proof}[Proof of Theorem \ref{Main result}]
    We proceed as in \cite[Theorem 2]{brownianplane}. Let $(k_n)_{n\geq1}$ be a sequence of non-negative real numbers converging to $\infty$, such that $k_n=o(n)$ as $n\rightarrow\infty$. We will prove that : 
    \begin{align*}
        B_r\big(k_n^{-1}\cdot Q_n,\gamma_n(n)\big)\xrightarrow[n\rightarrow\infty]{(d)}B_r(\overline{\mathcal{BP}})\\
        B_r\big(k_n^{-1}\cdot Q_\infty,\gamma_\infty(n)\big)\xrightarrow[n\rightarrow\infty]{(d)}B_r(\overline{\mathcal{BP}})\\
        B_r\big(k_n^{-1}\cdot \overline{Q}_\infty,\overline{\gamma}_\infty(0)\big)\xrightarrow[n\rightarrow\infty]{(d)}B_r(\overline{\mathcal{BP}}).
        \end{align*}
        
        To simplify notations, we fix $r=1$, and we let $\varepsilon>0$. Then, by Theorem \ref{convergergenceBP}, there exists $\lambda_0>0$ such that for every $\lambda>\lambda_0$, we can construct $\overline{\mathcal{BP}}$ and $\cS$ (under $\N_0(\cdot\,|\,W_*<-1)$) in such a way that the equality 
        \begin{equation}\label{couplage BP}
        B_1(\lambda\cdot\cS,\Gamma_1)=B_1(\overline{\mathcal{BP}})
        \end{equation} 
        holds with probability at least $1-\varepsilon$. Then, choose $\alpha>0$ such that the result of Theorem
\ref{theorem 2} holds with probability at least $1-\varepsilon$. We can suppose that $\alpha<(2\lambda_0)^{-1}$ and that $k_n<\alpha n$ for every $n\in\N$. We define $m_n=\alpha^{-1}k_n$. Note that for every $n\in\N$, $m_n\leq n$ and $k_n\leq\alpha m_n$.

        Hence, for $n$ large enough, we can construct $Q_n,\,Q_{m_n},\,Q_\infty$ and $\overline{Q}_{\infty}$ on the same probability space such that the equality 
        \begin{equation}\label{couplage quad}
            B_{k_n}(Q_n,\gamma_n(n))=B_{k_n}(Q_{m_{n}},\gamma_{m_n}(m_n))=B_{k_n}(Q_\infty,\gamma_\infty(n))=B_{k_n}(\overline{Q}_\infty,\overline{\gamma}_\infty(0))
            \end{equation}
        holds with probability at least $1-\varepsilon$.\\
        By Theorem \ref{Convergergence sphere}, we have : 
        \[\left(V(Q_{m_n}),\,(\alpha^{-1}k_n)^{-1}d_{gr},\rho_{(m_n)}\right)\xrightarrow[n\rightarrow\infty]{(d)}\bigg(\cS,\bigg(\frac{8}{9}\bigg)^{1/4}D,\rho\bigg).\]
        In particular, this implies that 
        \begin{equation}\label{convergence boules}
        B_1(k_n^{-1}\cdot Q_{m_n},\gamma_{m_n}(m_n))\xrightarrow[n\rightarrow\infty]{(d)}B_1(\lambda\cdot\cS,\Gamma_1)
        \end{equation}
        where $\lambda=\bigg(\frac{8}{9}\bigg)^{1/4}\alpha^{-1}$. Note that since $\lambda>\lambda_0$, for every bounded continuous function $F:\mathbb{K}\longrightarrow\R$, we have : 
        \begin{align*}
            \E\big[|F(B_1(k_n^{-1}\cdot Q_{n},\gamma_n(n)))-&F(B_1(\overline{\mathcal{BP}}))|\big]\leq\\
            &\,\,\,\,\,\E\left[\left|F(B_1(k_n^{-1}\cdot Q_{n},\gamma_n(n)))-F(B_1(k_n^{-1}\cdot Q_{m_n},\gamma_{m_n}(m_n)))\right|\right]\\
            &+\E\left[\left|F(B_1(k_n^{-1}\cdot Q_{m_n},\gamma_{m_n}(m_n)))-F(B_1(\lambda\cdot\cS,\Gamma_b))\right|\right]\\
            &+\E\left[\left|F(B_1(\lambda\cdot\cS,\Gamma_b))-F(B_1(\overline{\mathcal{BP}}))\right|\right].
        \end{align*}
By \eqref{couplage BP} and \eqref{couplage quad}, the first and third term of the right-hand side are bounded above by $\varepsilon\sup F$. Similarly, by \eqref{convergence boules}, the second term converges to $0$ as $n\rightarrow\infty$, which concludes the proof.        
\end{proof}

\appendix
\section{Proof of Theorem \ref{Convergergence sphere}}\label{appendix0}

This appendix is devoted to the proof of Theorem \ref{Convergergence sphere}.

\begin{proof}[Proof of Theorem \ref{Convergergence sphere}]
    Consider $(X,L)=(X_n,L_n)_{n\in\N}$, where $X$ is a simple random walk, and $L$ is the label process of the infinite forest encoded by $X$. It is well known that each excursion of $(X,L)$ above the running infimum of $X$ encodes a random labelled tree whose law is $\rho_{(0)}$. Consequently, according to what was said in Section \ref{sec cvs}, each excursion encodes a critical positive pointed Boltzmann quadrangulation.
    
    Let us introduce some notations. For $k\in \N$, set
    \[t_k=\inf\{n\geq 0, X_n=-k\}.\]
    For $i\in\N$, let $(X^{(i,j,n)},L^{(i,j,n)})_{1\leq j\leq n}$ denote the excursions of $X$ above its running infimum between $t_{in^2}$ and $t_{(i+1)n^2}$, indexed by decreasing order of their duration $\sigma^{(i,j,n)}$. We also shift these excursions so that they start at $0$. Then, we denote by $\left(Q^{(i,j,n)},d_{Q^{(i,j,n)}},v_0^{(i,j,n)},v_*^{(i,j,n)}\right)_{i,j,n\in\N}$ be the quadrangulations associated to these excursions via the CVS bijection.
By \cite[Theorem 3]{ConvergenceSerpent}, we have : 
    \begin{equation}\label{Convergence Brownian1}
        \bigg(\bigg(\frac{X_{\lfloor n^4t \rfloor}}{n^{2}}\bigg)_{t\geq0},\,\bigg(\frac{L_{\lfloor n^4t \rfloor}}{n}\bigg)_{t\geq0}\bigg)\xrightarrow[n\rightarrow\infty]{(d)}\bigg(\mathcal{B}_t,\mathcal{Z}_t\bigg)_{t\geq0},
    \end{equation} 
   where $\mathcal{B}$ is a standard Brownian motion and $\mathcal{Z}$ the Brownian snake driven by $\mathcal{B}$. In what follows, we denote by $T_i$ the first hitting time of $-i$ by $\mathcal{B}$. By Skorokhod representation theorem, we can suppose that this convergence holds almost surely. Then, for every $i\in\N$, let $(\mathcal{B}^{(i,j)},\mathcal{Z}^{(i,j)})_{j\in\N}$ be the excursions of $(\mathcal{B},\mathcal{Z})$ above the running infimum of $\mathcal{B}$ between $T_i$ and $T_{i+1}$, indexed by decreasing order of their duration $\sigma^{(i,j)}$. Since we assumed that the convergence \eqref{Convergence Brownian1} holds almost surely, we also have 
   \[\left(\frac{1}{n^2}X^{(i,j,n)}_{\lfloor n^4\cdot\rfloor},\frac{1}{n}L^{(i,j,n)}_{\lfloor n^4\cdot\rfloor},\frac{1}{n^4}\sigma^{(i,j,n)}\right)_{i\in\N,1\leq j\leq n}\xrightarrow[n\rightarrow\infty]{a.s.}\left(\mathcal{B}^{(i,j)},\mathcal{Z}^{(i,j)},\sigma^{(i,j)}\right)_{i,j\in\N}.\]
Observe that, conditionally on $\sigma^{(i,j,n)}$, the random variables $\left(Q^{(i,j,n)},d_{Q^{(i,j,n)}},v_0^{(i,j,n)},v_*^{(i,j,n)}\right)$ is uniform among the set of pointed quadrangulations with $\sigma^{(i,j,n)}$ faces. Therefore, by the main result of \cite{uniqueness,convergence}, for every $i,j\in\N$ and $n$ large enough, we have 
\begin{multline}\label{Convergence excursion}
\left(\frac{1}{n^2}X^{(i,j,n)}_{\lfloor n^4\cdot\rfloor},\frac{1}{n}L^{(i,j,n)}_{\lfloor n^4\cdot\rfloor},\left(Q^{(i,j,n)},\frac{\sqrt{3}}{\sqrt{2}n }d_{Q^{(i,j,n)}},v_0^{(i,j,n)},v_*^{(i,j,n)}\right)\right)\\
\xrightarrow[n\rightarrow\infty]{(d)}\left(\mathcal{B}^{(i,j)},\mathcal{Z}^{(i,j)},\left(\mathcal{S}^{(i,j)},D^{(i,j)},x_0^{(i,j)},x_*^{(i,j)}\right)\right),
\end{multline}
where $\left(\mathcal{S}^{(i,j)},D^{(i,j)},x_0^{(i,j)},x_*^{(i,j)}\right)$ is the Brownian sphere with random volume $\sigma^{(i,j)}$, and equal to $\Psi\left(\cB^{(i,j)},\cZ^{(i,j)}\right)$. We mention that even though we are dealing with positive quadrangulations, the results of \cite{uniqueness,convergence} still hold.

In what follows, we say that $(i,j)<(i',j')$ if $i<i'$, or if $i=i'$, and $j<j'$. Now, for every $n\in\N$, define 
\[(i_n,j_n)=\inf\left\{(i,j),\inf L^{(i,j,n)}\leq-n\right\}.\]
The key observation is that, by definition, $\left(Q^{(i_n,j_n,n)},d_{Q^{(i_n,j_n,n)}},v_0^{(i_n,j_n,n)},v_*^{(i_n,j_n,n)}\right)$ is distributed as the random quadrangulation $Q_n$ introduced in Section \ref{random maps}. Since we assumed that the convergence \eqref{Convergence Brownian1} holds almost surely, there exists a random index $(\tilde{i},\tilde{j})$ such that 
\[(i_n,j_n)\xrightarrow[n\rightarrow\infty]{a.s.}(\tilde{i},\tilde{j}).\]
Then, since $\N_0$ is the excursion measure of the Brownian snake, it follows that \[\left(\frac{1}{n^2}X^{(i_n,j_n,n)}_{\lfloor n^4\cdot\rfloor},\frac{1}{n}L^{(i_n,j_n,n)}_{\lfloor n^4\cdot\rfloor}\right)\xrightarrow[n\rightarrow\infty]{a.s.} \left(\mathcal{B}^{(\tilde{i},\tilde{j})},\mathcal{L}^{(\tilde{i},\tilde{j})}\right),\] and that the latter is distributed under the probability measure $\N_0(\cdot\,|\,W_*<-1)$. Together with \eqref{Convergence excursion} and since $Q^{(i_n,j_n,n)}$ is distributed as $Q_n$, this yields the result. 
\end{proof}

\section{Proof of Lemma \ref{couplage processus 2}}\label{appendix}

This appendix is devoted to the proof of Lemma \ref{couplage processus 2}. This result is an enhanced version of \cite[Lemma 3.3]{dieuleveut}, and the idea of the proof is similar.

First, we introduce some notation. Recall that $(X_n)_{n\geq0}$ is the Markov chain with transition probabilities given by \eqref{Transition}. For every $x,k\in\N$, we define the Green function
\[H_x(k)=\sum_{n\geq1}\P_x(X_{n-1}=k)\]
and 
\[H_x^*(k)=\sum_{n\geq1}(n+1)\P_x(X_{n-1}=k).\]
These quantities will appear in the following proofs, and were computed in \cite{dieuleveut}. 
\begin{lemme}
    Fix $k\geq2,x\in\N$. We have 
    \begin{itemize}[label=\textbullet]
        \item If $x\leq k$, 
        \[H_x(k)=\frac{3}{10}(2k+3)\]
        \item If $x>k$, 
        \[H_x(k)=\frac{3h(k)}{10h(x)}(2k+3)\]
    \end{itemize} where $h$ was defined in \eqref{Probabilité de retour}. 
\end{lemme}
The value of $H_x^*(k)$ was also computed in \cite[Lemma 2.4]{dieuleveut}. However, the formula given for $H_x^*(k)$ in the case $x>k$ appears to be false (it tends to $-\infty$ as $x\rightarrow\infty$, but it is a non-negative quantity). The following lemma gives the correct formula.
\begin{lemme}\label{ERREUR 404}
    Fix $k\geq2,x\in\N$, and let $C_x=\frac{3}{14}((x+1)(x+2)-6)$. Then, we have 
    \begin{itemize}
        \item If $x\leq k$, 
        \[H_x^*(k)=\frac{3f(k)}{f(1)w(k)}-\frac{3C_x}{10}(2k+3)\]
        where f was defined in \eqref{Transition}.
        \item If $x> k$,
        \[H_x^*(k)=\frac{h(k)}{h(x)}\left(H_k^*(k)+\frac{3A_{x,k}}{10}(2k+3)\right)\] where 
\[A_{x,k}=(k+2)+\frac{(x-k-1)(x+k+4)}{2}\]
and $h(x)=x(x+1)(x+2)(x+3)(2x+3)$.
    \end{itemize}
\end{lemme}
\begin{proof}
To lighten notation, we set $p_x:=p(x,x+1),\,r_x:=p(x,x)$ and $q_x:=p(x,x-1)$, where $p(i,j)$ are the transition probabilities introduced in Section \ref{Random labeled trees}. From \cite[Section 2.3]{dieuleveut}, we have the following recurrence relation :
    \begin{equation}\label{Récurrence}
    H^*_{x+1}(k)=\frac{1}{p_x}\left((1-r_x)H_x^*(k)-q_x H^*_{x-1}(k)-H_x(k)-\1_{\{x=k\}}\right)
\end{equation}
with $H^*_1(k)=\frac{3f(k)}{f(1)w(k)}$. The formula given in the lemma satisfies this relation if and only if 
\begin{align*}
   &H^*_k(k)\left(\frac{h(k)}{h(x+1)}-\frac{h(k)}{h(x)}-\frac{q_x}{p_x}\left(\frac{h(k)}{h(x)}-\frac{h(k)}{h(x-1)}\right)\right)+\frac{3A_{x+1,k}h(k)}{10h(x+1)}(2k+3)\\
   &=\frac{1}{p_x}\left((1-r_x)\frac{3A_{x,k}h(k)}{10h(x)}(2k+3)-q_x\frac{3A_{x-1,k}h(k)}{10h(x-1)}(2k+3)-\frac{3h(k)}{10h(x)}(2k+3)\right).
\end{align*} 
For every $x\geq2$, we define 
\begin{align*}
    p_x^*&=\frac{x}{3(x+2)},\\
    r_x^*&=\frac{x(x+3)}{3(x+1)(x+2)},\\
    q_x^*&=\frac{x+3}{3(x+1)},
\end{align*}
and $p_1^* =1/3,\,r_1^*=2/3$ and $q_1^*=0$. These terms were introduced in \cite{dieuleveut} as the transition probabilities of a Markov chain, but we do not need this fact. The term in parentheses in the first line is equal to 
\begin{equation*}
    \frac{h(k)}{h(x)}\left(\frac{h(x)}{h(x+1)}-1-\frac{q_x}{p_x}\left(1-\frac{h(x)}{h(x-1)}\right)\right)=\frac{h(k)}{h(x)}\left(\frac{p^*_x}{p_x}-1-\frac{q_x}{p_x}\left(1-\frac{p_{x-1}}{p_{x-1}^*}\right)\right).
\end{equation*}

Notice that $\frac{q_x p_{x-1}}{p^*_{x-1}}=q_x^*$. Therefore, the last term is  
\begin{equation*}
    \frac{h(k)}{p_x h(x)}(p_x^*-p_x-q_x+q_x^*)=0.
\end{equation*}
This gives us a recurrence relation on $(A_{x,k})_{x\geq k}$ : 
\begin{equation}\label{recurrence 2}
    A_{x+1,k}=\frac{1}{p^*_x}((1-r^*_x)A_{x,k}-q^*_x A_{x-1,k}-1)
\end{equation}
with $A_{k,k}=0$ and $A_{k+1,k}$ prescribed by the relation \eqref{Récurrence} for $x=k$. \\
Let us first compute $A_{k+1,k}$. By the first part of Lemma \ref{ERREUR 404} (which was proved in \cite{dieuleveut}), we have 
\[H^*_{k-1}(k)=H^*_k(k)+\frac{3(2k+3)(C_k-C_{k-1})}{10}=H^*_k(k)+\frac{9(2k+3)(k+1)}{70}\]
From \eqref{Récurrence} with $x=k$, and after replacing some terms, we have 
\begin{align*}
    \frac{p_k h(k)}{h(k+1)}\left(H^*_k(k)+\frac{3A_{k+1,k}(2k+3)}{10}\right)=&(1-r_k)H^*_k(k)-q_k\left(H^*_k(k)+\frac{9(2k+3)(k+1)}{70}\right)\\&-\frac{3(2k+3)}{10}-1.
\end{align*}
Then, we can use the first part of Lemma \ref{ERREUR 404} to replace $H_k^*(k)$ by its explicit value. After a few manipulations, we obtain 
\begin{align*}
    A_{k+1,k}&=\frac{1}{(2k+3)p_k^*}\left(-\frac{(k-1)(2k+1)}{7}-(2k+3)-\frac{10}{3}-\frac{5(k-1)(k+4)}{7}+\frac{5(k+1)(k+2)}{3}\right)\\
     &=k+2.
\end{align*}

Let us now deal with \eqref{recurrence 2}. We can rewrite the recurrence as  
\[p_x^*(A_{x+1,k}-A_{x,k})=q_x^*(A_{x,k}-A_{x-1,k})-1.\]
By induction, we get  
\begin{equation}\label{sommation}
    A_{x+1,k}-A_{x,k}=A_{k+1,k}\frac{q^*_xq_{x-1}^*...q_{k+1}^*}{p^*_xp^*_{x-1}...p^*_{k+1}}-\sum_{y=k+1}^x\frac{q^*_xq_{x-1}^*...q_{y+1}^*}{p^*_xp^*_{x-1}...p^*_y}.
\end{equation}

Note that we have the following expressions
\[\frac{q^*_xq_{x-1}^*...q_{y+1}^*}{p^*_xp^*_{x-1}...p^*_{y+1}}=\frac{(x+3)(x+2)^2(x+1)}{(y+1)(y+2)^2(y+3)}\] and 
\[\frac{q^*_xq_{x-1}^*...q_{y+1}^*}{p^*_xp^*_{x-1}...p^*_y}=3\frac{(x+1)(x+2)^2(x+3)}{y(y+1)(y+2)(y+3)}.\]
Using \eqref{sommation} and the fact that $A_{k+1,k}=k+2$, we obtain
\begin{align*}
    A_{x+1,k}&=A_{k+1,k}+\sum^x_{y=k+1}A_{y+1,k}-A_{y,k}\\
    &=(k+2)+\sum_{y=k+1}^x\bigg(\frac{(y+1)(y+2)^2(y+3)}{(k+1)(k+2)(k+3)}\\
    &\quad\quad\quad\quad\quad\quad\quad-3(y+1)(y+2)^2(y+3)\sum_{i=k+1}^y\frac{1}{i(i+1)(i+2)(i+3)}\bigg).
\end{align*}

Set $\ell(x)=(x+1)(x+2)(x+3)$. The last sum is equal to 
\[\sum_{i=k+1}^y\frac{1}{i(i+1)(i+2)(i+3)}=\frac{\ell(y)-\ell(k)}{3\ell(k)\ell(y)}.\]
Hence, 
\begin{align*}
A_{x+1,k}&=(k+2)+\sum_{y=k+1}^x\left(\frac{(y+2)\ell(y)}{(k+1)(k+2)(k+3)}-\frac{(y+2)(\ell(y)-\ell(k))}{\ell(k)}\right)\\
&=(k+2)+\sum_{y=k+1}^x\left((y+2)\ell(y)\left(\frac{1}{(k+1)(k+2)(k+3)}-\frac{1}{\ell(k)}\right)+(y+2)\right)\\
&=(k+2)+\sum_{y=k+1}^x(y+2).
\end{align*}
This gives 
\[A_{x+1,k}=(k+2)+\frac{(x-k)(x+k+5)}{2}\]
which completes the proof.
\end{proof}
\begin{proof}[Proof of Proposition \ref{couplage processus 2}]
    The proof is very similar to the one of \cite[Lemma 3.3]{dieuleveut}. By \cite[Lemma 3.3]{dieuleveut}, for every measurable bounded function $f$, we have :
    \begin{equation}\label{Densité}    \E\left[f((X^{(n)}_{\infty,k},\widehat{X}^{(n)}_{\infty,k})_{0\leq k\leq\beta n^2})\1_{\{a_\infty^{(n)},b_\infty^{(n)}\geq\beta n^2\}}\right]=\E\left[\left(\sum_{k\geq1}\mathcal{H}_{X^{(\infty)}_{\beta n^2},\widehat{X}^{(\infty)}_{\beta n^2},k}(n)\right)f((X^{(\infty)}_k,\widehat{X}^{(\infty)}_k)_{0\leq k\leq\beta n^2})\right] 
    \end{equation}
    where   
    \begin{align*}
        \mathcal{H}_{x,y,k}(n)=&\frac{f(1)w(k+1)}{3f(k+1)}\frac{f(1)w(k)}{3f(k)}G_k(n-1)\\
        &\times\left(H_x^*(k+1)H_y(k)+H_x(k+1)H^*_y(k)+(2\beta n^2-5)H_x(k+1)H_y(k)\right)
    \end{align*}
    and 
    \begin{equation}\label{eq}
    G_k(n)=\left\{
    \begin{array}{ll}
 \frac{3g(k)}{35(n+1)(n+2)(n+3)}\quad&\text{ if }k\leq n,\\
 \frac{3g(n)}{35(k+1)(k+2)(k+3)} \quad&\text{ if }k>n,
    \end{array}\right.
    \end{equation}
    with $g(n)=n(n+4)(5n^2+20n+17)$.
    
    Note that the convergence \eqref{convergence Bessel} implies that for every $\beta>0$ and $\varepsilon>0$, there exists $c_{\beta}$ such that for $n$ large enough :
    \begin{equation}\label{equation1}
        \P\bigg(\sup_{0\leq i\leq\beta n^2}X^{(\infty)}_i\leq c_{\beta}n\bigg)>1-\varepsilon,
    \end{equation}
    and that $c_{\beta}\rightarrow0$ as $\beta$ tends to $0$. Therefore, even if it means reducing $\beta$, we can suppose that $c_\beta<1$. First, we work on the event $\{\sup_{0\leq i\leq\beta n^2}(X_i^\infty\vee\widehat{X}_i^\infty)\leq c_{\beta}n\}$. By decomposing according to the values of $X_{\beta n^2}^\infty$ and $\widehat{X}_{\beta n^2}^\infty$, it is enough to show that $\sum_{k\geq1}\mathcal{H}_{x,y,k}(n)$ converges to $1$ as $n\rightarrow\infty$ uniformly in $0\leq x,y\leq c_{\beta}n$.\\
    Let us first consider the terms of the sum such that $k\geq n$. By \cite[Lemmas 2.3, 2.4]{dieuleveut}, for every $y\leq c_{\beta}n$ and $k$ large enough, we have :
    \[\frac{f(1)w(k)}{3f(k)}H_y^*(k)=1-\frac{3f(1)w(k)}{140f(k)}((y+1)(y+2)-6)(2k+3)=1-\delta(k,y)\]
    where $\delta(k,y)$ is uniformly bounded by some constant $C_{\beta}$ which goes to $0$ when $\beta\rightarrow0$. Therefore, we can choose $\beta$ such that $|\delta(k,y)|\leq\varepsilon$. \\
    Similarly, uniformly in $y\leq c_{\beta}n$,
    \[\frac{f(1)w(k)}{3f(k)}H_y(k)\underset{k\rightarrow\infty}{\sim}\frac{2}{k^2}.\]
     Finally, the last term $(2\beta n^2-5)H_x(k+1)H_y(k)\underset{k\rightarrow\infty}{\sim}C\beta n^2k^2$ disappears at the limit when divided by $f(k)^2$. Consequently, using \eqref{eq}, we have 
     \begin{align}\label{sum1}
         \sum_{k\geq n}\mathcal{H}_{x,y,k}(n)&\underset{n\rightarrow\infty}{\sim}\frac{3g(n-1)}{35}\sum_{k\geq n}\frac{2\times(2-\delta(k,y)-\delta(k+1,x))}{k^5}\nonumber\\
         &\underset{n\rightarrow\infty}{\sim}\frac{3n^4}{7}\frac{2-\delta(k,y)-\delta(k+1,x)}{2n^4}=\frac{3}{7}+o(\beta), 
     \end{align} uniformly in $x,y\leq c_{\beta}n$. \\
     Now, let us look at the terms for which $x\vee y\leq k<n-1$. By \cite[Lemmas 2.3, 2.4]{dieuleveut}, for every $y\leq c_{\beta}n$, we have :
     \begin{align*}
     \frac{f(1)w(k)}{3f(k)}H^*_y(k)&=1-\frac{3C_y}{10(k+1)(k+2)}\\
     &=1-\eta(k,y)
     \end{align*} where $|\eta(k,y)|\leq C\bigg(\frac{c_{\beta}n}{k}\bigg)^2$ for some constant $C>0$.
     Consequently, we have 
     \begin{align}\label{sum2}
         \sum_{k=x\vee y+1}^n\mathcal{H}_{x,y,k}(n)&\underset{n\rightarrow\infty}{\sim}\frac{3}{35n^3}\sum_{k=x\vee y+1}^n\frac{2\times(2-\eta(k,y)-\eta(k+1,x))}{k^2}\times k^4\nonumber\\
         &\underset{n\rightarrow\infty}{\sim}\frac{4}{7}+\eta(\beta)
     \end{align} where 
     \[|\eta(\beta)|\leq\frac{C_1c_{\beta}^2n^2}{n^3}\sum_{k=x\vee y+1}^n\frac{C_2}{k^4}\times k^4=C_3c_{\beta}^2.\]
     The remaining term to evaluate is \[\sum_{k=1}^{x\vee y}\mathcal{H}_{x,y,k}(n).\]
     Note that, using the expressions of the different functions, we can find some constant $C(\beta)>0$ which tends to $0$ as $\beta\rightarrow0$ such that, for every $1\leq x,y\leq c_{\beta}n$, 
     \[|\mathcal{H}_{x,y,k}(n)|\leq C(\beta)n^{-1}.\]
     Therefore, we have :
     \begin{equation}\label{sum3}
     \left|\sum_{k=1}^{x\vee y}\mathcal{H}_{x,y,k}(n)\right|\leq C(\beta)\times c_{\beta}.
     \end{equation}
     Set 
     \[\phi(n,\beta)=\E\left[\1_{\{\sup_{0\leq i\leq\beta n^2}X_i^\infty,\widehat{X}_i^\infty\leq c_{\beta}n\}}\left|1-\sum_{k\geq1}\mathcal{H}_{X_{\beta n^2}^\infty,\widehat{X}_{\beta n^2}^\infty,k}(n)\right|\right].\]
     By \eqref{sum1}, \eqref{sum2} and \eqref{sum3}, one can see that for every $\varepsilon>0$, we can take $\beta$ small enough so that :
     \begin{equation}\label{Lim inf}
     \limsup_{n\rightarrow\infty}\phi(n,\beta)\leq\varepsilon.
     \end{equation}
     To conclude, we need to check that quantity $\sum_{k\geq1}\mathcal{H}_{x,y,k}(n)$ is uniformly bounded in $n$ for any $x,y>0$. This can be done by splitting the sum into three terms as previously, and using the explicit expressions for the different quantities involved (we leave this verification to the reader). Nonetheless, we emphasize on the fact that this step relies on the corrected formula of Lemma \ref{ERREUR 404}. Therefore, by \eqref{Densité}, for every non-negative measurable function $F$ bounded by $1$, we have 
     \begin{multline*}         \E\left[\left|f((X^{(n)}_{\infty,k},\widehat{X}^{(n)}_{\infty,k})_{0\leq k\leq\beta n^2})-f((X_k^\infty,\widehat{X}_k^\infty)_{0\leq k\leq\beta n^2})\right|\right]\leq\P(a_\infty^{(n)}<\beta n^2)+\P(b_\infty^{(n)}<\beta n^2)\\+\phi(n,\beta)
     +\P\left(\sup_{0\leq i\leq\beta n^2}X_i^\infty\leq c_{\beta}n\right)+\P\left(\sup_{0\leq i\leq\beta n^2}\widehat{X}_i^\infty\leq c_{\beta}n\right).
     \end{multline*} 
     Finally, by \eqref{equation1}, \eqref{Lim inf} and Proposition \ref{control 1}, we can choose $\beta$ small enough so that each term on the right-hand side is smaller than $\varepsilon/5$, which concludes the proof. 
\end{proof}
\printbibliography

@article{brownianplane,
  title={The brownian plane},
  author={Curien, Nicolas and Le Gall, Jean-Fran{\c{c}}ois},
  journal={Journal of Theoretical Probability},
  volume={27},
  number={4},
  pages={1249--1291},
  year={2014},
  publisher={Springer}
}

@article{dieuleveut,
author = {Daphn{\'e} Dieuleveut},
title = {{The UIPQ seen from a point at infinity along its geodesic ray}},
volume = {21},
journal = {Electronic Journal of Probability},
year = {2016},
doi = {10.1214/16-EJP4730},
}

@article{uniqueness,
author = {Jean-Fran{\c{c}}ois Le Gall},
title = {{Uniqueness and universality of the Brownian map}},
volume = {41},
journal = {The Annals of Probability},
number = {4},
publisher = {Institute of Mathematical Statistics},
pages = {2880 -- 2960},
year = {2013},
doi = {10.1214/12-AOP792},
URL = {https://doi.org/10.1214/12-AOP792}
}

@article{convergence,
author = {Gr{\'e}gory Miermont},
title = {{The Brownian map is the scaling limit of uniform random plane quadrangulations}},
volume = {210},
journal = {Acta Mathematica},
number = {2},
publisher = {Institut Mittag-Leffler},
pages = {319 -- 401},
year = {2013},
doi = {10.1007/s11511-013-0096-8},
URL = {https://doi.org/10.1007/s11511-013-0096-8}
}

@article{geodesic1,
author = {Jean-Fran{\c{c}}ois Le Gall},
title = {{Geodesics in large planar maps and in the Brownian map}},
volume = {205},
journal = {Acta Mathematica},
number = {2},
publisher = {Institut Mittag-Leffler},
pages = {287 -- 360},
year = {2010},
doi = {10.1007/s11511-010-0056-5},
URL = {https://doi.org/10.1007/s11511-010-0056-5}
}

@article{sphericity,
author = {Miermont, Grégory},
journal = {Electronic Communications in Probability [electronic only]},
language = {eng},
pages = {248-257},
publisher = {University of Washington},
title = {On the sphericity of scaling limits of random planar quadrangulations.},
url = {http://eudml.org/doc/231565},
volume = {13},
year = {2008},
}

@book{serpent,
author = {Jean-Fran{\c{c}}ois Le Gall},
booktitle = {Spatial branching processes, random snakes and partial differential equations},
isbn = {3-7643-6126-3},
publisher = {Birkhäuser Verlag},
series = {Lectures in mathematics ETH Zürich},
title = {Spatial branching processes random snakes and partial differential equations },
year = {1999},
}

@article{peeling,
  title={Peeling random planar maps},
  author={Curien, Nicolas},
  journal={Saint-Flour lecture notes},
  year={2019}
}

@PHDTHESIS{schaeffer,
url = "http://www.theses.fr/1998BOR10636",
title = "Conjugaison d'arbres et cartes combinatoires aléatoires",
author = "Schaeffer, Gilles",
year = "1998",
pages = "229 p",
note = "Thèse de doctorat dirigée par Cori, Robert Informatique Bordeaux 1 1998",
note = "1998BOR10636",
}

@article{cori_vauquelin_1981, 
title={Planar Maps are Well Labeled Trees}, 
volume={33}, number={5}, journal={Canadian Journal of Mathematics}, 
publisher={Cambridge University Press}, 
author={Cori, Robert and Vauquelin, Bernard}, 
year={1981}, 
pages={1023–1042}
}

@misc{Bigeodesicbrownianplane,
      title={The bigeodesic Brownian plane}, 
      author={Mathieu Mourichoux},
      year={2024},
      eprint={2410.00426},
      archivePrefix={arXiv},
      primaryClass={math.PR},
      url={https://arxiv.org/abs/2410.00426}, 
}

@article{Neveu1986,
author = {Neveu, J.},
journal = {Annales de l'I.H.P. Probabilités et statistiques},
number = {2},
pages = {199-207},
publisher = {Gauthier-Villars},
title = {Arbres et processus de Galton-Watson},
volume = {22},
year = {1986},
}

@article{ChassaingDurhuus,
author = {Philippe Chassaing and Bergfinnur Durhuus},
title = {{Local limit of labeled trees and expected volume growth in a random quadrangulation}},
volume = {34},
journal = {The Annals of Probability},
number = {3},
publisher = {Institute of Mathematical Statistics},
year = {2006},
doi = {10.1214/009117905000000774},
}

@article{Lamperti,
author = {John Lamperti},
title = {{A new class of probability limit theorems}},
volume = {67},
journal = {Bulletin of the American Mathematical Society},
number = {3},
publisher = {American Mathematical Society},
pages = {267 -- 269},
year = {1961},
}

@article{Viewfrominfinity,
  title = {{A view from infinity of the uniform infinite planar quadrangulation}},
  author = {Curien, Nicolas and M{\'e}nard, Laurent and Miermont, Grégory},
  journal = {{ALEA : Latin American Journal of Probability and Mathematical Statistics}},
  volume = {10},
  number = {1},
  pages = {45-88},
  year = {2013},
}

@misc{Krikun,
      title={Local structure of random quadrangulations}, 
      author={Maxim Krikun},
      year={2006},
      eprint={math/0512304},
      archivePrefix={arXiv},
}

@article{lawlerbessel,
  title={Notes on the Bessel process},
  author={Lawler, Gregory F},
  journal={Lecture notes. Available on the webpage of the author},
  pages={523--531},
  year={2018}
}

@article{MenardLaw,
author = {Laurent M{\'e}nard},
title = {{The two uniform infinite quadrangulations of the plane have the same law}},
volume = {46},
journal = {Annales de l'Institut Henri Poincaré, Probabilités et Statistiques},
number = {1},
publisher = {Institut Henri Poincaré},
pages = {190 -- 208},
year = {2010},
doi = {10.1214/09-AIHP313},
URL = {https://doi.org/10.1214/09-AIHP313}
}

@article{LegallMenard,
author = {Jean-Fran{\c{c}}ois Le Gall and Laurent M{\'e}nard},
title = {{Scaling limits for the uniform infinite quadrangulation}},
volume = {54},
journal = {Illinois Journal of Mathematics},
number = {3},
publisher = {Duke University Press},
pages = {1163 -- 1203},
year = {2010},
doi = {10.1215/ijm/1336049989},
URL = {https://doi.org/10.1215/ijm/1336049989}
}

@article{SeparatingCycles,
author = {Jean-Fran{\c{c}}ois Le Gall and Thomas Leh{\'e}ricy},
title = {{Separating cycles and isoperimetric inequalities in the uniform infinite planar quadrangulation}},
volume = {47},
journal = {The Annals of Probability},
number = {3},
publisher = {Institute of Mathematical Statistics},
pages = {1498 -- 1540},
year = {2019},
doi = {10.1214/18-AOP1289},
URL = {https://doi.org/10.1214/18-AOP1289}
}

@article {AbrahamConvergence,
    AUTHOR = {Abraham, C\'{e}line},
     TITLE = {Rescaled bipartite planar maps converge to the {B}rownian map},
   JOURNAL = {Ann. Inst. Henri Poincar\'{e} Probab. Stat.},
  FJOURNAL = {Annales de l'Institut Henri Poincar\'{e} Probabilit\'{e}s et
              Statistiques},
    VOLUME = {52},
      YEAR = {2016},
    NUMBER = {2},
     PAGES = {575--595},
}

@article {ConvergenceSimple,
    AUTHOR = {Addario-Berry, Louigi and Albenque, Marie},
     TITLE = {The scaling limit of random simple triangulations and random
              simple quadrangulations},
   JOURNAL = {Ann. Probab.},
  FJOURNAL = {The Annals of Probability},
    VOLUME = {45},
      YEAR = {2017},
    NUMBER = {5},
     PAGES = {2767--2825},
}

@article {ConvergenceBJM,
    AUTHOR = {Bettinelli, J\'{e}r\'{e}mie and Jacob, Emmanuel and Miermont,
              Gr\'{e}gory},
     TITLE = {The scaling limit of uniform random plane maps, {\it via} the
              {A}mbj\o rn-{B}udd bijection},
   JOURNAL = {Electron. J. Probab.},
  FJOURNAL = {Electronic Journal of Probability},
    VOLUME = {19},
      YEAR = {2014},
     PAGES = {no. 74, 16},
}

@article{UIHPQ,
  title={Uniform infinite planar quadrangulations with a boundary},
  author={Nicolas Curien and Gr{\'e}gory Miermont},
  journal={Random Structures \& Algorithms},
  year={2012},
  volume={47}
}

@article{GeometryUIHPQ,
author = {Caraceni, Alessandra and Curien, Nicolas},
year = {2015},
month = {08},
title = {Geometry of the Uniform Infinite Half-Planar Quadrangulation},
volume = {52},
journal = {Random Structures \& Algorithms}
}

@article{basu2021environment,
author = {Riddhipratim Basu and Manan Bhatia and Shirshendu Ganguly},
title = {{Environment seen from infinite geodesics in Liouville Quantum Gravity}},
volume = {52},
journal = {The Annals of Probability},
number = {4},
publisher = {Institute of Mathematical Statistics},
pages = {1399 -- 1486},
year = {2024},
}

@article{LeGall2005,
author = {Le Gall, Jean-François},
journal = {Probability Surveys [electronic only]},
language = {eng},
pages = {245-311},
publisher = {Sponsored by Institute of Mathematical Statistics and by the Bernoulli Society},
title = {Random trees and applications.},
url = {http://eudml.org/doc/229216},
volume = {2},
year = {2005},
}

@article{ConvergenceSerpent,
author = {Gr{\'e}gory Miermont},
title = {{Invariance principles for spatial multitype Galton–Watson trees}},
volume = {44},
journal = {Annales de l'Institut Henri Poincaré, Probabilités et Statistiques},
number = {6},
publisher = {Institut Henri Poincaré},
pages = {1128 -- 1161},
year = {2008},
doi = {10.1214/07-AIHP157},
}

@book{Duquesne,
     author = {Duquesne, Thomas and Le Gall, Jean-Fran\c{c}ois},
     title = {Random trees, {L\'evy} processes and spatial branching processes},
     series = {Ast\'erisque},
     publisher = {Soci\'et\'e math\'ematique de France},
     number = {281},
     year = {2002},
     mrnumber = {1954248},
     zbl = {1037.60074},
     language = {en},
     url = {https://www.numdam.org/item/AST_2002__281__R1_0/}
}

@misc{bettinelli2025compactbrowniansurfacesii,
      title={Compact Brownian surfaces II. Orientable surfaces}, 
      author={Jérémie Bettinelli and Grégory Miermont},
      year={2025},
      eprint={2212.12511},
      archivePrefix={arXiv},
      primaryClass={math.PR},
      url={https://arxiv.org/abs/2212.12511}, 
}
\end{document}